\documentclass{article}

\usepackage[colorlinks,linkcolor=black,citecolor=black]{hyperref}
\usepackage{bm}
\usepackage{bbm}
\usepackage{mathrsfs}
\usepackage{amsmath}
\usepackage{amsthm}
\usepackage{amssymb}
\usepackage{mathdots}
\usepackage{multirow}
\usepackage{geometry}
\usepackage{mathtools}
\usepackage{fancyhdr}
\usepackage{authblk}
\usepackage[pagewise]{lineno}
\usepackage{indentfirst}
\usepackage[backend=biber,style=numeric,giveninits=true,sorting=nyt,maxbibnames=99]{biblatex}
\allowdisplaybreaks
\numberwithin{equation}{section}
\theoremstyle{plain}
\newtheorem{theorem}{Theorem}[section]
\newtheorem{proposition}[theorem]{Proposition}
\newtheorem{lemma}[theorem]{Lemma}
\newtheorem{corollary}[theorem]{Corollary}
\newtheorem{claim}[theorem]{Claim}

\theoremstyle{definition}

\newtheorem{remark}[theorem]{Remark}
\newcommand{\re}{\operatorname{Re}}

\newcommand{\op}{\operatorname{Op}}
\newcommand{\sgn}{\operatorname{sgn}}

\DeclareFieldFormat[article]{title}{#1}
\DeclareFieldFormat[article]{pages}{#1}
\DeclareFieldFormat[article]{number}{no\adddot\space#1}
\DeclareFieldFormat[article]{eid}{Art\adddot\space#1}

\renewbibmacro*{in:}{%
	\ifentrytype{article}
	{}
	{\printtext{\bibstring{in}\intitlepunct}}}
\renewbibmacro*{journal+issuetitle}{%
	\usebibmacro{journal}%
	\setunit*{\addspace}%
	\printfield{volume}%
	\setunit*{\addspace}%
	\usebibmacro{issue+date}%
	\setunit{\addcomma\space}%
	\printfield{number}%
	\setunit{\addcomma\space}%
	\printfield{eid}%
	\newunit}

\title{Linear Stability and Inviscid Damping of Monotone Shear Flows for 2D Compressible Euler Equations}
\author{
	Zhile Li$^{1,*}$,
	Junyan Zhang$^{1,\dagger}$,
	and Lifeng Zhao$^{1,\ddagger}$
}

\date{}

\begin{document}
	\maketitle
	
	\begingroup
	\renewcommand{\thefootnote}{}
	\footnotetext{
		$^{1}$School of Mathematical Sciences,
		University of Science and Technology of China,
		Hefei 230026, Anhui, China.
		
		\quad $^{*}$Email: \texttt{lizhile@mail.ustc.edu.cn}
		\qquad
		$^{\dagger}$Email: \texttt{yx3x@ustc.edu.cn}
		\qquad
		$^{\ddagger}$Email: \texttt{zhaolf@ustc.edu.cn}
	}
	\endgroup

	\begin{abstract}
		
		We study 2D compressible Euler equations linearized around monotone shear flows \((U(y),0)\) on \(\mathbb{T} \times \mathbb{R}\). The shear rate \(U'\) is strictly positive, not necessarily close to any constant, and varies sufficiently slowly. For every fixed Mach number \(M > 0\), we prove that the density and the irrotational velocity obey algebraic growth bounds, whereas the solenoidal velocity undergoes componentwise inviscid damping. Although a non-uniform shear couples the transverse Fourier frequencies and precludes the full Fourier reduction available for Couette flow, we are still able to recover the Couette rates without loss. The proof hinges on two new ingredients: a time-dependent pseudodifferential energy that restores a coercive structure for the variable-coefficient shear dynamics, and terminal-time-dependent higher- and lower-order weighted energies that capture the long-time effects of shear mixing.
		
	\end{abstract}

	\section{Introduction}
	
	The compressible Euler equations govern the motion of inviscid compressible fluids and constitute a fundamental model in fluid dynamics. In this paper, we consider the isentropic compressible Euler equations on \(\mathbb{T} \times \mathbb{R}\),
	\begin{equation}
		\label{Euler}
		\begin{cases}
			\partial_t \widetilde{\rho} + \nabla \cdot (\widetilde{\rho} u) = 0, \\
			\partial_t (\widetilde{\rho} u) + \nabla \cdot (\widetilde{\rho} u \otimes u) + \frac{1}{M^2} \nabla p(\widetilde{\rho}) = 0.
		\end{cases}
	\end{equation}
	Here \(\widetilde{\rho}\) and \(u\) denote the density and velocity, respectively, \(M\) is the Mach number, and \(p = p(\widetilde{\rho})\) is the isentropic pressure law.
	
	A natural class of steady solutions of \eqref{Euler} is given by shear flows of the form
	\[u_E = (U(y), 0), \qquad \widetilde{\rho}_E = 1.\]
	Given their ubiquity in fluid dynamics, understanding the stability and long-time behavior of shear flows has been a central topic in the study of hydrodynamic stability \cite{DrazinReid2004}. To analyze the stability of this equilibrium, we consider perturbations around the shear flow. More precisely, writing
	\[\widetilde{\rho} = \widetilde{\rho}_E + \rho, \qquad u = u_E + v,\]
	substituting this decomposition into \eqref{Euler}, and neglecting terms that are nonlinear in the perturbation, we obtain the following linearized system:
	\begin{equation}
		\label{perturbation}
		\begin{cases}
			\partial_t \rho + U(y) \partial_x \rho + \nabla \cdot v = 0, \\
			\partial_t v + U(y) \partial_x v + (U'(y) v_2, 0) + \frac{1}{M^2} \nabla \rho = 0,
		\end{cases}
	\end{equation}
	where we set \(p'(1) = 1\). The stability problem is therefore reduced to understanding the long-time behavior of solutions to \eqref{perturbation}.
	
	The study of shear-flow stability has a much longer history in the incompressible setting, where the classical theory began with questions of spectral stability. Rayleigh \cite{rayleigh1879} studied the normal-mode problem for general shear flows and established the celebrated inflection-point criterion, which gives a necessary condition for spectral instability. For the plane Couette flow, Kelvin \cite{kelvin1887} and Orr \cite{orr1907} investigated the evolution of perturbations through the initial-value problem. Their analyses already exhibited the phase-mixing mechanism and transient growth underlying what is now known as inviscid damping. Case \cite{case1960} later studied the inviscid initial-value problem through a continuous-spectrum analysis and formally derived the associated velocity decay.
	
	Within the incompressible setting, extending inviscid damping beyond the explicitly solvable Couette flow to general shear profiles is considerably more difficult. Bouchet--Morita \cite{bouchet2009} gave heuristic arguments suggesting that linear stability and inviscid damping should persist for general monotone shear flows. Stepin \cite{Stepin1995} studied the problem from a spectral point of view and obtained asymptotic stability results under suitable assumptions. Zillinger \cite{Zillinger2017} developed a rigorous linear theory for a large class of strictly monotone shear flows, establishing stability, scattering, and inviscid damping with optimal decay rates in both infinite and finite periodic channels. Wei--Zhang--Zhao \cite{wei2018} established velocity decay estimates and scattering for a class of monotone shear flows in a finite channel. The same authors later proved linear damping and vorticity depletion for possibly non-monotone shear flows with stationary streamlines under a suitable spectral assumption \cite{wei2019}. Jia subsequently gave an alternative proof of the sharp decay rates and derived precise asymptotics \cite{jia2020}, and also established linear damping in Gevrey spaces \cite{jia2020a}. A different approach based on the conjugate-operator method was developed by Grenier--Nguyen--Rousset--Soffer \cite{grenier2020}. On the other hand, Deng--Zillinger \cite{deng2020} showed that monotonicity alone does not lead to a uniform linear damping theory, identifying a resonance-chain obstruction for a class of smooth bi-Lipschitzian shear profiles. This indicates that, beyond the exactly solvable Couette profile, additional structural information on the shear is relevant to the long-time dynamics.  Related spectral and damping questions have also been studied in the free-boundary setting: See Liu--Zeng \cite{LiuZeng2025WWST, LiuZeng2025spectra} for the analysis of gravity(-capillary) water waves linearized around monotone shear flows.
	
	At the nonlinear level, still in the two-dimensional incompressible Euler setting, Bedrossian--Masmoudi \cite{bedrossian2015} rigorously established nonlinear inviscid damping near the Couette flow in $\mathbb{T}\times\mathbb{R}$.  For the case of a finite channel, Lin--Zeng \cite{LinZeng2011} identified a regularity dichotomy near Couette flow, and the nonlinear asymptotic stability results for classes of monotone shear flows were later obtained by Ionescu--Jia \cite{ionescu2023} and Masmoudi--Zhao \cite{masmoudi2024}.
	
	Compared with the incompressible setting, compressibility introduces additional couplings among the density, irrotational velocity, and solenoidal velocity, leading to more complicated dynamics. Despite these difficulties, the stability of compressible shear flows has been investigated extensively through normal-mode and spectral methods \cite{Blumen1975, blumen1970, Drazin1977, leeslin1946, Subbiah1990}. These studies yielded stability criteria and provided information about unstable modes and spectral properties. However, such modal analyses do not by themselves describe the detailed time evolution of perturbations or provide quantitative long-time estimates for the corresponding initial-value problem. 
	
	A significant advance in this direction was achieved by Antonelli--Dolce--Marcati \cite{antonelli2021}, who studied the linearized compressible Euler equations around the homogeneous Couette flow. By exploiting the explicit structure of the Couette profile, they introduced sheared coordinates and reduced the problem to a family of non-autonomous equations for individual Fourier modes. Combined with carefully designed time-dependent energy estimates, this approach yielded precise growth and decay bounds for the density perturbation, the irrotational component, and the solenoidal component of the velocity perturbation. The corresponding Couette-flow problem for the two-dimensional non-isentropic compressible Euler equations was studied by Zhai \cite{Zhai2023Couette}. 
	
	However, the homogeneous Couette flow enjoys a particularly favorable structure that is absent for general shear profiles. In particular, the linear dependence of the velocity on the transverse variable yields explicit formulas in sheared coordinates and preserves a mode-by-mode Fourier representation of the linearized dynamics. Once \(U'\) is nonconstant, even if its variation is arbitrarily small, the transverse Fourier frequencies are coupled and the equations in sheared coordinates become genuinely variable-coefficient systems.
	
	This raises a fundamental question in compressible hydrodynamic stability: to what extent do the long-time mixing and inviscid damping mechanisms observed around homogeneous Couette flow persist for non-uniform shear profiles? From the mathematical viewpoint, the main difficulty is not merely the size of the deviation from Couette flow, but rather the loss of the exact commutation properties underlying the Couette analysis. The resulting coupling between acoustic waves and vorticity is governed by variable-coefficient operators. For such profiles, defining the basic energy through Fourier multipliers is generally not viable, since differentiating such an energy can produce high-order commutator terms that cannot be absorbed by the available dissipation.
	
	Motivated by this question, we consider monotone shear flows with slowly varying shear rate: the spatial derivatives of \(U'\) are required to be small, but \(U'\) need not be close to a constant. For example, consider a family of profiles whose shear rates are given by
	\[U'(y) = \lambda \left( 1 + \frac{1}{2} \tanh \left( \frac{y}{\ell} \right) \right), \qquad \lambda > 0, \ell \ge 1.\]
	For every fixed integer \(n \ge 0\), we have
	\[\frac{\lambda}{2} \le U'(y) \le \frac{3\lambda}{2}, \qquad \left\| U'' \right \|_{W^{n,\infty}} \lesssim \frac{\lambda}{\ell}, \qquad \inf_{c \in \mathbb{R}} \left\| U' - c \right\|_{L^\infty} = \frac{\lambda}{2}.\]
	Thus, as \(\ell\) becomes large, the spatial derivatives of the shear rate become small, while \(U'\) does not become close in \(L^\infty\) to any constant. In this sense, slow variation refers to the spatial scale of the variation rather than to its amplitude.
	
	We also mention an earlier preprint of Antonelli--Dolce--Marcati \cite{antonelli2020near}, which studies monotone shear flows perturbatively close to Couette and states polynomial growth and inviscid damping estimates with a small loss in the temporal exponents. The class of shear profiles considered here is not subject to this perturbative closeness assumption. Our analysis is based on a time-dependent pseudodifferential energy adapted to the variable-coefficient shear dynamics. In this framework, the relevant commutator errors are of lower order and are absorbed by the dissipation generated by the energy weights. Combined with terminal-time-dependent higher- and lower-order weighted energies, this yields the Couette-scale growth and componentwise decay bounds without any loss in the temporal exponents.
	
	\subsection{Main result}
	
	To formulate our main result, we introduce the Helmholtz decomposition of the velocity perturbation. Let
	\[\alpha = \nabla \cdot v, \qquad \omega = \nabla^\perp \cdot v,\]
	where \(\nabla^\perp = (- \partial_y, \partial_x)\). Then the velocity perturbation admits the decomposition
	\begin{equation}
		\label{Helmholtz}
		v = \nabla \Delta^{-1} \alpha + \nabla^\perp \Delta^{-1} \omega \eqqcolon v_{\mathrm{irr}} + v_{\mathrm{sol}}.
	\end{equation}
	Here \(v_{\mathrm{irr}}\) and \(v_{\mathrm{sol}}\) denote the irrotational and solenoidal components of the velocity perturbation, respectively. In terms of \(\rho,\alpha,\omega\), the linearized system \eqref{perturbation} can be rewritten as
	\begin{equation}
		\label{rho_alpha_omega}
		\begin{cases}
			\partial_t \rho + U \partial_x \rho + \alpha = 0,\\
			\partial_t \alpha + U \partial_x \alpha + 2 U' \partial_x (\partial_y \Delta^{-1} \alpha + \partial_x \Delta^{-1} \omega) + \frac{1}{M^2} \Delta \rho = 0,\\
			\partial_t \omega + U \partial_x \omega - U' \alpha = U'' (\partial_y \Delta^{-1} \alpha + \partial_x \Delta^{-1} \omega).
		\end{cases}
	\end{equation}
	
	Our main result is stated as follows.
	
	\begin{theorem}
		\label{thm main}
		Fix \(M > 0\). Assume that the shear profile \(U\) satisfies
		\[0 < U'_{\inf} \le U'(y) \le U'_{\sup} < \infty, \qquad \|U''\|_{W^{n,\infty}} \le \epsilon,\]
		where \(n\) is a sufficiently large fixed integer and \(\epsilon = \epsilon(U'_{\inf}, U'_{\sup}, M)\) is a sufficiently small threshold depending only on \(U'_{\inf}, U'_{\sup}\) and \(M\). Let the initial data satisfy \(\rho_{in} \in L_x^2 H_y^5\) and \(\alpha_{in}, \omega_{in} \in H_x^{-1} H_y^4\), and assume they all have zero \(x\)-average, namely,
		\[\int_{\mathbb{T}} \rho_{in} (x,y) dx = \int_{\mathbb{T}} \alpha_{in} (x,y) dx = \int_{\mathbb{T}} \omega_{in} (x,y) dx = 0,\qquad \forall\, y\in\mathbb{R}.\]
		We measure the size of the initial data by
		\begin{equation}
			\label{N_in}
			\mathcal{N}_{in}^2 = \frac{1}{M^2} \| \rho_{in} \|_{L_x^2 H_y^5}^2 + \| \alpha_{in} \|_{H_x^{-1} H_y^4}^2 + \| \omega_{in} \|_{H_x^{-1} H_y^4}^2.
		\end{equation}
		Then the corresponding solution of \eqref{rho_alpha_omega} satisfies the following estimates for every \(t \ge 0\):
		\begin{align}
			\label{thm:acoustic growth} \frac{1}{M} \| \rho \|_{L_x^2 L_y^2} (t) + \| v_{\mathrm{irr}} \|_{L_x^2 L_y^2} (t) &\lesssim \mathcal{N}_{in} \langle t \rangle^{1/2},\\
			\label{thm:horizontal damping} \| (v_{\mathrm{sol}})_1 \|_{L_x^2 L_y^2} (t) &\lesssim \mathcal{N}_{in} \langle t \rangle^{-1/2},\\
			\label{thm:vertical damping} \| (v_{\mathrm{sol}})_2 \|_{L_x^2 L_y^2} (t) &\lesssim \mathcal{N}_{in} \langle t \rangle^{-3/2},
		\end{align}
		where \((v_{\mathrm{sol}})_1, (v_{\mathrm{sol}})_2\) denote the two components of \(v_{\mathrm{sol}}\) and the implicit constants in \(\lesssim\) are allowed to depend on \(U'_{\inf}\), \(U'_{\sup}\), and \(M\).
	\end{theorem}
	
	We make the following remark concerning the above result.
	
	\begin{remark}
		When the shear profile is the homogeneous Couette flow, namely, \(U(y) = y\), the upper bounds obtained in this paper agree with those established in \cite{antonelli2021}. Moreover, Antonelli--Dolce--Marcati \cite{antonelli2021} proved that the growth rate \(\langle t \rangle^{1/2}\) is attained for a generic set of initial data, showing that this exponent is sharp in the homogeneous Couette case. Consequently, the exponent in our growth bound is optimal within the class of shear flows considered here, while the two componentwise decay bounds recover the known inviscid damping upper bounds for the homogeneous Couette flow.
	\end{remark}
	
	\subsection{Strategy of the proof}
	\label{subsec strategy}
	
	The proof has two main ingredients: a basic pseudodifferential energy for the variable-coefficient acoustic-vorticity system, and terminal-time-dependent higher- and lower-order weighted energies that yield the growth and damping estimates.
	
	\subsubsection{Reformulation in sheared coordinates}
	
	The first difficulty is to construct a coercive, non-increasing energy for the linearized system in the presence of a non-uniform shear profile. The key idea is to require the time-dependent pseudodifferential weights \(T\) and \(S\) to commute exactly with the sheared derivative \(D_t\). This exact commutation preserves the highest-order cancellations of the acoustic wave system and is the main structural ingredient in our basic energy construction.
	
	We first reformulate the system mode by mode in the periodic direction. After taking the Fourier transform in \(x\), we fix a nonzero mode \(k \in \mathbb{Z} \setminus \{0\}\) and denote the corresponding Fourier modes of \(\rho\), \(\alpha\), and \(\omega\) by \(\rho_k\), \(\alpha_k\), and \(\omega_k\), respectively. Using the unitary operator \(V=e^{ikU(y)t}\), we introduce the sheared variables
	\[R = V \rho_k, \qquad A = V \alpha_k, \qquad W = V\omega_k.\]
	The transverse derivative and Laplacian in sheared coordinates are
	\[D_t \coloneqq V \partial_y V^{-1} = \partial_y - ikU'(y)t, \qquad \Delta_t \coloneqq V (-k^2 + \partial_y^2) V^{-1} = -k^2 + D_t^2.\]
	We further introduce the modified vorticity variable
	\[\Theta = W + U' R + U'' D_t \Delta_t^{-1} R \eqqcolon W + LR.\]
	The correction \(LR\) is chosen so that, upon differentiating \(\Theta\), all terms involving \(A\) cancel. In terms of \(R\), \(A\), and \(\Theta\), system \eqref{rho_alpha_omega} becomes
	\[
	\begin{cases}
		\partial_t R + A = 0, \\
		\partial_t A + 2ikU'(D_t \Delta_t^{-1} A + ik \Delta_t^{-1} (\Theta - U'R - U'' D_t \Delta_t^{-1} R)) + \frac{1}{M^2} \Delta_t R = 0,\\
		\partial_t \Theta = ikU'' \Delta_t^{-1} \Theta + 2ik^3U'' \Delta_t^{-1} U' \Delta_t^{-1} R.
	\end{cases}
	\]
	
	We shall define an energy adapted to this system. To ensure the exact commutation property described above, we construct the energy weights by conjugating suitable Fourier multipliers with \(V\). Indeed, for any Fourier multiplier \(\op(m)\), with \(m = m(\eta)\), the operator \(\mathcal{M} \coloneqq V \op(m) V^{-1}\) satisfies
	\[[\mathcal{M}, D_t] = 0, \qquad [\mathcal{M}, \Delta_t] = 0.\]
	In contrast, a fixed Fourier multiplier \(\op(m)\) satisfies
	\[[\op(m), D_t] = -ikt[\op(m), U'].\]
	The resulting commutator contains the growing factor \(t\), together with derivatives of \(U'\), and therefore cannot in general be controlled uniformly for large times. This is the basic reason for using \(V\)-conjugated weights in the energy.
	
	\begin{remark}[The limitation of shear straightening]
		A classical alternative in problems of this type is to straighten the shear by introducing \(z = U(y)\), which removes the explicit growing factor \(t\) appearing in the preceding commutator. More precisely, in the new coordinate, the sheared derivative becomes
		\[\widetilde{D}_t \coloneqq U'(U^{-1}(z)) (\partial_z - ikt).\]
		Although Fourier multipliers in the \(z\)-coordinate commute with \(\partial_z - ikt\) and therefore avoid the preceding commutator with explicit \(t\)-growth, they do not commute with the variable coefficient \(U'(U^{-1}(z))\). In an energy adapted to the natural wave structure, the highest-order acoustic terms cancel exactly, leaving these variable-coefficient commutators as leading uncanceled contributions that cannot in general be absorbed by the multiplier-generated dissipation. Thus, within the energy framework developed below, a basic energy defined solely through Fourier multipliers in the \(z\)-coordinate is generally not viable.
	\end{remark}
	
	\subsubsection{A non-increasing pseudodifferential energy}
	Our goal is to construct a non-increasing energy adapted to this system. We begin with the coupled equations for \(R\) and \(A\). Disregarding for the moment the lower-order terms carrying the factor \(2ikU'\) in the equation for \(A\), they have the structure of a wave equation, whose natural energy is
	\[\frac{1}{2M^2} \| kR \|_{L^2}^2 + \frac{1}{2M^2} \| D_tR \|_{L^2}^2 + \frac{1}{2} \|A\|_{L^2}^2.\]
	When differentiating this energy, the time derivative falling on \(D_t\) and the term \(2ikU' D_t \Delta_t^{-1} A\) omitted above produce, respectively, the indefinite contributions \(- \frac{1}{M^2} \re \langle ikU' R, D_t R \rangle\) and \(-2 \re \langle ikU'D_t \Delta_t^{-1} A, A \rangle\). Both occur at the operator scale \(U'|k|(-\Delta_t)^{-1/2}\), while all the other omitted terms are of lower order. 
	
	\paragraph{The principal weight \(T\).} To absorb these two leading terms, we construct a weight \(T\) and an auxiliary operator \(P\) so that, at the principal-symbol level,
	\[- (\partial_t T) T^{-1} \approx P^*P \approx \frac{1}{\gamma} U'(y) |k| (-\Delta_t)^{-1/2},\]
	where \(\gamma > 0\) is a constant to be chosen below. Differentiating the \(T\)-weighted wave energy then produces negative \(P\)-scale contributions that can be used to control these two terms.
	
	Formally, \(T\) can be viewed as a function of the sheared derivative, 
	\[T = \exp \left( \frac{1}{\gamma} \sinh^{-1} \left( \frac{D_t}{ik} \right)\right).\]
	Since \(\exp \left( \frac{1}{\gamma} \sinh^{-1} \xi \right) \approx \langle \xi \rangle^{(\sgn \xi)/\gamma}\), the parameter \(\gamma\) determines the frequency order of \(T\), and hence the resulting growth and decay exponents. To obtain the optimal rates permitted by this construction, we therefore seek the largest value of \(\gamma\) compatible with the energy estimate. At this stage, without the cross term, the \(A\)-component would restrict us to \(\gamma \le \frac{1}{2}\). 
	
	To enlarge the admissible range of \(\gamma\), we consider the cross term \(\lambda \re \langle PTA,PTR \rangle\), where \(\lambda > 0\) is to be chosen. At the principal \(P\)-scale, when the time derivative falls on \(R\), the identity \(\partial_t R = - A\) produces the negative contribution \(- \lambda \| PTA \|_{L^2}^2\); when it falls on \(A\), the acoustic term produces positive contributions to the \(R\)-components. The critical coefficients of \(\frac{1}{M^2} \| PTD_tR \|_{L^2}^2\) and \(\| PTA \|_{L^2}^2\) are then \(-1 + \lambda + \gamma\) and \(-1 - \lambda + 2\gamma\), respectively. Balancing these two coefficients gives \(\lambda = \frac{1}{3}\) and \(\gamma = \frac{2}{3}\). Thus, the correcting cross term enlarges the admissible value of \(\gamma\) from \(\frac{1}{2}\) to \(\frac{2}{3}\).
	
	\paragraph{The lower-order weight \(S\).} 
	
	Beyond the principal contributions used in the preceding balance, differentiating the cross term produces additional terms when the time derivative falls on \(P\) or \(T\), as well as terms arising from the lower-order part of the equation for \(A\). The highest-order terms that remain to be controlled occur at the scale \(U' |k|^2 (-\Delta_t)^{-1}\). To absorb these contributions, we introduce a second weight \(S\), whose symbolic representation and leading logarithmic time derivative are given by
	\[S = \exp \left( \frac{1}{\beta} \tan^{-1} \left( \frac{D_t}{ik} \right) \right), \qquad (\partial_t S) S^{-1} \approx - \frac{1}{\beta} U'(y) |k|^2 (- \Delta_t)^{-1}.\]
	For a sufficiently small \(\beta > 0\), the coercive negative contributions generated by differentiating \(S\) absorb the contributions described above. Since \(\tan^{-1}\) is bounded, \(S\) is uniformly bounded from above and below, so the choice of \(\beta\) does not affect the frequency orders of the energy or the resulting temporal exponents.
	
	\paragraph{The coercivity for full range of the Mach number.}  Finally, after applying the weights \(T\) and \(S\), the component \(\Theta\) is controlled by the resulting energy structure. However, the cross term needed to reach the critical value $\gamma=\frac{2}{3}$ introduces an additional coercivity issue, whose principal part reads
	\[
	\frac{1}{2}
	\operatorname{Re}\left\langle TSA,\,TS|k|(-\Delta_t)^{-1/2}U'R\right\rangle .
	\]
	If this term were controlled solely by the natural acoustic energy, then we have
	\[
	\left|\frac{1}{2}\operatorname{Re}\left\langle TSA,\,TS|k|(-\Delta_t)^{-1/2}U'R\right\rangle\right|\le M U'_{\mathrm{sup}}\|TSA\|_{L^2}\frac{1}{M}\|TS(-\Delta_t)^{-1/2}R\|_{L^2},
	\]
	where a direct coercivity argument would require a ``subsonic constraint" $M U'_{\mathrm{sup}}\le 1$. To remove this restriction, we introduce $\Xi=k(-\Delta_t)^{-1/2}(\Theta-U'R)$ and add $\frac{1}{2}\|TS\Xi\|_{L^2}^2$ to the energy. The identity
	\[
	|k|(-\Delta_t)^{-1/2}U'R=\operatorname{sgn}(k)\left[k(-\Delta_t)^{-1/2}\Theta-\Xi\right]
	\]
	leads to
	\[
	\frac{1}{2}\left|\left\langle TSA,\,TS|k|(-\Delta_t)^{-1/2}U'R\right\rangle\right|\leq\frac{1}{3}\|TSA\|_{L^2}^2+\frac{3}{8}\|TS\Theta\|_{L^2}^2+\frac{3}{8}\|TS\Xi\|_{L^2}^2,
	\]
	and accordingly we get
	\begin{align*}
		&~\frac{1}{2}\|TSA\|_{L^2}^2 +\frac{1}{2}\|TS\Theta\|_{L^2}^2+
		\frac{1}{2}\|TS\Xi\|_{L^2}^2+\frac{1}{2}\operatorname{Re}
		\left\langle TSA,\,TS|k|(-\Delta_t)^{-1/2}U'R\right\rangle \\
		\geq&~ \frac{1}{6}\|TSA\|_{L^2}^2+\frac{1}{8}\|TS\Theta\|_{L^2}^2+\frac{1}{8}\|TS\Xi\|_{L^2}^2.
	\end{align*}
	
	Combining these considerations, we shall define the basic energy as follows in order for the coercivity
	\begin{align*}
		E =& \frac{1}{2M^2} \| TSkR \|_{L^2}^2 + \frac{1}{2M^2} \| TSD_tR \|_{L^2}^2 + \frac{1}{2} \| TSA \|_{L^2}^2 + \frac{1}{2} \| TS \Theta \|_{L^2}^2 \\
		&+ \frac{1}{2} \| TS k (- \Delta_t)^{-1/2} (\Theta - U'R) \|_{L^2}^2 + \frac{1}{3} \re \langle PTSA, PTSR \rangle.
	\end{align*}
	
	To justify this construction, it remains to control the commutators of the weights with \(U'\) and \(U''\), as well as those involving the auxiliary operator \(P\). Pseudodifferential calculus shows that, under the slow-variation assumption, these commutators are of lower order and sufficiently small to be absorbed by the lower-order dissipation generated by \(S\). This yields the non-increasing property of the basic energy.
	
	\subsubsection{Terminal-time-dependent weighted energies}
	The second main ingredient consists of two terminal-time-dependent weights. The higher-order weighted energy compensates for the negative order of the basic weight \(TS\) in the negative-frequency region, whereas the lower-order weighted energy yields a negative-order estimate with decay in the terminal time.
	
	A crucial point is that these terminal-time-dependent weights remain compatible with the variable-coefficient basic energy; namely, they are chosen to commute exactly with \(D_t\). However, they do not commute with \(U'\) and \(U''\), giving rise to additional commutator errors. The careful choice of these new weights ensures that their time derivatives generate additional dissipation, which, under the assumption \(\| U'' \|_{W^{n,\infty}} \le \epsilon\), absorbs these commutator errors.
	
	\paragraph{Higher-order weighted energy.}
	The basic weight \(TS\) has negative order in the negative-frequency region, whereas it already provides sufficient positive regularity in the positive-frequency region. Hence, to obtain the desired \(L^2\)-bounds, we need to strengthen the energy only in the negative-frequency region. We therefore choose a positive function \(g\) of order \(1\) for negative frequencies and order \(0\) for positive frequencies, with \(g' < 0\).
	
	For each fixed terminal time \(s \ge 0\), we introduce
	\[G_s = g \left( \frac{D_t}{ik} + 2U'_{\sup} (t-s) \right).\]
	When differentiating the \(G_s\)-weighted energy, we distinguish whether the time derivative falls on \(G_s\) or on the remaining factors. Since \(G_s\) commutes with \(D_t\), \(\Delta_t\), \(T\), and \(S\), the terms in the latter class have the same structure as the corresponding terms in the basic energy estimate, apart from commutator errors involving \(G_s\). The former class contains the essential new contribution. At the principal-symbol level,
	\[(\partial_t G_s) G_s^{-1} \approx (2U'_{\sup} - U'(y)) g' \left( \frac{D_t}{ik} + 2U'_{\sup} (t-s) \right) g \left( \frac{D_t}{ik} + 2U'_{\sup} (t-s) \right)^{-1} < 0.\]
	This favorable negative sign is a consequence of two choices: the asymmetric frequency orders of \(g\) allow it to be chosen with \(g' < 0\) throughout the frequency space, while the terminal-time-dependent translation changes the coefficient arising from time differentiation from \(-U'(y) < 0\) to \(2U'_{\sup} - U'(y) > 0\). Combining the resulting dissipative contribution with the basic energy estimate, we conclude that the \(G_s\)-weighted energy is non-increasing on \(0 \le t \le s\).
	
	At the terminal time \(t = s\), the translation vanishes, and comparison of the weighted energies at \(t = 0\) and \(t = s\) yields the desired higher-order estimate.
	
	\paragraph{Lower-order weighted energy.}
	For the complementary lower-order estimate, we proceed similarly. For each fixed terminal time \(s \ge 0\), we introduce
	\[H_s = h \left( \frac{D_t}{ik} + \frac{1}{2} U'_{\inf} (t-s) \right),\]
	where \(h\) is positive, of order \(-2\) in the negative-frequency region and of order \(0\) in the positive-frequency region, with \(h' > 0\) throughout the frequency space.
	
	As in the \(G_s\)-weighted estimate, we focus on the terms generated when the time derivative falls on \(H_s\). To leading order,
	\[(\partial_t H_s) H_s^{-1} \approx \left( \frac{1}{2} U'_{\inf} - U'(y) \right) h' \left( \frac{D_t}{ik} + \frac{1}{2} U'_{\inf} (t-s) \right) h \left( \frac{D_t}{ik} + \frac{1}{2} U'_{\inf}(t-s) \right)^{-1} < 0.\]
	Unlike in the higher-order construction, the favorable negative sign is already present in the unshifted case and is preserved by the translation, whose role here is instead to generate decay in \(s\). At \(t = 0\), the translation produces a factor of order \(\langle s \rangle^{-2}\), owing to the order \(-2\) of \(h\) in the negative-frequency region, while at \(t = s\) the translation vanishes. Comparing the resulting non-increasing \(H_s\)-weighted energy at \(t=0\) and \(t=s\) then yields the desired lower-order estimate with decay in \(s\).
	
	Interpolating between the higher- and lower-order estimates yields growth bounds for the density perturbation and the irrotational component of the velocity perturbation.
	
	An advantage of this construction is that both estimates are derived from the same basic energy, without differentiating the system or introducing separate energies for derivatives of the unknowns. Moreover, the terminal-time dependence converts the monotonicity of the weighted energies into explicit growth and decay rates in \(s\).
	
	Finally, we apply Duhamel's formula to \(\Theta\) and recover \(W\) from \(W = \Theta - LR\). Combining the resulting representation with the propagator bounds, the mixing estimates, and the higher- and lower-order energy estimates yields the componentwise decay bounds for \(v_{\mathrm{sol}}\) and completes the proof of the main theorem.
	
	\subsection{Outline of the paper}
	
	The remainder of the paper is organized as follows. In Section \ref{sec2}, we recall the necessary preliminaries on pseudodifferential operators. In Section \ref{sec3}, we reformulate the linearized system in sheared coordinates and introduce the basic weighted energy structure. Section \ref{sec4} establishes the symbolic and operator estimates needed for the energy method. In Section \ref{sec5}, we use these properties to prove the basic weighted energy estimate. Section \ref{sec6} is devoted to the construction and monotonicity of the higher- and lower-order weighted energies. Finally, in Section \ref{sec7}, we establish the boundedness of the evolution propagator and the required mixing estimate, and then combine these with the energy estimates and the Duhamel representation to derive the growth and inviscid damping estimates, thereby completing the proof of the main theorem.
	
	\subsection{Notations}
	\label{subsec notations}
	
	Throughout the paper, we use the following notations for Fourier transforms and Sobolev spaces. 
	\begin{itemize}
		\item We write \(\langle \xi \rangle := \sqrt{1 + \xi^2}\).
		\item We define the Fourier transforms in the \(x\) and \(y\) directions by
		\[\mathcal{F}_x f(k) = \frac{1}{\sqrt{2\pi}} \int_{\mathbb{T}} e^{-ikx} f(x) dx,\qquad \mathcal{F}_y f(\eta) = \frac{1}{\sqrt{2\pi}} \int_{\mathbb{R}} e^{-iy\eta} f(y) dy.\]
		\item We use Sobolev spaces \(H_x^s H_y^r\), equipped with the norm 
		\[\| f \|_{H_x^s H_y^r}^2 = \sum_{k \in \mathbb{Z}} \int_{\mathbb{R}} \langle k \rangle^{2s} \langle \eta \rangle^{2r} | \mathcal{F}_x \mathcal{F}_y f (k, \eta) |^2 d \eta.\]
		\item We also use the \(k\)-scaled Sobolev norm \[\| f \|_{\widetilde{H}^r} \coloneqq \left\| \left\langle \frac{\partial_y}{k} \right\rangle^r f \right\|_{L^2}.\]
	\end{itemize}
	
	Since the Fourier modes in the \(x\)-direction are decoupled in the linearized system, we fix a nonzero Fourier mode \(k \in \mathbb{Z} \setminus \{0\}\) throughout the mode-by-mode analysis preceding the proof of Theorem \ref{thm main}. Unless otherwise specified, the norms \(\| \cdot \|_{L^2}\) and \(\| \cdot \|_{H^r}\), as well as the inner product \(\langle \cdot, \cdot \rangle\), are taken with respect to the \(y\)-variable.
	
	\medskip
	
	We introduce the following notations for inequalities between two nonnegative quantities \(A,B\).
	\begin{itemize}
		\item We write \(A \lesssim B\) if \(A \le CB\) for some constant \(C > 0\). Unless otherwise specified, the implicit constant may depend on \(U'_{\inf}\), \(U'_{\sup}\), and \(M\), but is independent of \(k, t\) and \(\epsilon\).
		\item The notation \(A \approx B\) means that \(A \lesssim B\) and \(B \lesssim A\) both hold.
		\item The notation \(O(\epsilon)\) denotes a quantity whose absolute value is bounded by \(C \epsilon\), where \(C\) has the same allowable dependence as the implicit constants in \(\lesssim\).
		\item We write \(A \ll B\) if \(A \le \delta B\), where \(\delta > 0\) is a sufficiently small absolute constant independent of \(U'_{\inf}\), \(U'_{\sup}\), and \(M\).
	\end{itemize}

	\noindent{\bf Acknowledgments.} J. Zhang was supported in part by National Natural Science Foundation of China (No. 12601416). L. Zhao was supported by National Natural Science Foundation of China (No. 12341102 and No. 12271497).
	
	\section{Preliminaries}
	\label{sec2}
	
	In this section, we recall some basic notions and facts about pseudodifferential operators; see, e.g., \cite{Taylor1991, Wong2014}. 
	
	Let \(a(y,\eta)\) be a symbol, and denote by \(\op(a)\) its left quantization, defined by
	\[\op(a)f(y) = \frac{1}{\sqrt{2\pi}} \int_{\mathbb{R}} e^{iy\eta} a(y,\eta) \mathcal{F}_y f(\eta) d\eta.\]
	Conversely, given a pseudodifferential operator \(A\), we denote its symbol by \(\sigma(A)\).
	
	For two symbols \(a\) and \(b\), we define their symbolic composition by
	\[a \# b = \sigma(\op(a) \circ \op(b)),\]
	where \(\circ\) denotes the composition of operators. More explicitly,
	\begin{equation}
		\label{composition_formula}
		(a \# b) (y,\eta) = \frac{1}{2\pi} \int_{\mathbb{R}} \int_{\mathbb{R}} e^{-iz\lambda} a(y,\eta+\lambda) b(y+z,\eta) d \lambda dz.
	\end{equation}
	All integrals of this type are understood as oscillatory integrals. For smooth symbols satisfying the standard symbol estimates, the composition admits the well-known asymptotic expansion
	\begin{equation}
		\label{composition_expansion}
		a \# b \sim \sum_{m \ge 0} \frac{(-i)^m}{m!} \partial_\eta^m a \partial_y^m b.
	\end{equation}
	The symbols considered in this paper are only assumed to possess finitely many controlled derivatives and need not satisfy the hypotheses under which \eqref{composition_expansion} holds. Nevertheless, the formal expansion \eqref{composition_expansion} provides a convenient way to identify the principal symbol of a composition. Once these terms have been identified, all remainder estimates will be justified directly from the exact composition formula \eqref{composition_formula}, together with a finite-order Taylor expansion with integral remainder and integration by parts in the oscillatory variables.
	
	We define the symbolic commutator by
	\[[a,b] = a \# b - b \# a.\]
	For smooth symbols satisfying the standard symbol estimates, \eqref{composition_expansion} gives
	\begin{equation}
		\label{commutator_expansion}
		[a,b] \sim \sum_{m \ge 1} \frac{(-i)^m}{m!} (\partial_\eta^m a \partial_y^m b - \partial_\eta^m b \partial_y^m a).
	\end{equation}
	Hence, the principal symbol of the commutator is \(-i (\partial_\eta a \partial_y b - \partial_y a \partial_\eta b)\). As above, for the symbols used in this paper, this expression serves only to identify the leading term, while the corresponding remainder will always be estimated directly from \eqref{composition_formula}.
	
	The composition formula \eqref{composition_formula} also gives two special cases that will be used frequently. If \(a = a(y)\) depends only on \(y\) or \(b = b(\eta)\) depends only on \(\eta\), then
	\[a \# b = ab.\]
	Thus, composition from the left with a multiplication operator, or from the right with a Fourier multiplier, reduces to ordinary multiplication.
	
	We shall also use the symbol of the adjoint operator. If we define
	\[a^*(y,\eta) = \sigma (\op(a)^*),\]
	then
	\[a^*(y,\eta) = \frac{1}{2\pi} \int_{\mathbb{R}} \int_{\mathbb{R}} e^{-iz\lambda} \overline{a}(y + z, \eta + \lambda) d\lambda dz.\]
	In particular, if \(a = a(y)\) or \(a = a(\eta)\), then
	\[a^* = \overline{a}.\]
	Thus, real-valued multiplication symbols and Fourier multipliers define self-adjoint operators.
	
	We finally record the \(L^2\)-boundedness properties used below. If \(a = a(y)\), then \(\op(a)\) is multiplication by \(a\), and hence
	\[\| \op(a) \|_{L^2\to L^2} = \| a \|_{L^\infty}.\]
	If \(a = a(\eta)\), then \(\op(a)\) is a Fourier multiplier, and Plancherel's theorem gives
	\[\| \op(a) \|_{L^2\to L^2} = \| a \|_{L^\infty}.\]
	For general symbols, we use the following one-dimensional version of the Calder\'on--Vaillancourt theorem \cite{Hwang1987}.
	
	\begin{theorem}[Calder\'on--Vaillancourt]
		\label{thm CV}
		Let \(a:\mathbb R\times\mathbb R\to\mathbb C\) be a continuous symbol satisfying
		\[\| \partial_y^j \partial_\eta^l a \|_{L_y^\infty L_\eta^\infty} < \infty, \qquad j,l \in\{0,1\}.\]
		Then
		\[ \| \op(a) \|_{L^2\to L^2} \lesssim \max_{j, l \in \{0,1\}} \| \partial_y^j \partial_\eta^l a \|_{L_y^\infty L_\eta^\infty}.\]
	\end{theorem}
	
	\section{Reformulation and basic energy structure}
	\label{sec3}
	
	In this section, we carry out the transformation of the linearized system described in the introduction. After taking the Fourier transform in the \(x\)-variable and passing to the sheared coordinates, we derive the evolution system for \(R\), \(A\), and \(\Theta\). We then give the precise definitions of the operators entering the basic energy estimate and state the estimate outlined in the introduction. Its proof relies on the operator bounds established in the next section.
	
	We first perform the Fourier transform in the \(x\)-direction. Since the coefficients of the linearized system are independent of \(x\), the Fourier transform reduces the system to a family of decoupled equations for each Fourier mode. The zero \(x\)-average assumption eliminates the zero Fourier mode. We therefore fix \(k \in \mathbb{Z} \setminus \{0\}\) until the proof of Theorem \ref{thm main}, where the estimates for the individual modes are summed over \(k\).
	
	Denoting the Fourier modes of \(\rho, \alpha, \omega\) by \(\rho_k, \alpha_k, \omega_k\), respectively, we obtain
	\begin{equation}
		\label{k_mode}
		\begin{cases}
			\partial_t \rho_k + ikU \rho_k + \alpha_k = 0,\\
			\partial_t \alpha_k + ikU \alpha_k + 2 ikU' (\partial_y \Delta_k^{-1} \alpha_k + ik \Delta_k^{-1} \omega_k) + \frac{1}{M^2} \Delta_k \rho_k = 0,\\
			\partial_t \omega_k + ikU \omega_k - U' \alpha_k = U'' (\partial_y \Delta_k^{-1} \alpha_k + ik \Delta_k^{-1} \omega_k).
		\end{cases}
	\end{equation}	
	where \(\Delta_k = - k^2 + \partial_y^2\), whose inverse is well-defined for \(k \in \mathbb{Z} \setminus \{ 0 \}\).
	
	To remove the transport part induced by the background shear flow, we introduce the sheared coordinates as in \cite{Zillinger2017}. Equivalently, we define the unitary operator \(V=e^{ikU(y)t}\), and introduce the transformed variables
	\begin{equation}
		\label{coordinate_transformation}
		R = V \rho_k, \qquad A = V \alpha_k,\qquad W=V \omega_k.
	\end{equation}
	For simplicity, we omit the dependence of these variables on the Fourier mode \(k\). Then the system \eqref{k_mode} is transformed into
	\begin{equation}
		\label{R_A_W}
		\begin{cases}
			\partial_t R + A = 0,\\
			\partial_t A + 2ikU'(D_t \Delta_t^{-1} A + ik \Delta_t^{-1} W) + \frac{1}{M^2} \Delta_t R = 0,\\
			\partial_t W - U' A = U''(D_t \Delta_t^{-1} A + ik \Delta_t^{-1} W),
		\end{cases}
	\end{equation}
	where
	\[D_t = V \partial_y V^{-1} = \partial_y - ikU't, \qquad \Delta_t = V \Delta_k V^{-1} = - k^2 + D_t^2.\]
	
	It is convenient to introduce the new variable
	\[\Theta = W + U'R + U'' D_t \Delta_t^{-1} R \eqqcolon W + LR.\]
	Instead of working directly with the evolution equation for \(W\), we use the modified variable \(\Theta\). The two correction terms \(U'R\) and \(U'' D_t \Delta_t^{-1} R\) are chosen together so that all terms involving \(A\) cancel from the evolution equation for \(\Theta\). The resulting equation is simpler and better suited to the subsequent energy estimates. Indeed, we calculate that
	\[\partial_t \Theta = \partial_t W + U' \partial_t R + U'' \partial_t (D_t \Delta_t^{-1}) R + U'' D_t \Delta_t^{-1} \partial_t R = ikU'' \Delta_t^{-1} W + U'' \partial_t (D_t \Delta_t^{-1}) R.\] 
	Expressing \(W\) in terms of \(\Theta\), we obtain
	\[\partial_t \Theta = ikU'' \Delta_t^{-1} \Theta + [U'' \partial_t (D_t \Delta_t^{-1}) - ik U'' \Delta_t^{-1} U' - ik U'' \Delta_t^{-1} U'' D_t \Delta_t^{-1}] R.\]
	It remains to simplify the operator. For simplicity, we denote it by \(K\). More precisely,
	\[K \coloneqq U'' \partial_t (D_t \Delta_t^{-1}) - ik U'' \Delta_t^{-1} U' - ik U'' \Delta_t^{-1} U'' D_t \Delta_t^{-1}.\]
	Since \(\partial_t D_t = -ikU'\), it follows that
	\[\partial_t \Delta_t = (\partial_t D_t) D_t + D_t (\partial_t D_t) = -ik (2U'D_t + U''),\]
	and therefore
	\[\partial_t \Delta_t^{-1} = - \Delta_t^{-1} (\partial_t \Delta_t) \Delta_t^{-1} = ik \Delta_t^{-1} (2U' D_t + U'') \Delta_t^{-1}.\] 
	Thus
	\begin{align*}
		K &= -ikU' U'' \Delta_t^{-1} + ik U'' D_t \Delta_t^{-1} (2U' D_t + U'') \Delta_t^{-1} -ik U'' \Delta_t^{-1} U' -ik U'' \Delta_t^{-1} U'' D_t \Delta_t^{-1} \\
		&= ik U'' \Delta_t^{-1} [- \Delta_t U' + D_t (2U' D_t + U'') - U' \Delta_t - U'' D_t] \Delta_t^{-1}.
	\end{align*}
	To simplify the operator in brackets, note that \([D_t, U'] = U''\) and \([D_t, U''] = U'''\), which imply
	\[D_t (2U' D_t + U'') = U' D_t^2 + D_t^2 U' + U'' D_t.\]
	Therefore, since \(\Delta_t = -k^2 + D_t^2\), we have
	\[-U' \Delta_t -U'' D_t - \Delta_t U' + D_t (2U' D_t + U'') = 2k^2 U',\]
	and hence
	\[K = 2ik^3U'' \Delta_t^{-1} U' \Delta_t^{-1}.\]
	
	Consequently, the system \eqref{R_A_W} can be equivalently rewritten as
	\begin{equation}
		\label{R_A_Theta}
		\begin{cases}
			\partial_t R + A = 0, \\
			\partial_t A + 2ikU'(D_t \Delta_t^{-1} A + ik \Delta_t^{-1} (\Theta - U'R - U'' D_t \Delta_t^{-1} R)) + \frac{1}{M^2} \Delta_t R = 0,\\
			\partial_t \Theta = ikU'' \Delta_t^{-1} \Theta + 2ik^3U'' \Delta_t^{-1} U' \Delta_t^{-1} R.
		\end{cases}
	\end{equation}
	The reformulated system provides a suitable framework for the energy method developed below. Within this framework, we define the following operators to state the basic energy structure.
	\[T = V \op \left( \exp \left(\frac{1}{\gamma} \sinh^{-1} \frac{\eta}{k} \right) \right) V^{-1}, \qquad S = V \op \left( \exp \left( \frac{1}{\beta} \tan^{-1} \frac{\eta}{k} \right) \right) V^{-1},\]
	\[P = V \op \left( \left( \frac{U'(y) |k|}{\gamma \sqrt{k^2 + \eta^2}} \right)^{1/2} \right) V^{-1}, \qquad Q = V \op \left( \left( \frac{U'(y) k^2}{\beta (k^2 + \eta^2)} \right)^{1/2} \right) V^{-1},\]
	where \(\gamma = \frac{2}{3}\) and \(\beta = \beta(U'_{\sup}, M)\) is chosen so that \(0 < \beta \ll \min \{1, (MU'_{\sup})^{-1} \}\). The sufficiently small constant implicit in \(\ll\) is absolute and, in particular, independent of \(U'_{\inf}\), \(U'_{\sup}\), and \(M\).
	
	With the transformed system and the associated weighted operators defined above, we now state the basic energy estimate.
	
	\begin{proposition}
		\label{prop basic_energy}
		Define the energy functional
		\begin{equation}
			\label{energy}
			\begin{aligned}
				E =& \frac{1}{2M^2} \| TSkR \|_{L^2}^2 + \frac{1}{2M^2} \| TSD_tR \|_{L^2}^2 + \frac{1}{2} \| TSA \|_{L^2}^2 + \frac{1}{2} \| TS \Theta \|_{L^2}^2 \\
				&+ \frac{1}{2} \| TS k (- \Delta_t)^{-1/2} (\Theta - U'R) \|_{L^2}^2 + \frac{1}{3} \re \langle PTSA, PTSR \rangle.
			\end{aligned}
		\end{equation}
		Assume that \(U\) satisfies the shear-profile assumptions of Theorem \ref{thm main}, with \(\epsilon = \epsilon(U'_{\inf}, U'_{\sup}, M)\) sufficiently small. For notational convenience, we set
		\[\Xi = k (-\Delta_t)^{-1/2} (\Theta-U'R).\]
		Then
		\begin{equation}
			\label{equivalent_energy}
			E \approx \frac{1}{M^2} \| TS \sqrt{-\Delta_t} R \|_{L^2}^2 + \| TSA \|_{L^2}^2 + \| TS \Theta \|_{L^2}^2 + \| TS \Xi \|_{L^2}^2.
		\end{equation}
		Moreover, if \((R,A,\Theta)\) is a solution of \eqref{R_A_Theta}, then there exists an absolute constant \(c_0 > 0\), independent of \(U'_{\inf}\), \(U'_{\sup}\), and \(M\), such that
		\begin{equation}
			\label{derivative}
			\begin{aligned}
				\frac{d}{dt}E \le& -c_0 \bigg( \frac{1}{M^2} \| PTSkR \|_{L^2}^2 + \| PTS \Theta \|_{L^2}^2 + \| PTS \Xi \|_{L^2}^2 + \frac{1}{M^2} \| QTSkR \|_{L^2}^2 \\
				&+ \frac{1}{M^2} \| QTSD_tR \|_{L^2}^2 + \| QTSA \|_{L^2}^2 + \| QTS \Theta \|_{L^2}^2 + \| QTS \Xi \|_{L^2}^2 \bigg) \\
				\le&~ 0.
			\end{aligned}
		\end{equation}
		In particular,
		\[E(t)\le E(0).\]
	\end{proposition}
	
	The proof of Proposition~\ref{prop basic_energy} is postponed to Section \ref{sec5}, after the necessary operator estimates are established in the next section.
	
	\section{Properties of the energy operators}
	\label{sec4}
	
	In order to prove the basic energy estimate, we first collect some properties of the operators introduced in the previous section. These properties will be repeatedly used in the energy estimates below.
	
	\subsection{The basic dissipative structure}
	
	The following lemma summarizes the symbolic properties of the operators introduced above. In particular, we compute the principal symbols of \(P^*P\), \(Q^*Q\), and \([\partial_t(TS)](TS)^{-1}\).
	
	\begin{lemma}
		\label{lem P_Q_TS}
		For every \(f\) such that \(TSf \in L^2\), we have
		\[\left| \| PTSf \|_{L^2}^2 - \re \left\langle V \op \left( \frac{1}{\gamma} U'(y) \left\langle \frac{\eta}{k} \right\rangle^{-1} \right) V^{-1} TSf, TSf \right\rangle \right| \lesssim \epsilon \| QTSf \|_{L^2}^2,\]
		and
		\[\left| \| QTSf \|_{L^2}^2 - \re \left\langle V \op \left( \frac{1}{\beta} U'(y) \left\langle \frac{\eta}{k} \right\rangle^{-2} \right) V^{-1} TSf, TSf \right\rangle \right| \lesssim \epsilon \| QTSf \|_{L^2}^2.\]
		Moreover,
		\[\left| \re \langle \partial_t (TS) f, TS f \rangle + \| PTSf \|_{L^2}^2 + \| QTSf \|_{L^2}^2 \right| \lesssim \epsilon \| QTSf \|_{L^2}^2.\]
	\end{lemma}
	
	\begin{proof}
		We first compute the principal symbol of \(P^*P\). We work in the sheared coordinates and conjugate by the unitary operator \(V\). Recall that
		\[\sigma(V^{-1} P V) = \left( \frac{U'(y) |k|}{\gamma \sqrt{k^2 + \eta^2}} \right)^{1/2} = \gamma^{-1/2} U'(y)^{1/2} \# \left\langle \frac{\eta}{k} \right\rangle^{-1/2}.\]
		Therefore,
		\[\sigma(V^{-1} P^*P V) = \frac{1}{\gamma} \left\langle \frac{\eta}{k} \right\rangle^{-1/2} \# U'(y) \# \left\langle \frac{\eta}{k} \right\rangle^{-1/2}.\]
		We decompose this symbol into its principal part and a remainder term. More precisely,
		\begin{equation}
			\label{P^*P}
			\sigma(V^{-1} P^*P V) = \frac{1}{\gamma} U'(y) \left\langle \frac{\eta}{k} \right\rangle^{-1} + r_P (y, \eta),
		\end{equation}
		where \(r_P\) is given by
		\[r_P (y, \eta) = \frac{1}{\gamma} \left[ \left\langle \frac{\eta}{k} \right\rangle^{-1/2}, U'(y) \right] \left\langle \frac{\eta}{k} \right\rangle^{-1/2}.\]
		
		It remains to show that the remainder term is small. By the composition formula \eqref{composition_formula},
		\[\left\langle \frac{\eta}{k} \right\rangle^{-1/2} \# U'(y) = \frac{1}{2\pi} \int_{\mathbb{R}}\int_{\mathbb{R}} e^{-iz\lambda} \left\langle \frac{\eta + \lambda}{k} \right\rangle^{-1/2} U'(y+z) d \lambda dz.\]
		Applying Taylor's theorem with the integral remainder,
		\[\left\langle \frac{\eta + \lambda}{k} \right\rangle^{-1/2} = \left\langle \frac{\eta}{k} \right\rangle^{-1/2} - \lambda \int_{0}^{1} \frac{\eta + \theta \lambda}{2k^2} \left\langle \frac{\eta + \theta \lambda}{k} \right\rangle^{-5/2} d \theta.\]
		Hence,
		\[\left[ \left\langle \frac{\eta}{k} \right\rangle^{-1/2}, U'(y) \right] = - \frac{1}{2\pi} \int_{\mathbb{R}}\int_{\mathbb{R}} \int_{0}^{1} e^{-iz\lambda} \lambda \frac{\eta + \theta \lambda}{2k^2} \left\langle \frac{\eta + \theta \lambda}{k} \right\rangle^{-5/2} U'(y+z) d \theta d \lambda dz.\]
		Integrating by parts in \(z\) and \(\lambda\),
		\[\left[ \left\langle \frac{\eta}{k} \right\rangle^{-1/2}, U'(y) \right] = \frac{i}{2\pi} \int_{0}^{1} \int_{\mathbb{R}}\int_{\mathbb{R}} e^{-iz\lambda} \langle z \rangle^{-2} (1 - \partial_\lambda^2) \langle \lambda \rangle^{-4} (1 - \partial_z^2)^2 \frac{\eta + \theta \lambda}{2k^2} \left\langle \frac{\eta + \theta \lambda}{k} \right\rangle^{-5/2} U''(y+z) d \lambda dz d \theta.\]
		Using the smallness assumption on \(U''\), we obtain
		\[\left| \left[ \left\langle \frac{\eta}{k} \right\rangle^{-1/2}, U'(y) \right] \right| \lesssim \epsilon \int_{0}^{1} \int_{\mathbb{R}}\int_{\mathbb{R}} \langle z \rangle^{-2} \langle \lambda \rangle^{-4} \left\langle \frac{\eta}{k} \right\rangle^{-3/2} \langle \lambda \rangle^{3/2} d \lambda dz d \theta \lesssim \epsilon \left\langle \frac{\eta}{k} \right\rangle^{-3/2}.\]
		Consequently, for \(j \in \{0, 1, \cdots, 5\}, l \in \{0, 1\}\),
		\begin{equation}
			\label{r_P}
			|\partial_y^j \partial_\eta^l r_P (y, \eta)| \lesssim \epsilon \left\langle \frac{\eta}{k} \right\rangle^{-2 - l}.
		\end{equation}
		
		To show that the remainder is small compared with the principal part, we consider the normalized operator \((Q^{-1})^* \mathcal{R}_P Q^{-1}\), where \(\mathcal{R}_P = V \op (r_P) V^{-1}\). The symbol of this operator is given by
		\[\sigma (V^{-1} (Q^{-1})^* \mathcal{R}_P Q^{-1} V) = \beta U'(y)^{-1/2} \# \left\langle \frac{\eta}{k} \right\rangle \# r_P(y, \eta) \# \left\langle \frac{\eta}{k} \right\rangle \# U'(y)^{-1/2}.\]
		Since \(U'\) is bounded from above and below, it is sufficient to estimate the symbol \(\left\langle \frac{\eta}{k} \right\rangle \# r_P \# \left\langle \frac{\eta}{k} \right\rangle\). Using the composition formula \eqref{composition_formula}, we have
		\[\left\langle \frac{\eta}{k} \right\rangle \# r_P (y, \eta) \# \left\langle \frac{\eta}{k} \right\rangle = \frac{1}{2\pi} \int_{\mathbb R} \int_{\mathbb R} e^{-iz\lambda} \left\langle \frac{\eta + \lambda}{k} \right\rangle r_P(y+z,\eta) \left\langle \frac{\eta}{k} \right\rangle d \lambda dz.\]
		Integrating by parts in \(z\) and \(\lambda\), we obtain
		\[\left\langle \frac{\eta}{k} \right\rangle \# r_P (y, \eta) \# \left\langle \frac{\eta}{k} \right\rangle = \frac{1}{2 \pi} \int_{\mathbb{R}} \int_{\mathbb{R}} e^{-i z \lambda} \langle z \rangle^{-2} (1 - \partial_\lambda^2) \langle \lambda \rangle^{-4} (1 - \partial_z^2)^2 \left\langle \frac{\eta + \lambda}{k} \right\rangle r_P (y + z, \eta) \left\langle \frac{\eta}{k} \right\rangle d \lambda dz.\]
		Using the estimate \eqref{r_P},
		\[\left| \left\langle \frac{\eta}{k} \right\rangle \# r_P \# \left\langle \frac{\eta}{k} \right\rangle \right| \lesssim \epsilon \int_{\mathbb{R}} \int_{\mathbb{R}} \langle z \rangle^{-2} \langle \lambda \rangle^{-4} \left\langle \frac{\eta}{k} \right\rangle \left\langle \lambda \right\rangle \left\langle \frac{\eta}{k} \right\rangle^{-2} \left\langle \frac{\eta}{k} \right\rangle d \lambda dz \lesssim \epsilon.\]
		The same estimate holds for the derivatives of this symbol with respect to \(y\) and \(\eta\). Therefore, by the Calder\'on--Vaillancourt Theorem \ref{thm CV},
		\[\left\| \op \left( \left\langle \frac{\eta}{k} \right\rangle \# r_P \# \left\langle \frac{\eta}{k} \right\rangle \right) \right\|_{L^2 \to L^2} \lesssim \epsilon,\]
		which implies
		\[\| (Q^{-1})^* \mathcal{R}_P Q^{-1} \|_{L^2 \to L^2} \lesssim \epsilon.\]
		Thus, the first estimate holds. The proof of the second is similar.
		
		Now we estimate the dissipation generated by differentiating the weight \(TS\). Directly from the definition of \(T\) and \(S\), we have
		\[\partial_t (TS) = ik V \left[ U, \op \left( \exp \left(\frac{1}{\gamma} \sinh^{-1} \frac{\eta}{k} + \frac{1}{\beta} \tan^{-1} \frac{\eta}{k} \right) \right) \right] V^{-1}.\]
		By the composition formula \eqref{composition_formula},
		\[\exp \left(\frac{1}{\gamma} \sinh^{-1} \frac{\eta}{k} + \frac{1}{\beta} \tan^{-1} \frac{\eta}{k} \right) \# U (y) = \frac{1}{2 \pi} \int_{\mathbb{R}} \int_{\mathbb{R}} e^{-i z \lambda} \exp \left(\frac{1}{\gamma} \sinh^{-1} \frac{\eta + \lambda}{k} + \frac{1}{\beta} \tan^{-1} \frac{\eta + \lambda}{k} \right) U(y + z) d \lambda dz.\]
		Using Taylor's theorem,
		\begin{align*}
			&\exp \left(\frac{1}{\gamma} \sinh^{-1} \frac{\eta + \lambda}{k} + \frac{1}{\beta} \tan^{-1} \frac{\eta + \lambda}{k} \right) \\
			=& \exp \left(\frac{1}{\gamma} \sinh^{-1} \frac{\eta}{k} + \frac{1}{\beta} \tan^{-1} \frac{\eta}{k} \right) + \lambda \frac{1}{k} \left( \frac{1}{\gamma} \left\langle \frac{\eta}{k} \right\rangle^{-1} + \frac{1}{\beta} \left\langle \frac{\eta}{k} \right\rangle^{-2} \right) \exp \left(\frac{1}{\gamma} \sinh^{-1} \frac{\eta}{k} + \frac{1}{\beta} \tan^{-1} \frac{\eta}{k} \right) \\
			&+ \lambda^2 \int_{0}^{1} (1 - \theta) \frac{r_2 (\eta + \theta \lambda)}{k^2} \exp \left(\frac{1}{\gamma} \sinh^{-1} \frac{\eta + \theta \lambda}{k} + \frac{1}{\beta} \tan^{-1} \frac{\eta + \theta \lambda}{k} \right) d \theta.
		\end{align*}
		Here the remainder \(r_2\) satisfies
		\begin{equation}
			\label{r_2}
			| \partial_\lambda^j r_2 (\eta + \theta \lambda) | \lesssim \left\langle \frac{\eta + \theta \lambda}{k} \right\rangle^{-2 - j},
		\end{equation}
		for \(j \in \{0, 1, 2, 3\}\). The zeroth-order term cancels in the commutator, while the first-order term produces the principal contribution. Indeed,
		\begin{align*}
			&\left[ U(y), \exp \left(\frac{1}{\gamma} \sinh^{-1} \frac{\eta}{k} + \frac{1}{\beta} \tan^{-1} \frac{\eta}{k} \right) \right] \\
			=& \frac{i U'(y)}{k} \left( \frac{1}{\gamma} \left\langle \frac{\eta}{k} \right\rangle^{-1} + \frac{1}{\beta} \left\langle \frac{\eta}{k} \right\rangle^{-2} \right) \exp \left(\frac{1}{\gamma} \sinh^{-1} \frac{\eta}{k} + \frac{1}{\beta} \tan^{-1} \frac{\eta}{k} \right) \\
			&- \frac{1}{2 \pi} \int_{\mathbb{R}} \int_{\mathbb{R}} \int_{0}^{1} e^{-i z \lambda} \lambda^2 (1 - \theta) \frac{r_2 (\eta + \theta \lambda)}{k^2} \exp \left(\frac{1}{\gamma} \sinh^{-1} \frac{\eta + \theta \lambda}{k} + \frac{1}{\beta} \tan^{-1} \frac{\eta + \theta \lambda}{k} \right) U(y + z) d \theta d \lambda dz.
		\end{align*}
		
		After multiplying by the prefactor \(ik\) in \(\partial_t(TS)\) and conjugating by \((TS)^{-1}\), we obtain
		\begin{align*}
			\sigma(V^{-1} [\partial_t (TS)] (TS)^{-1} V) =& -U'(y) \left( \frac{1}{\gamma} \left\langle \frac{\eta}{k} \right\rangle^{-1} + \frac{1}{\beta} \left\langle \frac{\eta}{k} \right\rangle^{-2} \right) - \frac{i}{2 \pi} \int_{\mathbb{R}} \int_{\mathbb{R}} \int_{0}^{1} e^{-i z \lambda} \lambda^2 (1 - \theta) \frac{r_2 (\eta + \theta \lambda)}{k} U(y + z) \\
			&\times \exp \left(\frac{1}{\gamma} \sinh^{-1} \frac{\eta + \theta \lambda}{k} + \frac{1}{\beta} \tan^{-1} \frac{\eta + \theta \lambda}{k} \right) \exp \left(- \frac{1}{\gamma} \sinh^{-1} \frac{\eta}{k} - \frac{1}{\beta} \tan^{-1} \frac{\eta}{k} \right) d \theta d \lambda dz.
		\end{align*}
		Denote the second term on the right-hand side by \(r_{TS}\); we now estimate it. The exponential factor involving \(\sinh^{-1}\) requires particular care. We first recall a property of this function, which will be proved later.
		
		\begin{claim}
			\label{cla sinh}
			For any \(\xi, \zeta \in \mathbb{R}\), we have
			\[| \sinh^{-1} (\xi + \zeta) - \sinh^{-1} \xi | \le 2 \log \langle \zeta \rangle + 2 \log 2.\]
		\end{claim}
		
		Using this claim, we have
		\[\exp \left( \frac{1}{\gamma} \sinh^{-1} \frac{\eta + \theta \lambda}{k} - \frac{1}{\gamma} \sinh^{-1} \frac{\eta}{k} \right) \lesssim \left\langle \frac{\theta \lambda}{k} \right\rangle^{2/\gamma} \lesssim \langle \lambda \rangle^3.\]
		Moreover, since the arctangent function is bounded,
		\[\exp \left( \frac{1}{\beta} \tan^{-1} \frac{\eta + \theta \lambda}{k} - \frac{1}{\beta} \tan^{-1} \frac{\eta}{k} \right) \lesssim 1.\]
		Combining these estimates with \eqref{r_2}, we first integrate twice by parts in \(z\). Thus, the factor \(\lambda^2\) is transferred to \(U(y+z)\), so that the resulting integrand contains \(U''(y+z)\) and hence gains the small factor \(\epsilon\). Further integrations by parts with respect to the oscillatory variables \(\lambda\) and \(z\) yield
		\[| r_{TS} (y, \eta) | \lesssim \epsilon \int_{0}^{1} \int_{\mathbb{R}} \int_{\mathbb{R}} \langle z \rangle^{-2} \langle \lambda \rangle^{-8} \left\langle \frac{\eta + \theta \lambda}{k} \right\rangle^{-2} \langle \lambda \rangle^3 d \lambda dz d \theta \lesssim \epsilon \left\langle \frac{\eta}{k} \right\rangle^{-2}.\]
		The same estimate holds for the derivatives of \(r_{TS}\) with respect to \(y\) and \(\eta\). Therefore, by the Calder\'on--Vaillancourt Theorem \ref{thm CV} and the same argument as above, we obtain
		\[\| (Q^{-1})^* \mathcal{R}_{TS} Q^{-1} \|_{L^2 \to L^2} \lesssim \epsilon,\]
		where \(\mathcal{R}_{TS} = V \op (r_{TS}) V^{-1}\).
		
		Therefore, the principal symbol of \([\partial_t(TS)](TS)^{-1}\) coincides with the principal symbol of \(- P^*P - Q^*Q\), while the remaining contribution is a small term of size \(O(\epsilon)\). This proves the lemma.
	\end{proof}
	
	We now prove Claim \ref{cla sinh}.
	
	\begin{proof}[Proof of Claim \ref{cla sinh}]
		Let \(u = \sinh^{-1}(\xi + \zeta)\) and \(v = \sinh^{-1}(\xi)\). Then \(\xi + \zeta = \sinh u\) and \(\xi = \sinh v\). Hence
		\[|\zeta| = |\sinh u - \sinh v| = 2 \left| \cosh \frac{u + v}{2} \sinh \frac{u - v}{2}\right| \ge 2 \left|\sinh \frac{u - v}{2} \right| = 2 \sinh \frac{|u - v|}{2},\]
		where we have used \(\cosh \lambda \ge 1\) for all \(\lambda \in \mathbb{R}\). Since \(\sinh^{-1}\) is increasing, it follows that
		\[\frac{|u - v|}{2} \le \sinh^{-1} \frac{|\zeta|}{2}.\]
		Thus,
		\[|\sinh^{-1}(\xi + \zeta) - \sinh^{-1}(\xi)| = |u - v| \le 2 \sinh^{-1} \frac{|\zeta|}{2}.\]
		It remains to estimate the right-hand side. By the definition of
		\(\sinh^{-1}\),
		\[\sinh^{-1} \frac{|\zeta|}{2} = \log \left( \frac{|\zeta|}{2} + \left\langle \frac{|\zeta|}{2} \right\rangle \right) \le \log \left( 2 \langle \zeta \rangle \right) = \log \langle \zeta \rangle + \log 2,\]
		which implies the claim.
	\end{proof}
	
	Lemma \ref{lem P_Q_TS} shows that \(P^*P\) and \(Q^*Q\) agree with their corresponding principal operators up to remainders that are small relative to \(Q^*Q\). We record the resulting bilinear operator bounds in the following corollary, which will be used repeatedly in the energy estimates below.
	
	\begin{corollary}
		\label{cor operator_bounds}
		For every \(f_1, f_2 \in L^2\), we have
		\[|\langle U' k (-\Delta_t)^{-1/2} f_1, f_2 \rangle| \le \gamma \| Pf_1 \|_{L^2} \| Pf_2 \|_{L^2} + O(\epsilon) \| Qf_1 \|_{L^2} \| Qf_2 \|_{L^2},\]
		and
		\[|\langle U'k^2 (-\Delta_t)^{-1} f_1, f_2 \rangle| \le (\beta + O(\epsilon)) \| Qf_1 \|_{L^2} \| Qf_2 \|_{L^2}.\]
	\end{corollary}
	
	\begin{proof}
		The proof of Lemma \ref{lem P_Q_TS} gives
		\[\| (Q^{-1})^* (U' |k| (-\Delta_t)^{-1/2} - \gamma P^*P) Q^{-1} \|_{L^2 \to L^2} \lesssim \epsilon.\]
		Consequently,
		\begin{align*}
			|\langle U' k (-\Delta_t)^{-1/2} f_1, f_2 \rangle| &\le \gamma |\langle P^*P f_1, f_2 \rangle| + |\langle (U' |k| (-\Delta_t)^{-1/2} - \gamma P^*P) f_1, f_2 \rangle| \\
			&\le \gamma \| Pf_1 \|_{L^2} \| Pf_2 \|_{L^2} + O(\epsilon) \| Qf_1 \|_{L^2} \| Qf_2 \|_{L^2}.
		\end{align*}
		The same argument gives the second estimate. This proves the corollary.
	\end{proof}
	
	\subsection{Commutators involving \(TS\) and related operator bounds}
	
	We shall establish the smallness of the commutators involving the time-dependent weights \(TS\).
	
	\begin{lemma}
		\label{lem TS_commutator}
		For all nonnegative integers \(j\) and \(l\) up to a sufficiently large finite order, we have
		\[|\partial_y^j \partial_\eta^l \sigma (V^{-1} [U',TS] (TS)^{-1} V)| \lesssim \epsilon \left\langle \frac{\eta}{k} \right\rangle^{-1 - l},\]
		and
		\[|\partial_y^j \partial_\eta^l \sigma (V^{-1} [U'',TS] (TS)^{-1} V)| \lesssim \epsilon \left\langle \frac{\eta}{k} \right\rangle^{-1 - l}.\]
	\end{lemma}
	
	\begin{proof}
		The proof is the same as the computation of \([\partial_t(TS)](TS)^{-1}\) in the proof of Lemma \ref{lem P_Q_TS}. Indeed, replacing the multiplication operator \(ikU\) in that computation by \(U'\) or \(U''\), respectively, the same symbolic estimates apply.
	\end{proof}
	
	With the help of Lemma \ref{lem TS_commutator}, we obtain bounds for several operators that will be used later.
	
	\begin{corollary}
		\label{cor TS_commutator_bound}
		The following operator bounds hold:
		\[\left\| (Q^*)^{-1} [U',TS] (TS)^{-1} k(-\Delta_t)^{-1/2} Q^{-1} \right\|_{L^2 \to L^2} \lesssim \epsilon,\]
		\[\left\| (Q^*)^{-1} [U'',TS] (TS)^{-1} k(-\Delta_t)^{-1/2} Q^{-1} \right\|_{L^2 \to L^2} \lesssim \epsilon.\]
	\end{corollary}
	
	\begin{proof}
		After conjugating by \(V\), the symbol of the first operator is
		\[\sigma (V^{-1} (Q^*)^{-1} [U',TS] (TS)^{-1} k(-\Delta_t)^{-1/2} Q^{-1} V) = \beta (\sgn k) U'(y)^{-1/2} \# \left\langle \frac{\eta}{k} \right\rangle \# \sigma (V^{-1} [U',TS] (TS)^{-1} V ) \# U'(y)^{-1/2}.\]
		Since \(U'\) is bounded from above and below, it suffices to estimate the middle composition. By the composition formula \eqref{composition_formula} and Lemma \ref{lem TS_commutator}, this symbol and all its derivatives up to the order required in the Calder\'on--Vaillancourt Theorem \ref{thm CV} are bounded by \(O(\epsilon)\). The first estimate therefore follows from that theorem. The second estimate follows in exactly the same way.
	\end{proof}
	
	We next establish several operator bounds that will be used below.
	
	\begin{lemma}
		\label{lem U'_weighted_bound}
		The following operator bounds hold:
		\begin{align*}
			\| PTSU'(TS)^{-1}P^{-1} \|_{L^2 \to L^2} &\le (1 + O(\epsilon)) U'_{\sup}, & \| QTSU'(TS)^{-1}Q^{-1} \|_{L^2 \to L^2} &\le (1 + O(\epsilon)) U'_{\sup}, \\
			\| PTSU''(TS)^{-1}P^{-1} \|_{L^2 \to L^2} &\lesssim \epsilon, & \| QTSU''(TS)^{-1}Q^{-1} \|_{L^2 \to L^2} &\lesssim \epsilon.
		\end{align*}
	\end{lemma}
	
	\begin{proof}
		It suffices to prove the first estimate, since the others follow from the same argument.
		
		We decompose
		\[P TS U' (TS)^{-1} P^{-1} = U' + [P,U'] P^{-1} + P [TS,U'] (TS)^{-1} P^{-1}.\]
		After conjugating by \(V\), the symbol of the last operator is
		\[\sigma (V^{-1} P [TS,U'] (TS)^{-1} P^{-1} V) = U'(y)^{1/2} \# \left\langle \frac{\eta}{k} \right\rangle^{-1/2} \# \sigma \left( V^{-1} [TS,U'] (TS)^{-1} V \right) \# \left\langle \frac{\eta}{k} \right\rangle^{1/2} \# U'(y)^{-1/2}.\]
		Since \(U'\) is bounded from above and below, it suffices to estimate the composition of the three middle factors. By the composition formula \eqref{composition_formula}, the middle three factors are given by
		\begin{align*}
			&\left\langle \frac{\eta}{k} \right\rangle^{-1/2} \# \sigma \left( V^{-1} [TS,U'] (TS)^{-1} V \right) \# \left\langle \frac{\eta}{k} \right\rangle^{1/2} \\
			=& \frac{1}{2\pi} \int_{\mathbb R} \int_{\mathbb R} e^{-iz\lambda} \left\langle \frac{\eta + \lambda}{k} \right\rangle^{-1/2} \sigma \left( V^{-1} [TS,U'] (TS)^{-1} V \right) (y + z,\eta) \left\langle \frac{\eta}{k} \right\rangle^{1/2} d\lambda dz.
		\end{align*}
		Integrating by parts sufficiently many times in the oscillatory variables and using Lemma \ref{lem TS_commutator}, we obtain
		\[\left| \partial_y^j \partial_\eta^l \left\langle \frac{\eta}{k} \right\rangle^{-1/2} \# \sigma \left( V^{-1} [TS,U'] (TS)^{-1} V \right) \# \left\langle \frac{\eta}{k} \right\rangle^{1/2} \right| \lesssim \epsilon, \qquad j,l \in \{0,1\}.\]
		The Calder\'on--Vaillancourt Theorem \ref{thm CV} and the boundedness of \(U'\) therefore give
		\[\left\| P [U',TS] (TS)^{-1} P^{-1} \right\|_{L^2 \to L^2} \lesssim \epsilon.\]
		The same composition argument, applied to \([P,U']P^{-1}\), shows that
		\[\| [P,U'] P^{-1} \|_{L^2 \to L^2} \lesssim \epsilon,\]
		since every term in this commutator contains at least one derivative of \(U'\). Combining these estimates with \(\| U' \|_{L^\infty} = U'_{\sup}\), we obtain
		\[\| P TS U' (TS)^{-1} P^{-1} \|_{L^2 \to L^2} \le (1 + O(\epsilon)) U'_{\sup}.\]
		
		The other estimates follow in the same way.
	\end{proof}
	
	We also record an operator bound for the time derivative of \((-\Delta_t)^{-1/2}\).
	
	\begin{lemma}
		\label{lem partial_t_Delta}
		The following operator bound holds:
		\[\| (Q^*)^{-1} TS k \partial_t (-\Delta_t)^{-1/2} (TS)^{-1} Q^{-1} \|_{L^2 \to L^2} \le \beta + O(\epsilon).\]
	\end{lemma}
	
	\begin{proof}
		Since \((-\Delta_t)^{-1/2} = V \op ((k^2 + \eta^2)^{-1/2}) V^{-1}\), we have
		\[\sigma (V^{-1} \partial_t (-\Delta_t)^{-1/2} V) = ik [U, (k^2 + \eta^2)^{-1/2}] = ik U (k^2 + \eta^2)^{-1/2} - \frac{ik}{2\pi} \int_{\mathbb{R}} \int_{\mathbb{R}} e^{-iz\lambda} (k^2 + (\eta + \lambda)^2)^{-1/2} U(y + z) d \lambda dz.\]
		By Taylor's theorem,
		\[(k^2 + (\eta + \lambda)^2)^{-1/2} = (k^2 + \eta^2)^{-1/2} - \lambda \frac{\eta}{(k^2 + \eta^2)^{3/2}} + \lambda^2 \int_{0}^{1} (1 - \theta) \partial_\eta^2 \left(k^2 + (\eta + \theta\lambda)^2 \right)^{-1/2} d\theta.\]
		Substituting this expansion into the composition formula and integrating by parts, we obtain
		\[\sigma (V^{-1} k \partial_t (-\Delta_t)^{-1/2} V) = U' \frac{k^2 \eta}{(k^2 + \eta^2)^{3/2}} + r_\Delta(y, \eta),\]
		where, for all nonnegative \(j\) and \(l\) up to the finite order required below,
		\begin{equation}
			\label{r_Delta}
			|\partial_y^j \partial_\eta^l r_\Delta(y,\eta)| \lesssim \epsilon \left\langle \frac{\eta}{k} \right\rangle^{-3-l}.
		\end{equation}
		Conjugating by \(TS\) produces
		\begin{align*}
			\sigma (V^{-1} TS k \partial_t (-\Delta_t)^{-1/2} (TS)^{-1} V) =& U' \frac{k^2 \eta}{(k^2 + \eta^2)^{3/2}} + \sigma (V^{-1} [TS,U'] (TS)^{-1} V) \frac{k^2 \eta}{(k^2 + \eta^2)^{3/2}} \\
			&+ \sigma (V^{-1} TS V) \# r_\Delta(y,\eta) \# \sigma (V^{-1} (TS)^{-1} V).
		\end{align*}
		We denote the sum of the last two terms by \(r_{\Delta, TS}\). Combining the bounds from Lemma \ref{lem TS_commutator}, estimate \eqref{r_Delta}, and Claim \ref{cla sinh}, we obtain
		\[|\partial_y^j \partial_\eta^l r_{\Delta, TS}| \lesssim \epsilon \left\langle \frac{\eta}{k} \right\rangle^{-3-l}.\]
		
		After normalization by \((Q^*)^{-1}\) and \(Q^{-1}\), the resulting symbol is
		\begin{align*}
			&\sigma (V^{-1} (Q^*)^{-1} TS k \partial_t (-\Delta_t)^{-1/2} (TS)^{-1} Q^{-1} V) \\
			=& \beta U'(y)^{-1/2} \# \left\langle \frac{\eta}{k} \right\rangle \# \left( U'(y) \frac{k^2\eta} {(k^2+\eta^2)^{3/2}} + r_{\Delta, TS} (y, \eta) \right) \# \left\langle \frac{\eta}{k} \right\rangle \# U'(y)^{-1/2}.
		\end{align*}
		The composition expansion \eqref{composition_expansion} identifies the principal symbol as \(\beta \frac{\eta}{\sqrt{k^2 + \eta^2}}\). The corresponding Fourier multiplier has \(L^2\)-operator norm \(\beta\). Moreover, the exact composition formula \eqref{composition_formula}, the preceding estimate for \(r_{\Delta,TS}\), and the smallness of the derivatives of \(U'\) show that the difference between the full symbol and its principal part is bounded by \(O(\epsilon)\). The same bound holds for all its derivatives up to the order required in the Calder\'on--Vaillancourt Theorem \ref{thm CV}. This completes the proof.
	\end{proof}
	
	\subsection{A comparison principle}
	
	We now record a comparison property for the Fourier multipliers appearing in the weighted energy estimates. The following comparison principle allows a multiplier to be replaced by a pointwise larger one in the weighted \(L^2\)-norm, up to \(O(\epsilon)\) errors.
	
	\begin{lemma}
		\label{lem comparison}
		Let
		\[\mathcal{M} = V \op (m(\eta)) V^{-1}, \qquad \mathcal{N} = V \op (n(\eta)) V^{-1},\]
		where \(m\) and \(n\) are symbols, depending only on \(\eta\), satisfying
		\[|m(\eta)| \le |n(\eta)| \ne 0,\]
		and for \(l \in \{0, 1, \cdots, 4\}\),
		\[\left| \partial_\eta^l \frac{m}{n} (\eta) \right| \lesssim \left\langle \frac{\eta}{k} \right\rangle^{-l}.\]
		Then,
		\[\| P \mathcal{M} f \|_{L^2}^2 \le \| P \mathcal{N} f \|_{L^2}^2 + O(\epsilon) \| Q \mathcal{N} f \|_{L^2}^2, \qquad \| Q \mathcal{M} f \|_{L^2} \le (1 + O(\epsilon)) \| Q \mathcal{N} f \|_{L^2},\]
		for every \(f\) such that the above quantities are finite.
	\end{lemma}
	
	\begin{proof}
		We first prove the estimate involving \(P\). By the assumption that \(|n| \ne 0\), the inverse \(\mathcal{N}^{-1}\) is well-defined on the range of \(\mathcal{N}\). Hence, we have the following identity:
		\[\| P \mathcal{M} f \|_{L^2}^2 - \| P \mathcal{N} f\|_{L^2}^2 = \left\langle [(\mathcal{M} \mathcal{N} ^{-1})^* P^* P (\mathcal{M} \mathcal{N}^{-1}) - P^* P] \mathcal{N} f, \mathcal{N} f \right\rangle.\]
		Using the definition of \(P\), conjugating by \(V\), and observing that Fourier multipliers commute, we find that the symbol of the operator in brackets is given by
		\[\sigma (V^{-1} \mathcal{R}_{m,n} V) = \frac{1}{\gamma} \left\langle \frac{\eta}{k} \right\rangle^{-1/2} \# \left( \frac{\overline{m}}{\overline{n}} (\eta) \# U'(y) \# \frac{m}{n} (\eta) -U'(y) \right) \# \left\langle \frac{\eta}{k} \right\rangle^{-1/2},\]
		where, for simplicity, we have denoted \(\mathcal{R}_{m,n} = (\mathcal{M} \mathcal{N} ^{-1})^* P^* P (\mathcal{M} \mathcal{N}^{-1}) - P^* P\).
		
		The expression in parentheses admits the decomposition
		\begin{align*}
			\frac{\overline{m}}{\overline{n}} (\eta) \# U'(y) \# \frac{m}{n} (\eta) - U'(y) =& - U'(y)^{1/2} \# \left( 1 - \frac{|m|^2}{|n|^2} (\eta) \right) \# U'(y)^{1/2} \\
			&- \left[ U'(y), \frac{\overline{m}}{\overline{n}} (\eta) \right] \# \frac{m}{n} (\eta) + U'(y)^{1/2} \# \left[ U'(y)^{1/2}, \frac{|m|^2}{|n|^2} (\eta) \right].
		\end{align*}
		The first term on the right-hand side is self-adjoint and non-positive, since \(|m| \le |n|\). It therefore remains to estimate the last two terms. We denote their corresponding contributions to \(\mathcal{R}_{m,n}\) by \(\mathcal{R}_{m,n,1}\) and \(\mathcal{R}_{m,n,2}\), respectively. More precisely,
		\[\mathcal{R}_{m,n,1} \coloneqq - V \op \left( \frac{1}{\gamma} \left\langle \frac{\eta}{k} \right\rangle^{-1/2} \# \left[ U'(y), \frac{\overline{m}}{\overline{n}} (\eta) \right] \# \frac{m}{n} (\eta) \# \left\langle \frac{\eta}{k} \right\rangle^{-1/2} \right) V^{-1},\]
		\[\mathcal{R}_{m,n,2} \coloneqq V \op \left( \frac{1}{\gamma} \left\langle \frac{\eta}{k} \right\rangle^{-1/2} \# U'(y)^{1/2} \# \left[ U'(y)^{1/2}, \frac{|m|^2}{|n|^2} (\eta) \right] \# \left\langle \frac{\eta}{k} \right\rangle^{-1/2} \right) V^{-1}.\]
		
		We first consider \(\mathcal{R}_{m,n,1}\). To control this term by the \(Q\)-weighted norm, we estimate the normalized operator \((Q^{-1})^* \mathcal{R}_{m,n,1} Q^{-1}\), whose symbol is given by
		\[\sigma (V^{-1} (Q^{-1})^* \mathcal{R}_{m,n,1} Q^{-1} V) = - \frac{\beta}{\gamma} (U'(y))^{-1/2} \# \left\langle \frac{\eta}{k} \right\rangle^{1/2} \# \left[ U'(y), \frac{\overline{m}}{\overline{n}} (\eta) \right] \# \left\langle \frac{\eta}{k} \right\rangle^{1/2} \# \frac{m}{n} (\eta) \# (U'(y))^{-1/2}.\]
		Here we have used the fact that \(\frac{\overline{m}}{\overline{n}}\) and \(\left\langle \frac{\eta}{k} \right\rangle^{1/2}\) are Fourier multipliers and therefore commute. Since \(U'\) is bounded from above and below and \(\frac{\overline{m}}{\overline{n}}\) is uniformly bounded, it suffices to estimate the three factors in the middle. By the composition formula \eqref{composition_formula} and Taylor's theorem, we obtain
		\begin{align*}
			\left[ U'(y), \frac{\overline{m}}{\overline{n}} (\eta) \right] &= - \frac{1}{2\pi} \int_{\mathbb{R}} \int_{\mathbb{R}} \int_{0}^{1} e^{-iz\lambda} \lambda \partial_\eta \frac{\overline{m}}{\overline{n}} (\eta + \theta \lambda) U'(y+z) d \theta d \lambda dz \\
			&= \frac{i}{2\pi} \int_{\mathbb{R}} \int_{\mathbb{R}} \int_{0}^{1} e^{-iz\lambda} \langle z \rangle^{-2} (1 - \partial_\lambda^2) \langle \lambda \rangle^{-6} (1 - \partial_z^2)^3 \partial_\eta \frac{\overline{m}}{\overline{n}} (\eta + \theta \lambda) U''(y+z) d \theta d \lambda dz,
		\end{align*}
		where we have integrated by parts in the oscillatory variables \(z\) and \(\lambda\). The assumptions on the derivatives of \(\frac{m}{n}\) then imply
		\[\left| \partial_y^j \partial_\eta^l \left[ U'(y), \frac{\overline{m}}{\overline{n}} (\eta) \right] \right| \lesssim \epsilon \left\langle \frac{\eta}{k} \right\rangle^{-1 - l}, \qquad j \in \{0, 1, 2, 3\}, l \in \{0, 1\}.\]
		Applying the composition formula \eqref{composition_formula} once more to the surrounding factors and integrating by parts in the oscillatory variables, we have
		\begin{align*}
			\left| \left\langle \frac{\eta}{k} \right\rangle^{1/2} \# \left[ U'(y), \frac{\overline{m}}{\overline{n}} (\eta) \right] \# \left\langle \frac{\eta}{k} \right\rangle^{1/2} \right| &= \left| \frac{1}{2\pi} \int_{\mathbb{R}} \int_{\mathbb{R}} e^{-iz\lambda} \left\langle \frac{\eta + \lambda}{k} \right\rangle^{1/2} \left[ U'(y + z), \frac{\overline{m}}{\overline{n}} (\eta) \right] \left\langle \frac{\eta}{k} \right\rangle^{1/2} d \lambda dz \right| \\
			&\lesssim \int_{\mathbb{R}} \int_{\mathbb{R}} \langle z \rangle^{-2} \langle \lambda \rangle^{-2} \left\langle \frac{\eta}{k} \right\rangle^{1/2} \langle \lambda \rangle^{1/2} \epsilon \left\langle \frac{\eta}{k} \right\rangle^{-1} \left\langle \frac{\eta}{k} \right\rangle^{1/2} d \lambda dz \\
			&\lesssim \epsilon.
		\end{align*}
		The same estimate holds for the derivatives of the symbol with respect to \(y\) and \(\eta\). Thus, using the Calder\'on--Vaillancourt Theorem \ref{thm CV}, we obtain the desired operator bound for \(\mathcal{R}_{m,n,1}\),
		\[\| (Q^{-1})^* \mathcal{R}_{m,n,1} Q^{-1} \|_{L^2 \to L^2} \lesssim \epsilon.\]
		
		The term \(\mathcal{R}_{m,n,2}\) is estimated in the same way. Combining these estimates with the nonpositivity of the first term in the decomposition gives the desired estimate for \(P\).
		
		The estimate involving \(Q\) is simpler. Indeed, it suffices to prove
		\[\| Q \mathcal{M} \mathcal{N}^{-1} Q^{-1} \|_{L^2 \to L^2} \le 1 + O(\epsilon).\]
		Conjugating by the unitary operator \(V\), we compute the symbol of this operator,
		\[\sigma (V^{-1} Q \mathcal{M} \mathcal{N}^{-1} Q^{-1} V) = U'(y)^{1/2} \# \frac{m(\eta)}{n(\eta)} \# U'(y)^{-1/2} = \frac{m(\eta)}{n(\eta)} + U'(y)^{1/2} \left[ \frac{m(\eta)}{n(\eta)}, U'(y)^{-1/2} \right].\]
		The first term is a Fourier multiplier and satisfies
		\[\left\| \op \left( \frac{m}{n} \right) \right\|_{L^2 \to L^2} = \left\| \frac{m}{n} \right\|_{L^\infty} \le 1.\]
		Hence, it remains to estimate the commutator term. By the same argument as above, we have
		\[\left| \partial_y^j \partial_\eta^l \left[ \frac{m(\eta)}{n(\eta)}, U'(y)^{-1/2} \right] \right| \lesssim \epsilon, \qquad j, l \in \{0, 1\}.\]
		Thus, by the Calder\'on--Vaillancourt Theorem \ref{thm CV}, the lemma follows.
	\end{proof}
	
	\subsection{Auxiliary bounds for the cross term}
	
	We conclude this section with several auxiliary operator bounds needed to estimate the terms arising from differentiating the cross term in the basic energy.
	
	\begin{lemma}
		\label{lem P_Q_bound}
		The following operator bounds hold:
		\begin{equation}
			\label{P_Q_bound1}
			\| P^*P Q^{-1} \|_{L^2 \to L^2} \le (1 + O(\epsilon)) \frac{\sqrt{\beta U'_{\sup}}}{\gamma},
		\end{equation}
		\begin{equation}
			\label{P_Q_bound2}
			\| Q P^*P (U')^{-1} k^{-1} (- \Delta_t)^{1/2} Q^{-1} \|_{L^2 \to L^2} \le \frac{1}{\gamma} + O(\epsilon),
		\end{equation}
		and
		\begin{equation}
			\label{P_Q_bound3}
			\| (Q^{-1})^* k^{-1} (- \Delta_t)^{1/2} (U')^{-1} Q^*Q P^*P Q^{-1} \|_{L^2 \to L^2} \le \frac{1}{\gamma} + O(\epsilon).
		\end{equation}
	\end{lemma}
	
	\begin{proof}
		Using \eqref{P^*P}, we obtain
		\[\sigma (V^{-1} P^*P Q^{-1} V) = \frac{\sqrt{\beta U'(y)}}{\gamma} + \beta^{1/2} r_P(y, \eta) \# \left\langle \frac{\eta}{k} \right\rangle \# U'(y)^{-1/2},\]
		where \(r_P\) satisfies \eqref{r_P}. The symbol bounds for \(r_P\), together with the Calder\'on--Vaillancourt Theorem \ref{thm CV}, yield \eqref{P_Q_bound1}.
		
		We next prove \eqref{P_Q_bound2}. The symbol of the corresponding operator is
		\begin{align*}
			&\sigma (V^{-1} Q P^*P (U')^{-1} k^{-1} (- \Delta_t)^{1/2} Q^{-1} V)\\
			=&~ \frac{\sgn k}{\gamma} U'(y)^{1/2} \# \left\langle \frac{\eta}{k} \right\rangle^{-3/2} \# U'(y) \# \left\langle \frac{\eta}{k} \right\rangle^{-1/2} \# U'(y)^{-1} \# \left\langle \frac{\eta}{k} \right\rangle^2 \# U'(y)^{-1/2} \\
			=& ~\frac{\sgn k}{\gamma} + \frac{\sgn k}{\gamma} U'(y)^{1/2} \# \left\langle \frac{\eta}{k} \right\rangle^{-3/2} \\
			&~\# \left[ U'(y), \left\langle \frac{\eta}{k} \right\rangle^{-1/2} \right] \# U'(y)^{-1} \# \left\langle \frac{\eta}{k} \right\rangle^2 \# U'(y)^{-1/2}.
		\end{align*}
		The first term has \(L^2\)-operator norm \(\gamma^{-1}\). Arguing as in the proof of Lemma \ref{lem comparison}, we find that the remaining term has \(L^2\)-operator norm bounded by \(O(\epsilon)\). This proves \eqref{P_Q_bound2}.
		
		Finally, \eqref{P_Q_bound3} follows from the same symbolic argument used to prove \eqref{P_Q_bound2}.
	\end{proof}
	
	To estimate the contribution arising when the time derivative falls on \(P^*P\), we establish the following operator bound.
	
	\begin{lemma}
		\label{lem partial_tP^*P}
		The following operator bound holds:
		\[\| (Q^{-1})^*\partial_t (P^*P) Q^{-1} \|_{L^2 \to L^2} \le \left (\frac{\beta}{\gamma} + O(\epsilon) \right) U'_{\sup}.\]
	\end{lemma}
	
	\begin{proof}
		Recalling the decomposition of \(P^*P\) in \eqref{P^*P}, we have
		\[\partial_t (P^*P) = \partial_t \left( V \op \left( \frac{1}{\gamma} U'(y) \left\langle \frac{\eta}{k} \right\rangle^{-1} + r_P (y, \eta) \right) V^{-1} \right) = ik V \left[ U, \op \left( \frac{1}{\gamma} U'(y) \left\langle \frac{\eta}{k} \right\rangle^{-1} + r_P (y, \eta) \right) \right] V^{-1}.\]
		By the composition formula \eqref{composition_formula} and Taylor's theorem,
		\begin{align*}
			\left[ U(y), \frac{1}{\gamma} U'(y) \left\langle \frac{\eta}{k} \right\rangle^{-1} + r_P (y, \eta) \right] =& - \frac{iU'(y)^2}{\gamma} \frac{\eta}{k^2} \left\langle \frac{\eta}{k} \right\rangle^{-3} \\
			&- \frac{U'(y)}{2 \pi} \int_{\mathbb{R}} \int_{\mathbb{R}} \int_{0}^{1} e^{-i z \lambda} \lambda^2 (1 - \theta) \partial_\eta^2 \left( \left\langle \frac{\eta + \theta \lambda}{k} \right\rangle^{-1} \right) U(y + z) d \theta d \lambda dz \\
			&- \frac{1}{2\pi} \int_{\mathbb{R}} \int_{\mathbb{R}} \int_{0}^{1} e^{-iz\lambda} \lambda \partial_\eta r_P (y, \eta + \theta \lambda) U(y+z) d \theta d \lambda dz.
		\end{align*}
		By integration by parts in the oscillatory variables and the bounds for \(r_P\) in \eqref{r_P}, we find that, after multiplication by \(ik\) and conjugation by \(Q^{-1}\), the two non-principal terms, together with all their derivatives up to the order required by Theorem \ref{thm CV}, are bounded by \(O(\epsilon)\). It therefore remains to consider the principal contribution. Its symbol after conjugation by \(Q^{-1}\) is
		\[\sigma \left( V^{-1} (Q^{-1})^* ik \op \left( - \frac{iU'(y)^2}{\gamma} \frac{\eta}{k^2} \left\langle \frac{\eta}{k} \right\rangle^{-3} \right) Q^{-1} V \right) = \beta U'(y)^{-1/2} \# \left\langle \frac{\eta}{k} \right\rangle \# \frac{U'(y)^2}{\gamma} \# \frac{\eta}{k} \left\langle \frac{\eta}{k} \right\rangle^{-2} \# U'(y)^{-1/2}.\]
		By the asymptotic expansion \eqref{composition_expansion}, the principal symbol of this expression is \(\frac{\beta}{\gamma} U'(y) \frac{\eta}{k} \langle \frac{\eta}{k} \rangle^{-1}\), whose \(L^\infty\) norm is bounded by \(\frac{\beta}{\gamma} U'_{\sup}\). The remaining terms in the composition satisfy the same bound up to a factor of \(O(\epsilon)\). The Calder\'on--Vaillancourt Theorem \ref{thm CV} yields the result.
	\end{proof}
	
	\section{Proof of the basic energy estimate}
	\label{sec5}
	
	We now use the lemmas established above to prove the basic energy estimate stated in Proposition \ref{prop basic_energy}.
	
	\begin{proof}[Proof of Proposition \ref{prop basic_energy}]
		We divide the proof into five steps. 
		
		\noindent\textbf{Step 1. Equivalence of the energy functional.} We first prove the equivalence of the energy functional, namely \eqref{equivalent_energy}. Since \(T\) and \(S\) are unitarily conjugated Fourier multipliers, they commute with \(D_t\) and \(\Delta_t\). Hence,
		\[\| TSkR \|_{L^2}^2 + \| TSD_tR \|_{L^2}^2 = \| TS \sqrt{- \Delta_t} R \|_{L^2}^2.\]
		
		We next estimate the cross term in \(E\). Recalling that \(\mathcal{R}_P=V\op(r_P)V^{-1}\), the decomposition of \(P^*P\) in \eqref{P^*P} gives
		\begin{align*}
			\frac{1}{3} \re \langle PTSA, PTSR \rangle =& ~\frac{1}{3} \re \left\langle \left( \frac{1}{\gamma} U' |k| (- \Delta_t)^{-1/2} + \mathcal{R}_P \right) TSA, TSR \right\rangle \\
			=& ~\frac{1}{2} \re \langle TSA, |k| (- \Delta_t)^{-1/2} U' TSR \rangle + \frac{1}{3} \re \langle \mathcal{R}_P TSA, TSR \rangle \\
			=& ~\frac{1}{2} \re \langle TSA, TS |k| (- \Delta_t)^{-1/2} U' R \rangle + \frac{1}{2} \re \langle TSA, |k| (- \Delta_t)^{-1/2} [U', TS] R \rangle \\
			&+ \frac{1}{3} \re \langle \mathcal{R}_P TSA, TSR \rangle,
		\end{align*}
		where we used \(\gamma = \frac{2}{3}\). Recall that
		\begin{equation}
			\label{U'R}
			k (- \Delta_t)^{-1/2} U'R = k (- \Delta_t)^{-1/2} \Theta - \Xi.
		\end{equation}
		Combining \eqref{U'R} with the bound \(\| k (- \Delta_t)^{-1/2} \|_{L^2 \to L^2} \le 1\), we obtain the following estimate for the principal term,
		\[\left| \frac{1}{2} \re \langle TSA, TS |k| (- \Delta_t)^{-1/2} U' R \rangle \right| \le \frac{1}{2} \| TSA \|_{L^2} \| TS \Theta \|_{L^2} + \frac{1}{2} \| TSA \|_{L^2} \| TS \Xi \|_{L^2}.\]
		Applying Young's inequality, 
		\[\left| \frac{1}{2} \re \langle TSA, TS |k| (- \Delta_t)^{-1/2} U' R \rangle \right| \le \frac{1}{3} \| TSA \|_{L^2}^2 + \frac{3}{8} \| TS \Theta \|_{L^2}^2 + \frac{3}{8} \| TS \Xi \|_{L^2}^2.\]
		
		For the remaining two terms, Lemma \ref{lem TS_commutator}, the estimate \eqref{r_P} and the Calder\'on--Vaillancourt Theorem \ref{thm CV} give
		\[\left| \frac{1}{2} \re \langle TSA, |k| (- \Delta_t)^{-1/2} [U', TS] R \rangle + \frac{1}{3} \re \langle \mathcal{R}_P TSA, TSR \rangle \right| \lesssim \epsilon \| TSA \|_{L^2}^2 + \frac{\epsilon}{M^2} \| TSkR \|_{L^2}^2,\]
		where the factor \(M^{-2}\) is obtained at the expense of the \(M\)-dependent implicit constant in \(\lesssim\). Taking \(\epsilon = \epsilon(U'_{\inf}, U'_{\sup}, M)\) sufficiently small, the two terms are controlled by the positive part of \(E\). Hence, we obtain the equivalence \eqref{equivalent_energy}.
		
		\noindent\textbf{Step 2. The differentiated energy identity and the main dissipation.} To prove \eqref{derivative}, we differentiate the energy functional with respect to time. Using the equations \eqref{R_A_Theta}, we obtain
		\begin{equation}
			\label{dE/dt}
			\begin{aligned}
				\frac{d}{dt} E =& ~\frac{1}{M^2} \re \langle \partial_t (TS) kR, TSkR \rangle + \frac{1}{M^2} \re \langle \partial_t (TS)D_tR, TSD_tR \rangle + \re \langle \partial_t (TS) A, TSA \rangle \\
				&+ \re \langle \partial_t (TS) \Theta, TS \Theta \rangle + \re \langle \partial_t(TS) \Xi, TS \Xi \rangle - \frac{1}{3} \re \langle PTSA, PTSA \rangle \\
				&- \frac{1}{M^2} \re \langle TS ikU'R, TSD_tR \rangle - \re \langle TS 2ikU'D_t \Delta_t^{-1} A, TSA \rangle \\
				&- \re \langle TS 2k U' (- \Delta_t)^{-1/2} \Xi, TSA \rangle - \re \langle TS 2k^2U' \Delta_t^{-1} U'' D_t \Delta_t^{-1} R, TSA \rangle \\
				&+ \re \langle TS ikU'' \Delta_t^{-1} \Theta, TS \Theta \rangle + \re \langle TS 2ik^3 U'' \Delta_t^{-1} U' \Delta_t^{-1} R, TS \Theta \rangle \\
				&+ \re \langle TS [\partial_t (- \Delta_t)^{-1/2}] (- \Delta_t)^{1/2} \Xi, TS \Xi \rangle + \re \langle TS k (- \Delta_t)^{-1/2} U' A, TS \Xi \rangle \\
				&+ \re \langle TS k (- \Delta_t)^{-1/2} ik U'' \Delta_t^{-1} \Theta, TS \Xi \rangle + \re \langle TS k (- \Delta_t)^{-1/2} 2 ik^3 U'' \Delta_t^{-1} U' \Delta_t^{-1} R, TS \Xi \rangle \\
				&+ \frac{1}{3} \re \langle P \partial_t (TS) A, PTSR \rangle + \frac{1}{3} \re \langle PTSA, P \partial_t (TS) R \rangle + \frac{1}{3} \re \langle \partial_t (P^*P) TSA, TSR \rangle \\
				&- \frac{1}{3M^2} \re \langle PTS \Delta_t R, PTSR \rangle - \frac{1}{3} \re \langle PTS 2ikU'D_t\Delta_t^{-1} A, PTSR \rangle \\
				&- \frac{1}{3} \re \langle PTS 2k U' (- \Delta_t)^{-1/2} \Xi, PTSR \rangle - \frac{1}{3} \re \langle PTS 2k^2U' \Delta_t^{-1} U'' D_t \Delta_t^{-1} R, PTSR \rangle.
			\end{aligned}
		\end{equation}
		Here we have used the evolution identity
		\[\partial_t \Xi = [\partial_t (-\Delta_t)^{-1/2}] (-\Delta_t)^{1/2} \Xi + k (-\Delta_t)^{-1/2} (U'A + ikU'' \Delta_t^{-1} \Theta + 2ik^3U'' \Delta_t^{-1} U' \Delta_t^{-1} R),\]
		which follows from the definition of \(\Xi\). We have also used the identity
		\[- \re \langle TS \Delta_t R, TSA \rangle = \re \langle TSkA, TSkR \rangle + \re \langle TSD_tA, TSD_tR \rangle,\]
		since \(T\) and \(S\) both commute with \(D_t\) and \(\Delta_t\).
		
		We first isolate the terms in which the time derivative falls on the weights \(TS\) in the positive part of the energy. By Lemma \ref{lem P_Q_TS}, these terms provide the main negative contribution.
		
		\noindent\textbf{Step 3. Remaining contributions from the non-cross part of the energy.} We decompose
		\[- \frac{1}{M^2} \re \langle TS ikU'R, TSD_tR \rangle = - \frac{1}{M^2} \re \langle ikU' TSR, TSD_tR \rangle + \frac{1}{M^2} \re \langle ik [U', TS] R, TSD_tR \rangle.\]
		For the first term, Corollary \ref{cor operator_bounds} gives
		\begin{align*}
			- \frac{1}{M^2} \re \langle ikU' TSR, TSD_tR \rangle &= - \frac{1}{M^2} \re \langle ikU' (- \Delta_t)^{-1/2} TS \sqrt{- \Delta_t} R, TSD_tR \rangle \\
			&\le \frac{\gamma}{M^2} \| P TS \sqrt{- \Delta_t} R \|_{L^2} \| P TSD_tR \|_{L^2} + \frac{O(\epsilon)}{M^2} \| Q TS \sqrt{- \Delta_t} R \|_{L^2} \| Q TSD_tR \|_{L^2}.
		\end{align*}
		Since
		\[| \sigma (V^{-1} \sqrt{- \Delta_t} V) | = | \sigma (V^{-1} (k \pm D_t) V) | = \sqrt{k^2 + \eta^2} > 0,\]
		Lemma \ref{lem comparison} yields
		\[\| PTS \sqrt{- \Delta_t} R \|_{L^2}^2 \le \| PTS (k \pm D_t) R \|_{L^2}^2 + O(\epsilon) \| QTS (k \pm D_t) R \|_{L^2}^2.\]
		Adding the estimates corresponding to \(+\) and \(-\), and using the parallelogram identity, we obtain
		\begin{equation}
			\label{Pythagoras_P}
			\begin{aligned}
				2 \| PTS \sqrt{- \Delta_t} R \|_{L^2}^2 \le& (\| PTS (k + D_t) R \|_{L^2}^2 + \| PTS (k - D_t) R \|_{L^2}^2) \\
				&+ O(\epsilon) (\| QTS (k + D_t) R \|_{L^2}^2 + \| QTS (k - D_t) R \|_{L^2}^2)\\
				=& 2 (\| PTSkR \|_{L^2}^2 + \| PTSD_tR \|_{L^2}^2) + 2 O(\epsilon) (\| QTSkR \|_{L^2}^2 + \| QTSD_tR \|_{L^2}^2).
			\end{aligned}
		\end{equation}
		Similarly,
		\begin{equation}
			\label{Pythagoras_Q}
			\| QTS \sqrt{- \Delta_t} R \|_{L^2}^2 \le (1 + O(\epsilon)) (\| QTSkR \|_{L^2}^2 + \| QTSD_tR \|_{L^2}^2).
		\end{equation}
		Using Young's inequality in the preceding bilinear estimate, we obtain
		\begin{align*}
			- \frac{1}{M^2} \re \langle ikU' TSR, TSD_tR \rangle \le& \frac{\gamma}{2M^2} (\| PTSkR \|_{L^2}^2 + \| PTSD_tR \|_{L^2}^2 + O(\epsilon) \| QTSkR \|_{L^2}^2 + O(\epsilon) \| QTSD_tR \|_{L^2}^2) \\
			&+ \frac{\gamma}{2M^2} \| PTSD_tR \|_{L^2}^2 + \frac{O(\epsilon)}{M^2} \| QTSkR \|_{L^2}^2 + \frac{O(\epsilon)}{M^2} \| QTSD_tR \|_{L^2}^2 \\
			\le& \frac{\gamma}{2M^2} \| PTSkR \|_{L^2}^2 + \frac{\gamma}{M^2} \| PTSD_tR \|_{L^2}^2 + \frac{O(\epsilon)}{M^2} \| QTSkR \|_{L^2}^2 + \frac{O(\epsilon)}{M^2} \| QTSD_tR \|_{L^2}^2.
		\end{align*}
		For the commutator term, the Cauchy--Schwarz inequality and Corollary \ref{cor TS_commutator_bound} give
		\begin{align*}
			\frac{1}{M^2} \re \langle ik [U', TS] R, TSD_tR \rangle &\le \frac{1}{M^2} \| (Q^*)^{-1} [U', TS] (TS)^{-1} k (- \Delta_t)^{-1/2} Q^{-1} \|_{L^2 \to L^2} \| QTS \sqrt{- \Delta_t} R \|_{L^2} \| QTSD_tR \|_{L^2} \\
			&\lesssim \frac{\epsilon}{M^2} \| QTS \sqrt{- \Delta_t} R \|_{L^2} \| QTSD_tR \|_{L^2}.
		\end{align*}
		It follows from \eqref{Pythagoras_Q} and Young's inequality that
		\[\frac{1}{M^2} \re \langle ik [U', TS] R, TSD_tR \rangle \lesssim \frac{\epsilon}{M^2} \| QTSkR \|_{L^2}^2 + \frac{\epsilon}{M^2} \| QTSD_tR \|_{L^2}^2.\]
		Combining the preceding estimates, we conclude that
		\[- \frac{1}{M^2} \re \langle TSikU'R, TSD_tR \rangle \le \frac{\gamma}{2M^2} \| PTSkR \|_{L^2}^2 + \frac{\gamma}{M^2} \| PTSD_tR \|_{L^2}^2 + \frac{O(\epsilon)}{M^2} \| QTSkR \|_{L^2}^2 + \frac{O(\epsilon)}{M^2} \| QTSD_tR \|_{L^2}^2.\]
		
		The next term can be treated similarly but more directly. We decompose
		\[- \re \langle TS 2ikU'D_t\Delta_t^{-1} A, TSA \rangle = - \re \langle 2ik U' TS D_t\Delta_t^{-1} A, TSA \rangle + \re \langle 2ik [U', TS] D_t\Delta_t^{-1} A, TSA \rangle.\]
		By Corollary \ref{cor operator_bounds}, the first term is bounded by
		\begin{align*}
			- \re \langle 2ik U' TS D_t\Delta_t^{-1} A, TSA \rangle &= \re \langle 2ik U' (- \Delta_t)^{-1/2} TS D_t (- \Delta_t)^{-1/2} A, TSA \rangle \\
			&\le 2 \gamma \| PTS D_t (- \Delta_t)^{-1/2} A \|_{L^2} \| PTSA \|_{L^2} + O(\epsilon) \| QTS D_t (- \Delta_t)^{-1/2} A \|_{L^2} \| QTSA \|_{L^2}.
		\end{align*}
		The symbol of \(D_t (-\Delta_t)^{-1/2}\) satisfies
		\[|\sigma (V^{-1} D_t (- \Delta_t)^{-1/2} V)| = \frac{|\eta|}{\sqrt{k^2 + \eta^2}} \le 1,\]
		and its derivatives satisfy the assumptions of Lemma \ref{lem comparison}. Therefore,
		\[\| PTS D_t (- \Delta_t)^{-1/2} A \|_{L^2}^2 \le \| PTSA \|_{L^2}^2 + O(\epsilon) \| QTSA \|_{L^2}^2,\]
		and
		\[\| QTS D_t (- \Delta_t)^{-1/2} A \|_{L^2} \le (1 + O(\epsilon)) \| QTSA \|_{L^2}.\]
		Thus, Young's inequality yields
		\[- \re \langle 2ik U' TS D_t\Delta_t^{-1} A, TSA \rangle \le 2\gamma \| PTSA \|_{L^2}^2 + O(\epsilon) \| QTSA \|_{L^2}^2.\]
		On the other hand, Corollary \ref{cor TS_commutator_bound} and the Cauchy--Schwarz inequality give
		\[\re \langle 2ik [U', TS] D_t\Delta_t^{-1} A, TSA \rangle \lesssim \epsilon \| QTS D_t (- \Delta_t)^{-1/2} A \|_{L^2} \| QTSA \|_{L^2} \lesssim \epsilon \| QTSA \|_{L^2}^2.\]
		Consequently,
		\[- \re \langle TS 2ikU' D_t \Delta_t^{-1} A, TSA \rangle \le 2\gamma \| PTSA \|_{L^2}^2 + O(\epsilon) \| QTSA \|_{L^2}^2.\]
		
		We next decompose
		\[- \re \langle TS 2kU' (- \Delta_t)^{-1/2} \Xi, TSA \rangle = \re \langle 2k^2 U' \Delta_t^{-1} TS k^{-1} (- \Delta_t)^{1/2} \Xi, TSA \rangle + \re \langle 2k [U', TS] (- \Delta_t)^{-1/2} \Xi, TSA \rangle.\]
		Corollary \ref{cor operator_bounds} gives the following bound for the principal term,
		\[\re \langle 2k^2 U' \Delta_t^{-1} TS k^{-1} (- \Delta_t)^{1/2} \Xi, TSA \rangle \le 2 (\beta + O(\epsilon)) \| QTS k^{-1} (- \Delta_t)^{1/2} \Xi \|_{L^2} \| QTSA \|_{L^2}.\]
		Using the identity
		\[k^{-1} (- \Delta_t)^{1/2} \Xi = \Theta - U'R,\]
		together with Lemma \ref{lem U'_weighted_bound}, we have
		\begin{equation}
			\label{QTS_Xi_bound}
			\begin{aligned}
				\| QTS k^{-1} (- \Delta_t)^{1/2} \Xi \|_{L^2} &\le \| QTS \Theta \|_{L^2} + (1 + O(\epsilon)) U'_{\sup} \| QTSkR \|_{L^2} \\
				&\le \| QTS \Theta \|_{L^2} + (MU'_{\sup} + O(\epsilon)) \frac{1}{M} \| QTSkR \|_{L^2}.
			\end{aligned}
		\end{equation}
		Moreover, for the remaining commutator, Corollary \ref{cor TS_commutator_bound} implies
		\[\re \langle 2k [U', TS] (- \Delta_t)^{-1/2} \Xi, TSA \rangle \lesssim \epsilon \| Q TS \Xi \|_{L^2} \| QTSA \|_{L^2}.\]
		Therefore, Young's inequality yields
		\begin{align*}
			- \re \langle TS 2kU' (- \Delta_t)^{-1/2} \Xi, TSA \rangle \le& (\beta + O(\epsilon)) \| QTS \Theta \|_{L^2}^2 + (\beta MU'_{\sup} + O(\epsilon)) \frac{1}{M^2} \| QTSkR \|_{L^2}^2 \\
			&+ (\beta (1 + MU'_{\sup}) + O(\epsilon)) \| QTSA \|_{L^2}^2 + O(\epsilon) \| Q TS \Xi \|_{L^2}^2.
		\end{align*}
		
		The subsequent term is treated similarly. Replacing \(\Theta-U'R\) with \(U''D_t\Delta_t^{-1}R\) in the preceding argument and applying Lemmas \ref{lem U'_weighted_bound} and \ref{lem comparison}, we obtain
		\begin{align*}
			- \re \langle TS 2k^2U' \Delta_t^{-1} U'' D_t \Delta_t^{-1} R, TSA \rangle &\le 2 (\beta + O(\epsilon)) \| QTS U'' D_t \Delta_t^{-1} R \|_{L^2} \| QTSA \|_{L^2} \\
			&\lesssim \frac{\epsilon}{M} \| QTS D_t \Delta_t^{-1} R \|_{L^2} \| QTSA \|_{L^2} \\
			&\lesssim \frac{\epsilon}{M^2} \| QTSkR \|_{L^2}^2 + \epsilon \| QTSA \|_{L^2}^2.
		\end{align*}
		In the second inequality, Lemma \ref{lem U'_weighted_bound} provides the factor \(\epsilon\). We have also used the allowed \(M\)-dependence of \(\lesssim\) to extract the factor \(M^{-1}\).
		
		We estimate the next two terms. By the composition formula \eqref{composition_formula}, Lemma \ref{lem TS_commutator}, and the Calder\'on--Vaillancourt Theorem \ref{thm CV}, the same symbolic argument as in the proof of Corollary \ref{cor TS_commutator_bound} gives
		\[\| (Q^*)^{-1} TS U'' (TS)^{-1} k^2 (-\Delta_t)^{-1} Q^{-1} \|_{L^2 \to L^2} \lesssim \epsilon.\]
		Consequently, we obtain
		\[\re \langle TSikU'' \Delta_t^{-1} \Theta, TS \Theta \rangle \lesssim \epsilon \| QTS \Theta \|_{L^2}^2,\]
		and
		\begin{align*}
			\re \langle TS 2ik^3 U'' \Delta_t^{-1} U' \Delta_t^{-1} R, TS \Theta\rangle \lesssim \epsilon \| QTS k U' \Delta_t^{-1} R \|_{L^2} \| QTS \Theta \|_{L^2} \lesssim \frac{\epsilon}{M^2} \| QTSkR \|_{L^2}^2 + \epsilon \| QTS \Theta \|_{L^2}^2.
		\end{align*}
		Here, the second estimate additionally requires Lemma \ref{lem U'_weighted_bound} and Lemma \ref{lem comparison}.
		
		For the term involving \(\partial_t(-\Delta_t)^{-1/2}\), Lemma \ref{lem partial_t_Delta} yields
		\[\re \langle TS [\partial_t (- \Delta_t)^{-1/2}] (- \Delta_t)^{1/2} \Xi, TS \Xi \rangle \le (\beta + O(\epsilon)) \| QTS k^{-1} (- \Delta_t)^{1/2} \Xi \|_{L^2} \| QTS \Xi \|_{L^2}.\]
		The first factor is controlled by \eqref{QTS_Xi_bound}. Hence, Young's inequality gives
		\begin{align*}
			\re \langle TS [\partial_t (- \Delta_t)^{-1/2}] (- \Delta_t)^{1/2} \Xi, TS \Xi \rangle \le& \left( \frac{\beta}{2} + O(\epsilon) \right) \| QTS \Theta \|_{L^2}^2 + \left( \frac{\beta MU'_{\sup}}{2} + O(\epsilon) \right) \frac{1}{M^2} \| QTSkR \|_{L^2}^2 \\
			&+ \left( \frac{\beta (1 + MU'_{\sup})}{2} + O(\epsilon) \right) \| QTS \Xi \|_{L^2}^2.
		\end{align*}
		
		The remaining three contributions arising from the \(\Xi\)-energy are handled by the same argument used for the term \(- \re \langle TS 2kU' (- \Delta_t)^{-1/2} \Xi, TSA \rangle\). More precisely,
		\begin{align*}
			\re \langle TS k (- \Delta_t)^{-1/2} U' A, TS \Xi \rangle \le& \left( \frac{\beta}{2} + O(\epsilon) \right) \| QTS \Theta \|_{L^2}^2 + \left( \frac{\beta MU'_{\sup}}{2} + O(\epsilon) \right) \frac{1}{M^2} \| QTSkR \|_{L^2}^2 \\
			&+ \left( \frac{\beta (1 + MU'_{\sup})}{2} + O(\epsilon) \right) \| QTSA \|_{L^2}^2 + O(\epsilon) \| QTS \Xi \|_{L^2}^2,
		\end{align*}
		\[\re \langle TS k (- \Delta_t)^{-1/2} ik U'' \Delta_t^{-1} \Theta, TS \Xi \rangle \lesssim \epsilon \| QTS \Theta \|_{L^2}^2 + \frac{\epsilon}{M^2} \| QTSkR \|_{L^2}^2 + \epsilon \| QTS \Xi \|_{L^2}^2,\]
		\[\re \langle TS k (- \Delta_t)^{-1/2} 2 ik^3 U'' \Delta_t^{-1} U' \Delta_t^{-1} R, TS \Xi \rangle \lesssim \epsilon \| QTS \Theta \|_{L^2}^2 + \frac{\epsilon}{M^2} \| QTSkR \|_{L^2}^2 + \epsilon \| QTS \Xi \|_{L^2}^2.\]
		
		\noindent\textbf{Step 4. Contributions from the cross term.} We now consider the contributions from the cross term, beginning with \(\frac{1}{3} \re\langle P \partial_t(TS) A, PTSR \rangle\). By the decomposition of \([\partial_t (TS)] (TS)^{-1}\) established in the proof of Lemma \ref{lem P_Q_TS},
		\[\re \langle P \partial_t (TS) A, PTSR \rangle \le \| P^*P TSA \|_{L^2} \| P^*P TSR \|_{L^2} + (1 + O(\epsilon)) \| Q TSA \|_{L^2} \| Q P^*P TS R \|_{L^2}.\]
		Applying estimate \eqref{P_Q_bound1} in Lemma \ref{lem P_Q_bound} to both factors, we obtain
		\[\| P^*P TSA \|_{L^2} \| P^*P TSR \|_{L^2} \le (1 + O(\epsilon))\frac{\beta MU'_{\sup}}{\gamma^2} \frac{1}{M} \| QTSA \|_{L^2} \| QTSkR \|_{L^2}.\]
		To control the second term, we invoke estimate \eqref{P_Q_bound2} in the same lemma, which yields
		\[\| Q P^*P TS R \|_{L^2} \le \left( \frac{1}{\gamma} + O(\epsilon) \right) \| Q k (- \Delta_t)^{-1/2} U' TSR \|_{L^2}.\]
		Combining Corollary \ref{cor TS_commutator_bound} with Lemma \ref{lem comparison}, we have
		\[\| Q k (- \Delta_t)^{-1/2} U' TSR \|_{L^2} \le \| Q k (- \Delta_t)^{-1/2} TS U' R \|_{L^2} + \frac{O(\epsilon)}{M^2} \| QTSkR \|_{L^2},\]
		where we have made the factor \(M^{-2}\) explicit by absorbing the compensating factor \(M^2\) into \(O(\epsilon)\). It remains to estimate the noncommutator term. Using \eqref{U'R} and applying Lemma \ref{lem comparison} once more, we obtain
		\[\| Q k (- \Delta_t)^{-1/2} TS U' R \|_{L^2} \le \| QTS k (- \Delta_t)^{-1/2} \Theta \|_{L^2} + \| QTS \Xi \|_{L^2} \le \| QTS \Theta \|_{L^2} + \| QTS \Xi \|_{L^2}.\]
		Applying Young's inequality, we conclude that
		\begin{align*}
			\frac{1}{3} \re \langle P \partial_t (TS) A, PTSR \rangle \le& \left( \frac{\beta MU'_{\sup}}{6 \gamma^2} + \frac{2}{9 \gamma} + O(\epsilon) \right) \| QTSA \|_{L^2}^2 + \left( \frac{\beta MU'_{\sup}}{6 \gamma^2} + O(\epsilon) \right) \frac{1}{M^2} \| QTSkR \|_{L^2}^2 \\
			&+ \left( \frac{1}{4 \gamma} + O(\epsilon) \right) \| QTS \Theta \|_{L^2}^2 + \left( \frac{1}{4 \gamma} + O(\epsilon) \right) \| QTS \Xi \|_{L^2}^2.
		\end{align*}
		
		For the companion contribution, the decomposition established in the proof of Lemma \ref{lem P_Q_TS} yields
		\[\re \langle PTSA, P \partial_t(TS) R \rangle \le - \re \langle P^*PTSA, P^*PTSR \rangle - \re \langle P^*PTSA, Q^*Q TSR \rangle + O(\epsilon) \| Q P^*P TSA \|_{L^2} \| Q TSR \|_{L^2}.\]
		Since \(\| Q \|_{L^2 \to L^2} \lesssim 1\), estimate \eqref{P_Q_bound1} implies \(\| Q P^*P TSA\|_{L^2} \lesssim \| QTSA \|_{L^2}\); hence, the remainder is absorbed into the \(O(\epsilon)\)-terms below. The first term is controlled exactly as above by estimate \eqref{P_Q_bound1} in Lemma \ref{lem P_Q_bound}. For the second term, estimate \eqref{P_Q_bound3} in the same lemma gives
		\[- \re \langle P^*PTSA, Q^*QTSR \rangle \le \left(\frac{1}{\gamma} + O(\epsilon)\right) \| QTSA \|_{L^2} \| Q k (-\Delta_t)^{-1/2} U' TSR \|_{L^2}.\]
		Proceeding as in the preceding estimate and applying Young's inequality, we obtain
		\begin{align*}
			\frac{1}{3} \re \langle PTSA, P \partial_t (TS) R \rangle \le& \left( \frac{\beta MU'_{\sup}}{6 \gamma^2} + \frac{2}{9 \gamma} + O(\epsilon) \right) \| QTSA \|_{L^2}^2 + \left( \frac{\beta MU'_{\sup}}{6 \gamma^2} + O(\epsilon) \right) \frac{1}{M^2} \| QTSkR \|_{L^2}^2 \\
			&+ \left( \frac{1}{4 \gamma} + O(\epsilon) \right) \| QTS \Theta \|_{L^2}^2 + \left( \frac{1}{4 \gamma} + O(\epsilon) \right) \| QTS \Xi \|_{L^2}^2.
		\end{align*}
		
		For the contribution arising from \(\partial_t(P^*P)\), applying the Cauchy--Schwarz inequality, Lemma \ref{lem partial_tP^*P}, and Young's inequality, we obtain
		\begin{align*}
			\frac{1}{3} \re \langle \partial_t (P^*P) TSA, TSR \rangle &\le \frac{1}{3} \| (Q^{-1})^* \partial_t (P^*P) Q^{-1} \|_{L^2 \to L^2} \| QTSA \|_{L^2} \| QTSkR \|_{L^2} \\
			&\le \frac{1}{3} \left( \frac{\beta MU'_{\sup}}{\gamma} + O(\epsilon) \right) \frac{1}{M} \| QTSA \|_{L^2} \| QTSkR \|_{L^2} \\
			&\le \frac{1}{6} \left( \frac{\beta MU'_{\sup}}{\gamma} + O(\epsilon) \right) \| QTSA \|_{L^2}^2 + \frac{1}{6M^2} \left( \frac{\beta MU'_{\sup}}{\gamma} + O(\epsilon) \right) \| QTSkR \|_{L^2}^2.
		\end{align*}
		
		From the definition \(\Delta_t = - k^2 + D_t^2\), the next term can be written as
		\[- \frac{1}{3M^2} \re \langle PTS \Delta_t R, PTSR \rangle = \frac{1}{3M^2} \| PTSkR \|_{L^2}^2 + \frac{1}{3M^2} \| PTSD_tR \|_{L^2}^2 + \frac{1}{3M^2} \re \langle TSD_tR, [D_t, P^*P]TSR \rangle.\]
		After conjugating by \(V\), the symbol of the commutator is given by
		\begin{align*}
			\sigma (V^{-1} [D_t,P^*P] V) &= \left[ i\eta, \frac{1}{\gamma} \left\langle \frac{\eta}{k} \right\rangle^{-1/2} \# U'(y) \# \left\langle \frac{\eta}{k} \right\rangle^{-1/2} \right] \\
			&= \frac{1}{\gamma} \left\langle \frac{\eta}{k} \right\rangle^{-1/2} \# \left[ i\eta, U'(y) \right] \# \left\langle \frac{\eta}{k} \right\rangle^{-1/2} \\
			&= \frac{1}{\gamma}
			\left\langle\frac{\eta}{k}\right\rangle^{-1/2}
			\#U''(y)\#
			\left\langle\frac{\eta}{k}\right\rangle^{-1/2}.
		\end{align*}
		Consequently,
		\[\sigma (V^{-1} (Q^*)^{-1} [D_t,P^*P] (-\Delta_t)^{-1/2} Q^{-1} V) = \frac{\beta}{\gamma|k|} U'(y)^{-1/2} \# \left\langle \frac{\eta}{k} \right\rangle^{1/2} \# U''(y) \# \left\langle \frac{\eta}{k} \right\rangle^{-1/2} \# U'(y)^{-1/2}.\]
		By the composition formula \eqref{composition_formula}, the smallness assumption on \(U''\), and the Calder\'on--Vaillancourt Theorem \ref{thm CV}, we obtain
		\[\| (Q^*)^{-1} [D_t,P^*P] (-\Delta_t)^{-1/2} Q^{-1} \|_{L^2 \to L^2} \lesssim \epsilon.\]
		Therefore,
		\[\frac{1}{3M^2} \re \langle TSD_tR, [D_t,P^*P] TSR \rangle \lesssim \frac{\epsilon}{3M^2} \| QTSD_tR \|_{L^2} \| QTS \sqrt{-\Delta_t} R \|_{L^2} \lesssim \frac{\epsilon}{M^2} \| QTSkR \|_{L^2}^2 + \frac{\epsilon}{M^2} \| QTSD_tR \|_{L^2}^2,\]
		where we used \eqref{Pythagoras_Q} and Young's inequality in the last step. We conclude that
		\[- \frac{1}{3M^2} \re \langle PTS \Delta_t R, PTSR \rangle \le \frac{1}{3M^2} \| PTSkR \|_{L^2}^2 + \frac{1}{3M^2} \| PTSD_tR \|_{L^2}^2 + \frac{O(\epsilon)}{M^2} \| QTSkR \|_{L^2}^2 + \frac{O(\epsilon)}{M^2} \| QTSD_tR \|_{L^2}^2.\]
		
		Turning to the next term, Corollary \ref{cor TS_commutator_bound} allows us to reduce the analysis to \(- \frac{1}{3} \re \langle PU'TS2ikD_t\Delta_t^{-1} A, PTSR \rangle\). The Cauchy--Schwarz inequality and Lemma \ref{lem P_Q_bound} yield
		\begin{align*}
			- \frac{1}{3} \re \langle PU'TS2ikD_t\Delta_t^{-1} A, PTSR \rangle &\le \frac{2}{3} \sqrt{U'_{\sup}} \| (U')^{1/2} kD_t\Delta_t^{-1} TS A \|_{L^2} \| P^*P Q^{-1} \|_{L^2 \to L^2} \| QTSR \|_{L^2} \\
			&\le (1 + O(\epsilon)) \frac{2\sqrt{\beta}}{3\gamma} U'_{\sup} \| (U')^{1/2} kD_t\Delta_t^{-1} TS A \|_{L^2} \| QTSkR \|_{L^2}.
		\end{align*}
		The definition of \(Q\) gives
		\[\| (U')^{1/2} kD_t\Delta_t^{-1} TS A \|_{L^2} = \beta^{1/2} \| QD_t(-\Delta_t)^{-1/2}TSA \|_{L^2}.\]
		By Lemma \ref{lem comparison}, we have
		\[\| QD_t(-\Delta_t)^{-1/2}TSA \|_{L^2} \le (1 + O(\epsilon)) \| QTSA \|_{L^2}.\]
		Consequently, Young's inequality implies
		\[- \frac{1}{3} \re \langle PTS2ikU'D_t\Delta_t^{-1} A, PTSR \rangle \le \left( \frac{\beta MU'_{\sup}}{3\gamma} + O(\epsilon) \right) \| QTSA \|_{L^2}^2 + \left( \frac{\beta MU'_{\sup}}{3\gamma} + O(\epsilon) \right) \frac{1}{M^2} \| QTSkR \|_{L^2}^2.\]
		
		The subsequent term is handled in the same fashion. Indeed,
		\begin{align*}
			- \frac{1}{3} \re \langle P U' TS 2k (- \Delta_t)^{-1/2} \Xi, PTSR \rangle &\le \frac{2}{3} \sqrt{U'_{\sup}} \| \sqrt{U'} k (- \Delta_t)^{-1/2} TS \Xi \|_{L^2} \| P^*P Q^{-1} \|_{L^2 \to L^2} \| QTSR \|_{L^2} \\
			&\le \frac{2}{3} \sqrt{U'_{\sup}} \beta^{1/2} \| QTS \Xi \|_{L^2} (1 + O(\epsilon)) \frac{\sqrt{\beta U'_{\sup}}}{\gamma} \| QTSkR \|_{L^2} \\
			&\le \left( \frac{2\beta M U'_{\sup}}{3\gamma} + O(\epsilon) \right) \frac{1}{M} \| QTS \Xi \|_{L^2} \| QTSkR \|_{L^2}.
		\end{align*}
		By Corollary \ref{cor TS_commutator_bound} and Young's inequality, it follows that
		\[- \frac{1}{3} \re \langle P TS 2k U' (- \Delta_t)^{-1/2} \Xi, PTSR \rangle  \le \left( \frac{\beta MU'_{\sup}}{3\gamma} + O(\epsilon) \right) \| QTS \Xi \|_{L^2}^2 + \left( \frac{\beta MU'_{\sup}}{3\gamma} + O(\epsilon) \right) \frac{1}{M^2} \| QTSkR \|_{L^2}^2.\]
		
		Similarly, the last term is bounded by
		\[- \frac{1}{3} \re \langle PTS 2k^2U' \Delta_t^{-1} U'' D_t \Delta_t^{-1} R, PTSR \rangle \lesssim \frac{\epsilon}{M^2} \| QTSkR \|_{L^2}^2,\]
		where we also used Lemma \ref{lem comparison}.
		
		\noindent\textbf{Step 5. Closing the energy estimate.} Substituting all the preceding estimates into \eqref{dE/dt}, we obtain
		\begin{align*}
			\frac{d}{dt} E \le& \left( -\frac{2}{3} + \frac{\gamma}{2} \right) \frac{1}{M^2} \| PTSkR \|_{L^2}^2 + \left( -\frac{2}{3} + \gamma \right) \frac{1}{M^2} \| PTSD_tR \|_{L^2}^2 \\
			&+ \left( -\frac{4}{3} + 2\gamma \right) \| PTSA \|_{L^2}^2 - \| PTS \Theta \|_{L^2}^2 - \| PTS \Xi \|_{L^2}^2 \\
			&+ \left( -1 + 2 \beta MU'_{\sup} + \frac{\beta MU'_{\sup}}{3\gamma^2} + \frac{5\beta MU'_{\sup}}{6\gamma} + O(\epsilon) \right) \frac{1}{M^2} \| QTSkR \|_{L^2}^2 \\
			&+ \left( -1 + O(\epsilon) \right) \frac{1}{M^2} \| QTSD_tR \|_{L^2}^2 \\
			&+ \left( -1 + \frac{3\beta}{2}(1 + MU'_{\sup}) + \frac{\beta MU'_{\sup}}{3\gamma^2} + \frac{\beta MU'_{\sup}}{2\gamma} + \frac{4}{9\gamma} + O(\epsilon)\right) \| QTSA \|_{L^2}^2 \\
			&+ \left( -1 + 2\beta + \frac{1}{2\gamma} + O(\epsilon)\right) \| QTS \Theta \|_{L^2}^2 \\
			&+ \left(-1 + \frac{\beta}{2}(1 + MU'_{\sup}) + \frac{\beta MU'_{\sup}}{3\gamma} + \frac{1}{2\gamma} + O(\epsilon)\right) \| QTS \Xi \|_{L^2}^2.
		\end{align*}
		Taking \(\gamma = \frac{2}{3}\), the coefficients of \(\| PTSD_tR \|_{L^2}^2\) and \(\|PTSA\|_{L^2}^2\) vanish. The definition of \(\beta\) and the assumed smallness of \(\epsilon\) imply that the preceding estimate yields \eqref{derivative} for some absolute constant \(c_0 > 0\), independent of \(U'_{\inf}\), \(U'_{\sup}\), and \(M\). This completes the proof.
	\end{proof}
	
	\section{Higher- and lower-order weighted energy estimates}
	\label{sec6}
	
	The motivation for the terminal-time-dependent weights was explained in Subsection \ref{subsec strategy}, and here we give their precise definitions and establish the corresponding higher- and lower-order energy estimates. For each \(s \ge 0\), these weights are defined on the interval \(0 \le t \le s\). Since they are obtained by conjugating Fourier multipliers by \(V\), they commute with \(D_t\), \(\Delta_t\), \(T\), and \(S\), thereby preserving the structure of the basic energy estimate.
	
	\subsection{The higher-order weighted energy estimate}
	
	We first introduce the weight used in the higher-order energy estimate. Let
	\[g(\xi) = 1 + \langle \xi \rangle - \xi.\]
	It satisfies
	\[g(\xi) \approx
	\begin{cases}
		\langle \xi \rangle, &\qquad \xi \le 0, \\
		1, &\qquad \xi \ge 0,
	\end{cases}
	\qquad g'(\xi) = \xi \langle \xi \rangle^{-1} - 1 < 0.\]
	For each terminal time \(s \ge 0\), define
	\[G_s = V \op \left( g \left( \frac{\eta}{k} + 2 U'_{\sup} (t - s) \right) \right) V^{-1}.\]
	We now state the corresponding higher-order energy estimate.
	
	\begin{proposition}
		\label{prop higher}
		Let \(s \ge 0\) be fixed, and let \(G_s\) be defined as above. For any solution \((R, A, \Theta)\) of \eqref{R_A_Theta}, define
		\begin{align*}
			\widetilde{E}_{1, s} =& \frac{1}{2M^2} \| TSkG_sR \|_{L^2}^2 + \frac{1}{2M^2} \| TSD_tG_sR \|_{L^2}^2 + \frac{1}{2} \| TSG_sA \|_{L^2}^2 + \frac{1}{2} \| TSG_s \Theta \|_{L^2}^2 \\
			&+ \frac{1}{2} \| TS k (- \Delta_t)^{-1/2} (G_s \Theta - U' G_s R) \|_{L^2}^2 + \frac{1}{3} \re \langle PTSG_sA, PTSG_sR \rangle.
		\end{align*}
		Then
		\begin{equation}
			\label{higher_equivalent}
			\widetilde{E}_{1, s} \approx \frac{1}{M^2} \| TS\sqrt{-\Delta_t}G_sR \|_{L^2}^2 + \| TSG_sA \|_{L^2}^2 + \| TSG_s \Theta \|_{L^2}^2 + \| TS k (-\Delta_t)^{-1/2} G_s (\Theta - U'R) \|_{L^2}^2,
		\end{equation}
		where the equivalence constants are uniform for \(0 \le t \le s\). Moreover,
		\begin{equation}
			\label{higher_derivative}
			\frac{d}{dt} \widetilde{E}_{1, s} (t) \le 0, \qquad 0 \le t \le s.
		\end{equation}
		In particular,
		\[\widetilde{E}_{1, s} (s) \le \widetilde{E}_{1, s} (0).\]
	\end{proposition}
	
	\begin{proof}
		Applying the basic energy equivalence \eqref{equivalent_energy} to \((G_sR,G_sA,G_s\Theta)\), we first obtain
		\[\widetilde{E}_{1,s}\approx \frac{1}{M^2} \| TS \sqrt{-\Delta_t} G_s R \|_{L^2}^2 + \| TS G_s A \|_{L^2}^2 + \| TS G_s \Theta \|_{L^2}^2 + \| TS k (-\Delta_t)^{-1/2} (G_s \Theta - U' G_s R) \|_{L^2}^2.\]
		To derive \eqref{higher_equivalent}, it remains to replace the last term by \(\| TS k (-\Delta_t)^{-1/2} G_s (\Theta - U'R) \|_{L^2}^2\). The identity
		\[G_s \Theta - U' G_s R = G_s (\Theta - U' R) + [G_s, U'] R\]
		shows that it suffices to control the commutator term. 
		
		For simplicity, we denote \(\zeta = \eta + 2k U'_{\sup} (t - s)\). Then the composition formula \eqref{composition_formula} and Taylor's theorem give
		\[\sigma (V^{-1} [G_s, U'] G_s^{-1} V)(y,\eta) = \frac{1}{2 \pi} \int_{\mathbb{R}} \int_{\mathbb{R}} \int_{0}^{1} e^{-iz\lambda} \lambda \frac{1}{k} g' \left( \frac{\zeta + \theta\lambda}{k} \right) U'(y+z) g \left( \frac{\zeta}{k} \right)^{-1} d\theta d\lambda dz.\]
		Since \(g'\), \(g^{-1}\), and all their derivatives required below are uniformly bounded, the integrations by parts in \(z\) and \(\lambda\), together with the smallness assumption on the derivatives of \(U'\), yield
		\[|\partial_y^j \partial_\eta^l \sigma (V^{-1} [G_s, U'] G_s^{-1} V) (y,\eta)| \lesssim \epsilon\]
		for all nonnegative \(j\) and \(l\) up to a sufficiently large finite order. Applying the composition formula once more, together with the Calder\'on--Vaillancourt Theorem \ref{thm CV}, we have
		\[\| TS [G_s, U'] G_s^{-1} (TS)^{-1} \|_{L^2 \to L^2} \lesssim \epsilon.\]
		Hence
		\[\| TS k (-\Delta_t)^{-1/2} [G_s,U'] R \|_{L^2} \lesssim \epsilon \| TSG_sR \|_{L^2} \lesssim \frac{\epsilon}{M} \|TS \sqrt{-\Delta_t} G_s R \|_{L^2}.\]
		Here we obtained the factor \(M^{-1}\) at the expense of the \(M\)-dependent implicit constant in \(\lesssim\). Taking \(\epsilon\) sufficiently small, the commutator term is absorbed by the acoustic part of the energy, and \eqref{higher_equivalent} follows.
		
		To prove \eqref{higher_derivative}, we differentiate \(\widetilde E_{1,s}\). The resulting terms are divided into two classes, according to whether the time derivative falls on \(G_s\) or on the remaining factors. As argued above, the terms in the latter class consist of the corresponding terms in \eqref{dE/dt}, with \(R, A, \Theta\) replaced by \(G_sR, G_sA, G_s \Theta\), together with commutator errors. On the other hand, owing to the careful choice of \(G_s\), the terms arising from the positive quadratic part of the energy generate an additional dissipative contribution, which will be used to absorb these commutator errors.
		
		\noindent\textbf{Step 1. The basic terminal-time dissipation.} We first identify the dissipation generated when the time derivative falls on \(G_s\) in the positive quadratic part of the energy. Taking the contribution involving \(G_sA\) as the basic case, we compute the symbol of \(TS (\partial_t G_s) G_s^{-1} (TS)^{-1}\) and use its negative principal part to define the dissipative operator \(B_s\). The components requiring additional commutator estimates are treated in the next step.
		
		Inserting \(G_s^{-1} (TS)^{-1} TS G_s\), we write
		\[\re \langle TS \partial_t G_s A, TSG_sA \rangle = \re \langle TS (\partial_t G_s)G_s^{-1} (TS)^{-1} TSG_sA, TSG_sA \rangle.\]
		We first compute the symbol of \(TS (\partial_t G_s) G_s^{-1} (TS)^{-1}\). Differentiating \(G_s\) gives
		\[\partial_t G_s = ik V \left[ U, \op \left( g \left( \frac{\zeta}{k} \right) \right) \right] V^{-1} + 2 U'_{\sup} V \op \left( g' \left( \frac{\zeta}{k} \right) \right) V^{-1}.\]
		The composition formula \eqref{composition_formula} and Taylor's theorem yield
		\[\sigma \left( \left[ U, \op \left( g \left( \frac{\zeta}{k} \right) \right) \right] \right) = \frac{iU'(y)}{k} g' \left( \frac{\zeta}{k} \right) - \frac{1}{2 \pi} \int_{\mathbb{R}} \int_{\mathbb{R}} \int_{0}^{1} e^{-i z \lambda} \lambda^2 (1 - \theta) \frac{1}{k^2} g'' \left( \frac{\zeta + \theta \lambda}{k} \right) U(y + z) d \theta d \lambda dz.\]
		Thus,
		\begin{align*}
			\sigma (V^{-1} (\partial_t G_s) G_s^{-1} V) =& (2 U'_{\sup} - U'(y)) g' \left( \frac{\zeta}{k} \right) \left( g \left( \frac{\zeta}{k} \right) \right)^{-1} \\
			&- \frac{i}{2 \pi} \int_{\mathbb{R}} \int_{\mathbb{R}} \int_{0}^{1} e^{-i z \lambda} \lambda^2 (1 - \theta) \frac{1}{k} g'' \left( \frac{\zeta + \theta \lambda}{k} \right) \left( g \left( \frac{\zeta}{k} \right) \right)^{-1} U(y + z) d \theta d \lambda dz \\
			&\eqqcolon p_G (y, \zeta) + r_G (y, \zeta).
		\end{align*}
		We next estimate the remainder relative to this negative principal symbol. Since
		\[g''(\xi) = \langle \xi \rangle^{-3}, \qquad - g'(\xi) \gtrsim \langle \xi \rangle^{-2},\]
		we have, for every \(m \ge 2\),
		\[\left| g^{(m)} \left( \frac{\zeta + \theta \lambda}{k} \right) \right| \lesssim \left\langle \frac{\zeta + \theta \lambda}{k} \right\rangle^{-3} \lesssim \left\langle \frac{\zeta}{k} \right\rangle^{-2} \left\langle \frac{\theta \lambda}{k} \right\rangle^2 \lesssim - g' \left( \frac{\zeta}{k} \right) \langle \lambda \rangle^2.\]
		Integrating by parts in \(z\) and \(\lambda\), using the smallness assumption on the derivatives of \(U''\), together with the bound
		\[\left| \partial_\zeta^m \left( g \left( \frac{\zeta}{k} \right) \right)^{-1} \right| \lesssim \left( g \left( \frac{\zeta}{k} \right) \right)^{-1}, \qquad m \ge 0,\]
		we obtain
		\begin{equation}
			\label{r_G}
			| \partial_y^j \partial_\zeta^l r_G (y, \zeta) | \lesssim - \epsilon g' \left( \frac{\zeta}{k} \right) \left( g \left( \frac{\zeta}{k} \right) \right)^{-1} \lesssim - \epsilon p_G(y, \zeta),
		\end{equation}
		for all \(j\) and \(l\) up to a sufficiently large finite order. Here we used \(2 U'_{\sup} - U' \ge U'_{\sup}\), with the factor \((U'_{\sup})^{-1}\) absorbed into the implicit constant in \(\lesssim\).
		
		It remains to conjugate this operator by \(TS\). A further application of the composition formula \eqref{composition_formula} gives
		\begin{align*}
			\sigma (V^{-1} TS (\partial_t G_s) G_s^{-1} (TS)^{-1} V) =& (p_G + r_G) (y, \zeta) \\
			&+ \frac{1}{2\pi} \int_{\mathbb{R}} \int_{\mathbb{R}} \int_{0}^{1} e^{-iz\lambda} \lambda \partial_\eta \tau (\eta + \theta \lambda) (p_G + r_G) (y+z, \zeta) (\tau (\eta))^{-1} d \theta d \lambda dz.
		\end{align*}
		Here \(\tau = \sigma (V^{-1} TS V)\) satisfies
		\begin{equation}
			\label{tau}
			\tau (\eta) = \exp \left(\frac{1}{\gamma} \sinh^{-1} \frac{\eta}{k} + \frac{1}{\beta} \tan^{-1} \frac{\eta}{k} \right) \approx
			\begin{cases}
				\left\langle \frac{\eta}{k} \right\rangle^{-1/\gamma}, &\qquad \frac{\eta}{k} \le 0, \\
				\left\langle \frac{\eta}{k} \right\rangle^{1/\gamma}, &\qquad \frac{\eta}{k} \ge 0.
			\end{cases}
		\end{equation}
		
		Denote the integral term above by \(r_{TS,G}(y, \zeta)\). We first integrate by parts in \(z\). Since the other factors are independent of \(z\), the resulting derivative falls on \((p_G + r_G)(y+z, \zeta)\). It follows that
		\[|\partial_z p_G (y + z, \zeta)| \lesssim - \epsilon p_G(y, \zeta).\]
		Together with \eqref{r_G}, this shows that the first integration by parts produces an additional factor \(O(\epsilon)\). Moreover, Claim \ref{cla sinh} implies
		\[\partial_\eta^m \tau (\eta + \theta \lambda) (\tau (\eta))^{-1} \lesssim \langle \lambda \rangle^3, \qquad m \ge 1.\]
		Further integrations by parts in \(z\) and \(\lambda\), using this bound, then yield
		\[| r_{TS,G}(y, \zeta)| \lesssim - \epsilon p_G(y, \zeta).\]
		The same estimate holds for the derivatives of \(r_{TS,G}\) with respect to \(y\) and \(\zeta\). Combining this term with \(r_G\), we set
		\[\widetilde{r}_G = r_G + r_{TS,G}.\]
		It follows that
		\begin{equation}
			\label{TS_G}
			\sigma (V^{-1} TS (\partial_t G_s) G_s^{-1} (TS)^{-1} V) = p_G + \widetilde{r}_G,
		\end{equation}
		with
		\begin{equation}
			\label{widetilde_r_G}
			|\partial_y^j \partial_\zeta^l \widetilde{r}_G(y, \zeta)| \lesssim - \epsilon p_G(y, \zeta),
		\end{equation}
		for all \(j\) and \(l\) up to a sufficiently large finite order.
		
		We define
		\[B_s = V \op (\sqrt{- p_G}) V^{-1} = V \op \left( (2 U'_{\sup} - U'(y))^{1/2} b(\zeta) \right) V^{-1},\]
		where, for simplicity, we denote
		\[b(\zeta) = \sqrt{- g' \left( \frac{\zeta}{k} \right) \left( g \left( \frac{\zeta}{k} \right) \right)^{-1}}.\]
		Then
		\[\sigma (V^{-1} B_s^* B_s V) = b(\zeta) \# (2 U'_{\sup} - U'(y)) \# b(\zeta) = (2 U'_{\sup} - U'(y)) b(\zeta)^2 + [b(\zeta), 2 U'_{\sup} - U'(y)] b(\zeta).\]
		The first term is equal to \(- p_G\), and we denote the second term by \(r_B\). By the composition formula \eqref{composition_formula} and Taylor's theorem, we have
		\[[b(\zeta), 2 U'_{\sup} - U'(y)] = [b(\zeta), - U'(y)] = - \frac{1}{2\pi} \int_{\mathbb{R}} \int_{\mathbb{R}} \int_{0}^{1} e^{-iz\lambda} \lambda \partial_\zeta b(\zeta + \theta \lambda) U'(y+z) d \theta d \lambda dz.\]
		To estimate this commutator, recall that
		\[g(\xi) \approx \langle \xi_- \rangle, \qquad g'(\xi) \approx - \langle \xi_+ \rangle^{-2},\]
		where \(\xi_+ = \max \{ \xi, 0 \}\) and \(\xi_- = \max \{ - \xi, 0 \} \). It follows that
		\[\sqrt{- g' (\xi) g(\xi)^{-1}} \approx \langle \xi_+ \rangle^{-1} \langle \xi_- \rangle^{-1/2} \approx \langle \xi \rangle^{-1/2} \langle \xi_+ \rangle^{-1/2}.\]
		Thus, for \(m \ge 1\),
		\begin{align*}
			|\partial_\zeta^m b(\zeta + \theta \lambda)| &\lesssim b(\zeta + \theta \lambda) \\
			&\approx \left\langle \frac{\zeta + \theta \lambda}{k} \right\rangle^{-1/2} \left\langle \left( \frac{\zeta + \theta \lambda}{k} \right)_+ \right\rangle^{-1/2} \\
			&\lesssim \left\langle \frac{\zeta}{k} \right\rangle^{-1/2} \left\langle \frac{\theta \lambda}{k} \right\rangle^{1/2} \left\langle \left(\frac{\zeta}{k}\right)_+ \right\rangle^{-1/2} \left\langle \frac{\theta \lambda}{k} \right\rangle^{1/2} \\
			&\lesssim b(\zeta) \langle \lambda \rangle.
		\end{align*}
		Integrating by parts in the oscillatory variables and then composing with \(b(\zeta)\), we obtain
		\begin{equation}
			\label{r_B}
			|r_B(y, \zeta)| \lesssim - \epsilon p_G(y, \zeta).
		\end{equation}
		The same estimate holds for the derivatives of \(r_B\) with respect to \(y\) and \(\zeta\).
		
		Combining these decompositions, we obtain
		\[\re \langle TS (\partial_t G_s)G_s^{-1} (TS)^{-1} TSG_sA, TSG_sA \rangle + \| B_s TSG_sA \|_{L^2}^2 = \re \langle \mathcal{R}_{G, B} TSG_sA, TSG_sA \rangle,\]
		where \(\mathcal{R}_{G, B} = V \op (\widetilde{r}_G + r_B) V^{-1}\). We only need to estimate \(\| (B_s^{-1})^* \mathcal{R}_{G, B} B_s^{-1} \|_{L^2 \to L^2}\). Its symbol is
		\[\sigma (V^{-1} (B_s^{-1})^* \mathcal{R}_{G, B} B_s^{-1} V) = (2 U'_{\sup} - U'(y))^{-1/2} \# b(\zeta)^{-1} \# (\widetilde{r}_G + r_B) \# b(\zeta)^{-1} \# (2 U'_{\sup} - U'(y))^{-1/2}.\]
		The composition formula \eqref{composition_formula}, the estimates \eqref{widetilde_r_G} and \eqref{r_B} for \(\widetilde{r}_G\) and \(r_B\), respectively, and the Calder\'on--Vaillancourt Theorem \ref{thm CV} give
		\[\left\| b(\zeta)^{-1} \# (\widetilde{r}_G + r_B) \# b(\zeta)^{-1} \right\|_{L^2 \to L^2} \lesssim \epsilon.\]
		Moreover, since the two outer factors \((2U'_{\sup} - U'(y))^{-1/2}\) can be absorbed into the implicit constant in \(\lesssim\), we obtain
		\begin{equation}
			\label{B_s_remainder}
			\| (B_s^{-1})^* \mathcal{R}_{G, B} B_s^{-1} \|_{L^2 \to L^2} \lesssim \epsilon.
		\end{equation}
		Thus,
		\[\re \langle TS \partial_t G_s A, TSG_sA \rangle = \re \langle TS (\partial_t G_s)G_s^{-1} (TS)^{-1} TSG_sA, TSG_sA \rangle \le (-1 + O(\epsilon)) \| B_s TSG_sA \|_{L^2}^2 \le 0.\]
		
		\noindent\textbf{Step 2. The remaining contributions from the positive quadratic part.} The term involving \(D_t G_s R\) requires a minor modification. When the time derivative falls on \(G_s\), the resulting contribution is
		\[\frac{1}{M^2} \re \langle TS D_t \partial_t G_s R, TS D_t G_s R\rangle = \frac{1}{M^2} \re \langle TS \partial_t G_s D_t R, TS D_t G_s R \rangle + \frac{1}{M^2} \re \langle TS [D_t, \partial_t G_s] R, TS D_t G_s R \rangle.\]
		Although \(D_t\) commutes with \(G_s\), it does not commute with \(\partial_t G_s\) in general. Differentiating the identity \([D_t, G_s] = 0\) with respect to time, we obtain
		\[0 = \partial_t [D_t, G_s] = [\partial_t D_t, G_s] + [D_t, \partial_t G_s].\]
		Since \(\partial_t D_t = -ikU'\), it follows that
		\[[D_t, \partial_t G_s] = - [\partial_t D_t, G_s] = ik [U', G_s].\]
		Consequently,
		\[\frac{1}{M^2} \re \langle TS D_t \partial_t G_s R, TS D_t G_s R\rangle = \frac{1}{M^2} \re \langle TS \partial_t G_s D_t R, TS D_t G_s R \rangle + \frac{1}{M^2} \re \langle TS ik [U', G_s] R, TS D_t G_s R \rangle.\]
		The first term is treated exactly as the term involving \(A\), with \(A\) replaced by \(D_tR\), and therefore generates the corresponding \(B_s\)-dissipation. The second term is an additional commutator error, which is estimated in the same way as the commutator terms considered below. We therefore omit the details.
		
		We also consider the contribution involving \(k (-\Delta_t)^{-1/2} (G_s \Theta - U' G_s R)\), as it does not admit the same simple decomposition as the \(D_t G_s R\)-term. The identity
		\[(\partial_t G_s) \Theta - U' (\partial_t G_s) R = (\partial_t G_s) G_s^{-1} (G_s \Theta - U' G_s R) + [(\partial_t G_s) G_s^{-1}, U'] G_s R\]
		separates this contribution into a principal term and a commutator term. For the principal term, we need to consider the operator \(TS (-\Delta_t)^{-1/2} (\partial_t G_s) G_s^{-1} (-\Delta_t)^{1/2} (TS)^{-1}\), whose symbol after conjugation by \(V\) is
		\[\sigma (V^{-1} TS (-\Delta_t)^{-1/2} (\partial_t G_s) G_s^{-1} (-\Delta_t)^{1/2} (TS)^{-1} V) = \left\langle \frac{\eta}{k} \right\rangle^{-1} \# (p_G + \widetilde{r}_G) \# \left\langle \frac{\eta}{k} \right\rangle.\]
		By the composition formula \eqref{composition_formula}, the principal symbol of this operator remains \(p_G\), while the remainder and all its derivatives required by the Calder\'on--Vaillancourt Theorem \ref{thm CV} are bounded by \(O(\epsilon)(-p_G)\). Therefore, arguing as in the proof of \eqref{B_s_remainder}, we obtain
		\begin{align*}
			&\re \langle TS k (-\Delta_t)^{-1/2} (\partial_t G_s) G_s^{-1} (G_s \Theta - U' G_s R), TS k (-\Delta_t)^{-1/2} (G_s \Theta - U' G_s R) \rangle \\
			\le& (-1 + O(\epsilon)) \| B_s TS k (-\Delta_t)^{-1/2} (G_s \Theta - U' G_s R) \|_{L^2}^2.
		\end{align*}
		
		For the commutator term, we first recall the decomposition \(\sigma (V^{-1} (\partial_t G_s) G_s^{-1} V) = p_G + r_G\). The composition formula \eqref{composition_formula} and Taylor's theorem give
		\[\sigma (V^{-1} [(\partial_t G_s) G_s^{-1}, U'] V) = - \frac{i}{2\pi} \int_{\mathbb{R}} \int_{\mathbb{R}} \int_{0}^{1} e^{-iz\lambda} \partial_\zeta (p_G + r_G)(y,\zeta+\theta\lambda) U''(y+z) d\theta d\lambda dz.\]
		Indeed, in the Taylor expansion of \(p_G+r_G\) with respect to \(\zeta\), the zeroth-order term cancels in the commutator, whereas the first-order term contains a factor \(\lambda\). Integrating by parts in \(z\) transfers this factor to \(U'\), thereby producing \(U''\). Using \eqref{r_G}, the smallness assumptions on the derivatives of \(U'\), and further integrations by parts in \(z\) and \(\lambda\), we obtain
		\[|\partial_y^j \partial_\zeta^l \sigma (V^{-1} [(\partial_t G_s) G_s^{-1}, U'] V)| \lesssim - \epsilon p_G\]
		for all \(j\) and \(l\) up to the finite order required by the Calder\'on--Vaillancourt Theorem \ref{thm CV}. Conjugating by \(TS\) and repeating the argument leading to \eqref{TS_G} and \eqref{widetilde_r_G}, we further obtain
		\[|\partial_y^j \partial_\zeta^l \sigma (V^{-1} TS [(\partial_t G_s) G_s^{-1}, U'] (TS)^{-1} V)| \lesssim - \epsilon p_G.\]
		The same argument as that leading to \eqref{B_s_remainder} yields
		\[\| (B_s^{-1})^* TS k (-\Delta_t)^{-1/2} [(\partial_t G_s) G_s^{-1}, U'] (TS)^{-1} B_s^{-1} \|_{L^2 \to L^2} \lesssim \epsilon.\]
		Therefore, by the Cauchy--Schwarz inequality and Young's inequality, the commutator contribution is absorbed by the \(B_s\)-dissipation associated with \(k (-\Delta_t)^{-1/2} (G_s \Theta-U' G_s R)\) and \(k G_s R\). Combining this estimate with the principal contribution, we obtain
		\begin{align*}
			&\re\langle TS k (-\Delta_t)^{-1/2} ((\partial_t  G_s) \Theta - U' (\partial_t G_s) R), TS k (-\Delta_t)^{-1/2} (G_s \Theta - U' G_s R) \rangle \\
			\le& (-1 + O(\epsilon)) \| B_s TS k (-\Delta_t)^{-1/2} (G_s \Theta - U' G_s R) \|_{L^2}^2 + \frac{O(\epsilon)}{M^2} \| B_s TS kG_sR \|_{L^2}^2.
		\end{align*}
		
		\noindent\textbf{Step 3. Commutators with the evolution equations.} We next consider the terms in the second class, namely, those for which the time derivative falls on factors other than \(G_s\). Consider, for example, one of the terms generated by \(\partial_t A\), which can be decomposed as
		\[- \re \langle TSG_s2ikU'D_t\Delta_t^{-1} A, TSG_sA \rangle = - \re \langle TS2ikU'D_t\Delta_t^{-1}G_sA, TSG_sA \rangle + \re \langle TS 2ik[U', G_s] D_t\Delta_t^{-1} A, TSG_sA \rangle.\]
		With \(A\) replaced by \(G_sA\), the first term has exactly the same structure as the corresponding term in \eqref{dE/dt} and is therefore controlled by the estimate used in the proof of Proposition \ref{prop basic_energy}. It remains to estimate the commutator term.
		
		By a \(B_s\)-normalized symbol calculation analogous to that leading to \eqref{B_s_remainder}, we obtain
		\begin{equation}
			\label{B_s_commutator}
			\| (B_s^{-1})^* TS [U',G_s] G_s^{-1} (TS)^{-1} B_s^{-1} \|_{L^2 \to L^2} \lesssim \epsilon.
		\end{equation}
		We also need the \(B_s\)-weighted analogue of Lemma \ref{lem comparison}. More precisely, under the assumptions of that lemma,
		\begin{equation}
			\label{B_s_comparison}
			\| B_s \mathcal{M} f \|_{L^2} \le (1 + O(\epsilon)) \| B_s \mathcal{N} f \|_{L^2}.
		\end{equation}
		Indeed, after conjugating by \(V\), the symbol \(B_s \mathcal{M} \mathcal{N}^{-1} B_s^{-1}\) is
		\begin{align*}
			\sigma (V^{-1} B_s \mathcal{M} \mathcal{N}^{-1} B_s^{-1} V) &= (2 U'_{\sup} - U'(y))^{1/2} \# b(\zeta) \# \frac{m}{n} (\eta) \# b(\zeta)^{-1} \# (2 U'_{\sup} - U'(y))^{-1/2} \\
			&= (2 U'_{\sup} - U'(y))^{1/2} \# \frac{m}{n} (\eta) \# (2 U'_{\sup} - U'(y))^{-1/2}.
		\end{align*}
		The same argument as in the proof of Lemma \ref{lem comparison} therefore gives \eqref{B_s_comparison}.
		
		Thus, combining \eqref{B_s_commutator} and \eqref{B_s_comparison}, we obtain
		\begin{align*}
			\re \langle TS 2ik [U',G_s] D_t \Delta_t^{-1} A, TSG_sA \rangle &\le \| (B_s^{-1})^* TS [U',G_s] G_s^{-1} (TS)^{-1} B_s^{-1} \|_{L^2 \to L^2} \\
			&\quad \times \| B_s 2k D_t \Delta_t^{-1} TSG_sA \|_{L^2} \| B_s TSG_sA \|_{L^2} \\
			&\lesssim \epsilon \| B_s TSG_sA \|_{L^2}^2.
		\end{align*}
		
		All the remaining terms in the second class are treated in the same way. Their noncommutator parts are controlled by the corresponding estimates in the proof of Proposition \ref{prop basic_energy}. In each commutator part, we first isolate either \([U',G_s]G_s^{-1}\) or \([U'',G_s]G_s^{-1}\), and estimate the remaining operator factors using the corresponding \(B_s\)-weighted bounds, which follow from the same symbolic arguments as above. Hence, all the commutator errors are absorbed by the \(B_s\)-dissipation, and we omit the details.
		
		\noindent\textbf{Step 4. Contributions from the cross term.} Finally, we consider the terms generated when the time derivative falls on \(G_s\) in the cross part of the energy. We take \(\frac{1}{3} \re \langle PTSG_sA, PTS \partial_t G_s R \rangle\) as an example. Inserting \(U'^{-1/2} U'^{1/2}\) into the inner product and applying the Cauchy--Schwarz inequality, we obtain
		\begin{align*}
			\frac{1}{3} \re \langle PTSG_sA, PTS \partial_t G_s R \rangle \le& \frac{1}{3} \| U'^{-1/2} P^*P Q^{-1} \|_{L^2 \to L^2} \| QTSG_sA \|_{L^2} \\
			&\times \| U'^{1/2} TS (\partial_t G_s) G_s^{-1} (TS)^{-1} U'^{-1/2} \|_{L^2 \to L^2} \| U'^{1/2} TSG_sR \|_{L^2}.
		\end{align*}
		The symbol computation in the proof of Lemma \ref{lem P_Q_bound} gives
		\[\| U'^{-1/2} P^*P Q^{-1} \|_{L^2 \to L^2} \le (1 + O(\epsilon)) \frac{\sqrt{\beta}}{\gamma}.\]
		Commuting the multiplication operator \(U'^{-1/2}\) through the middle operator produces only a remainder of size \(O(\epsilon)\). Thus, \eqref{TS_G} and \eqref{widetilde_r_G} imply
		\begin{align*}
			\| U'^{1/2} TS (\partial_t G_s) G_s^{-1} (TS)^{-1} U'^{-1/2} \|_{L^2 \to L^2} &\le (1 + O(\epsilon)) \| \op (p_G) \|_{L^2 \to L^2} \\
			&\le (1 + O(\epsilon)) \| 2 U'_{\sup} - U'(y) \|_{L^\infty} \left\| g' \left( \frac{\zeta}{k} \right) \left( g \left( \frac{\zeta}{k} \right) \right)^{-1} \right\|_{L^\infty} \\
			&\le 2 (1 + O(\epsilon)) U'_{\sup},
		\end{align*}
		where we have used
		\[0 < - \frac{g'(\xi)}{g(\xi)} = \frac{\langle \xi \rangle - \xi }{\langle \xi \rangle(1 + \langle \xi \rangle - \xi)} \le 1.\]
		The definition of \(Q\) gives
		\[\| (U')^{1/2}TSG_sR \|_{L^2} = \beta^{1/2} \| Q |k|^{-1} \sqrt{- \Delta_t} TSG_sR \|_{L^2}.\]
		Arguing as in the proof of \eqref{Pythagoras_P}, and using Lemma \ref{lem comparison} together with the parallelogram identity, we obtain
		\[\| Q |k|^{-1} \sqrt{- \Delta_t} TSG_sR \|_{L^2} \le (1 + O(\epsilon)) (\| QTSkG_sR \|_{L^2}^2 + \| QTSD_tG_sR \|_{L^2}^2)^{1/2}.\]
		Combining these estimates and using Young's inequality, we conclude that
		\begin{align*}
			\frac{1}{3} \re \langle PTSG_sA, PTS \partial_t G_s R \rangle \le& \left( \frac{\beta M U'_{\sup}}{2} + O(\epsilon) \right) \| QTSG_sA \|_{L^2}^2 \\
			&+ \left( \frac{\beta M U'_{\sup}}{2} + O(\epsilon) \right) \frac{1}{M^2} \| QTSkG_sR \|_{L^2}^2 + \left( \frac{\beta M U'_{\sup}}{2} + O(\epsilon) \right) \frac{1}{M^2} \| QTSD_tG_sR \|_{L^2}^2.
		\end{align*}
		By the definition of \(\beta\) and the assumed smallness of \(\epsilon\), this contribution is absorbed by the \(Q\)-dissipation arising from differentiating \(TS\).
		
		Combining all the preceding estimates yields
		\[\frac{d}{dt} \widetilde{E}_{1,s}(t) \le 0,\]
		which proves the proposition.
	\end{proof}
	
	\subsection{The lower-order weighted energy estimate}
	
	We next introduce the weight used in the lower-order energy estimate. Let
	\[h(\xi) = 1 + \xi \langle \xi \rangle^{-1}.\]
	It satisfies
	\[h(\xi) \approx
	\begin{cases}
		\langle \xi \rangle^{-2}, &\qquad \xi \le 0, \\
		1, &\qquad \xi \ge 0,
	\end{cases}
	\qquad h'(\xi) = \langle \xi \rangle^{-3} > 0.\]
	For each terminal time \(s \ge 0\), define
	\[H_s = V \op \left( h \left( \frac{\eta}{k} + \frac{1}{2} U'_{\inf} (t - s) \right) \right) V^{-1}.\]
	The corresponding lower-order energy estimate is stated as follows.
	
	\begin{proposition}
		\label{prop lower}
		Let \(s \ge 0\) be fixed, and let \(H_s\) be defined as above. For any solution \((R, A, \Theta)\) of \eqref{R_A_Theta}, define
		\begin{align*}
			\widetilde{E}_{-2, s} =& \frac{1}{2M^2} \| TSkH_sR \|_{L^2}^2 + \frac{1}{2M^2} \| TSD_tH_sR \|_{L^2}^2 + \frac{1}{2} \| TSH_sA \|_{L^2}^2 + \frac{1}{2} \| TSH_s \Theta \|_{L^2}^2 \\
			&+ \frac{1}{2} \| TS k (- \Delta_t)^{-1/2} (H_s \Theta - U' H_s R) \|_{L^2}^2 + \frac{1}{3} \re \langle PTSH_sA, PTSH_sR \rangle.
		\end{align*}
		Then
		\[\widetilde{E}_{-2, s} \approx \frac{1}{M^2} \| TS\sqrt{-\Delta_t}H_sR \|_{L^2}^2 + \| TSH_sA \|_{L^2}^2 + \| TSH_s \Theta \|_{L^2}^2 + \| TS k (-\Delta_t)^{-1/2} H_s (\Theta - U'R) \|_{L^2}^2,\]
		where the equivalence constants are uniform for \(0 \le t \le s\). Moreover,
		\[\frac{d}{dt} \widetilde{E}_{-2, s} (t) \le 0, \qquad 0 \le t \le s.\]
		In particular,
		\[\widetilde{E}_{-2, s} (s) \le \widetilde{E}_{-2, s} (0).\]
	\end{proposition}
	
	\begin{proof}
		The proof is the same as that of Proposition \ref{prop higher}, with \(G_s\) and \(g\) replaced by \(H_s\) and \(h\), respectively. Indeed, the principal symbol of \((\partial_t H_s) H_s^{-1}\) is
		\[p_H (y, \eta) = \left( \frac{1}{2} U'_{\inf} - U'(y) \right) h' \left( \frac{\eta}{k} + \frac{1}{2} U'_{\inf} (t-s) \right) h \left( \frac{\eta}{k} + \frac{1}{2} U'_{\inf} (t-s) \right)^{-1}.\]
		Since \(U' \ge U'_{\inf}\) and \(h' > 0\), we have
		\[p_H(y,\eta) \le -\frac{U'_{\inf}}{2} h' \left( \frac{\eta}{k} + \frac{1}{2} U'_{\inf} (t-s) \right) h \left( \frac{\eta}{k} + \frac{1}{2} U'_{\inf} (t-s) \right)^{-1} < 0.\]
		Moreover,
		\[\left| \frac{h'(\xi)}{h(\xi)} \right| = \left| \frac{\langle \xi \rangle - \xi}{\langle \xi \rangle^2} \right| \lesssim 1, \qquad \left| \partial_\xi^m \frac{h'(\xi)} {h(\xi)} \right| \lesssim \frac{h'(\xi)}{h(\xi)},\qquad m \ge 0.\]
		Together with Peetre's inequality, these estimates provide precisely the relative symbol bounds used in the proof of Proposition \ref{prop higher}. Hence, the remainder, commutator, and cross-term estimates follow from the same arguments, and the proposition is proved.
	\end{proof}
	
	\section{Proof of the main theorem}
	\label{sec7}
	
	In this section, we combine the weighted energy estimates with propagator and mixing bounds to prove Theorem \ref{thm main}. The acoustic growth bounds follow from the terminal-time energies, whereas the solenoidal damping estimates additionally require propagator and two-time mixing bounds.
	
	\subsection{Propagator estimates}
	
	To obtain the desired decay bound for \(v_{\mathrm{sol}}\), it suffices to estimate \(W\). Rather than applying Duhamel's formula directly to \(W\), we apply it to the modified variable
	\[\Theta = W + L(t) R,\qquad L(t) = U' + U'' D_t \Delta_t^{-1},\]
	and then recover \(W\) from this relation. It follows from \eqref{R_A_Theta} that \(\Theta\) satisfies
	\begin{equation}
		\label{Theta}
		\partial_t \Theta = ikU'' \Delta_t^{-1} \Theta + 2 ik^3 U'' \Delta_t^{-1} U' \Delta_t^{-1} R.
	\end{equation}
	Let \(\Phi(t,s)\) denote the evolution propagator associated with the homogeneous equation, that is,
	\[\partial_t \Phi(t,s) = ikU'' \Delta_t^{-1} \Phi(t,s), \qquad \Phi(s,s) = I.\]
	We record the following boundedness properties of \(\Phi(t,s)\).
	
	\begin{proposition}
		\label{prop Phi}
		Let \(\Phi(t,s)\) be defined as above. Then, for every \(0 \le r \le 2\),
		\begin{equation}
			\label{Phi_bound}
			\| \Phi(t,s) \|_{\widetilde{H}^r \to \widetilde{H}^r} \lesssim 1,
		\end{equation}
		uniformly for all \(0\le s \le t\). Recall that \(\widetilde H^r\) denotes the \(k\)-scaled Sobolev space introduced in Subsection \ref{subsec notations}, with norm
		\[\| f \|_{\widetilde{H}^r} = \left\| \op \left( \left\langle \frac{\eta}{k} \right\rangle^r \right) f \right\|_{L^2}.\]
		Moreover, for every \(0 \le r \le 1\) and every \(f\) such that \((- \Delta_s)^r f \in L^2\),
		\begin{equation}
			\label{Phi_adapted}
			\| (-\Delta_s)^r \Phi(t,s) f \|_{L^2} \lesssim \| (-\Delta_s)^r f \|_{L^2}.
		\end{equation}
		Here the fractional powers are defined by
		\[(- \Delta_t)^r = e^{ikUt} (- \Delta_k)^r e^{-ikUt}, \qquad r \in \mathbb{R},\]
		where \((-\Delta_k)^r\) denotes the Fourier multiplier with symbol \((k^2 + \eta^2)^r\).
	\end{proposition}
	
	\begin{proof}
		Let \(Z \in \widetilde{H}^r\) solve the homogeneous equation associated with \eqref{Theta}, namely,
		\begin{equation}
			\label{Z}
			\partial_t Z=ikU'' \Delta_t^{-1} Z.
		\end{equation}
		We shall prove that
		\[\| Z(t) \|_{\widetilde{H}^r} \lesssim \| Z(s) \|_{\widetilde{H}^r}, \qquad 0 \le s \le t.\]
		
		\noindent\textbf{Step 1. The \(L^2\) estimate.} For the case \(r = 0\), we consider the energy \(\frac{1}{2} \| SZ \|_{L^2}^2\). Its time derivative is
		\[\frac{d}{dt} \frac{1}{2} \| SZ \|_{L^2}^2 = \re \langle S ikU'' \Delta_t^{-1} Z, SZ \rangle + \re \langle \partial_t S Z, SZ \rangle.\]
		As in the proof of Lemma \ref{lem P_Q_TS}, we have
		\[\re \langle \partial_t S Z, SZ \rangle \le (-1 + O(\epsilon)) \| QSZ \|_{L^2}^2.\]
		Repeating the symbolic argument in the proof of Lemma \ref{lem TS_commutator}, with \(TS\) replaced by \(S\), we obtain
		\[|\re \langle SikU'' \Delta_t^{-1}Z,SZ \rangle - \re \langle ikU'' \Delta_t^{-1} SZ,SZ \rangle| \lesssim \epsilon \| QSZ \|_{L^2}^2.\]
		It therefore suffices to estimate \(\re \langle ikU'' \Delta_t^{-1} SZ, SZ \rangle\). By the composition formula \eqref{composition_formula} and the smallness assumption on \(U''\), the symbol of \((Q^{-1})^* ik U'' \Delta_t^{-1} Q^{-1}\) and all its derivatives required by the Calder\'on--Vaillancourt Theorem \ref{thm CV} are bounded by \(O(\epsilon)\). Therefore,
		\[\re \langle ikU''\Delta_t^{-1} SZ, SZ \rangle \lesssim \epsilon \| QSZ \|_{L^2}^2.\]
		Thus
		\[\frac{d}{dt} \frac{1}{2} \| SZ \|_{L^2}^2 \le 0.\]
		Since \(V^{-1}SV\) is a Fourier multiplier whose symbol is uniformly bounded above and below, and \(V\) is unitary, it follows that
		\[\| Z(t) \|_{L^2} \approx \| SZ(t) \|_{L^2} \le \| SZ(s) \|_{L^2} \approx \| Z(s) \|_{L^2}.\]
		This proves the desired estimate for \(r = 0\).
		
		\noindent\textbf{Step 2. Estimates for one adapted derivative.} We now consider the case \(r=1\). For convenience, set
		\[Y = \frac{1}{U'} \frac{\partial_y}{k} Z.\]
		Then \(Y\) satisfies
		\[\partial_t Y = ikU'' \Delta_t^{-1} Y + i \left[ \frac{1}{U'} \partial_y, U'' \Delta_t^{-1} \right] Z.\]
		We next expand the commutator explicitly and determine its order. Since
		\[\left[ \frac{1}{U'} \partial_y, D_t \right] = \left[ \frac{1}{U'} D_t + ikt, D_t \right] = \left[ \frac{1}{U'} D_t, D_t \right] = -\left( \frac{1}{U'} \right)' D_t,\]
		we have
		\[\left[ \frac{1}{U'} \partial_y, \Delta_t \right] = \left[ \frac{1}{U'} \partial_y, -k^2 + D_t^2 \right] = D_t \left[ \frac{1}{U'} \partial_y, D_t \right] + \left[ \frac{1}{U'} \partial_y, D_t \right] D_t = -2 \left( \frac{1}{U'} \right)' D_t^2 - \left( \frac{1}{U'} \right)''D_t.\]
		Consequently,
		\begin{align*}
			\left[ \frac{1}{U'} \partial_y, U'' \Delta_t^{-1} \right] &= U'' \left[ \frac{1}{U'} \partial_y, \Delta_t^{-1} \right] + \left[ \frac{1}{U'} \partial_y, U'' \right] \Delta_t^{-1} \\
			&= - U'' \Delta_t^{-1} \left[ \frac{1}{U'} \partial_y, \Delta_t \right] \Delta_t^{-1} + \frac{U'''}{U'} \Delta_t^{-1} \\
			&= U'' \Delta_t^{-1} \left( 2 \left(\frac{1}{U'}\right)' D_t^2 + \left( \frac{1}{U'} \right)'' D_t \right) \Delta_t^{-1} + \frac{U'''}{U'} \Delta_t^{-1}
		\end{align*}
		By the asymptotic expansion \eqref{composition_expansion}, the principal symbol of \(V^{-1} \left[ \frac{1}{U'} \partial_y, U'' \Delta_t^{-1} \right] V\) is given by
		\[- 2U''(y) \left( \frac{1}{U'(y)} \right)' \frac{\eta^2}{(k^2 + \eta^2)^2} - \frac{U'''(y)}{U'(y)} \frac{1}{k^2 + \eta^2}.\]
		It follows that the commutator has order \(-2\), the same as \(\Delta_t^{-1}\), and its principal symbol carries an additional factor of size \(O(\epsilon)\). Therefore, we obtain
		\begin{equation}
			\label{commutator_bound}
			\left\| (Q^{-1})^* \left[ \frac{1}{U'} \partial_y, U'' \Delta_t^{-1} \right] Q^{-1} \right\|_{L^2 \to L^2} \lesssim \epsilon.
		\end{equation}
		
		We consider the energy \(\frac{1}{2} \| SY \|_{L^2}^2 + \frac{1}{2} \| SZ \|_{L^2}^2\). Using the estimate established in the case \(r = 0\), we obtain
		\begin{align*}
			\frac{d}{dt} \left( \frac{1}{2} \| SY \|_{L^2}^2 + \frac{1}{2} \| SZ \|_{L^2}^2 \right) =& \re \langle S ikU'' \Delta_t^{-1} Y, SY \rangle + \re \langle \partial_t S Y, SY \rangle + \re \left\langle S i \left[ \frac{1}{U'} \partial_y, U'' \Delta_t^{-1} \right] Z, SY \right\rangle \\
			&+ \re \langle S ikU'' \Delta_t^{-1} Z, SZ \rangle + \re \langle \partial_t S Z, SZ \rangle \\
			\le& (-1 + O(\epsilon)) \| QSY \|_{L^2}^2 + (-1 + O(\epsilon)) \| QSZ \|_{L^2}^2 + \re \left\langle S i \left[ \frac{1}{U'} \partial_y, U'' \Delta_t^{-1} \right] Z, SY \right\rangle
		\end{align*}
		The operator norm bound \eqref{commutator_bound} implies
		\begin{align*}
			\re \left\langle S i \left[ \frac{1}{U'} \partial_y, U'' \Delta_t^{-1} \right] Z, SY \right\rangle &\le \left\| (Q^{-1})^* \left[ \frac{1}{U'} \partial_y, U'' \Delta_t^{-1} \right] Q^{-1} \right\|_{L^2 \to L^2} \| QZ \|_{L^2} \| QS^*SY \|_{L^2} \\
			&\lesssim \epsilon \| QZ \|_{L^2} \| QS^*SY \|_{L^2}.
		\end{align*}
		Applying Lemma \ref{lem comparison}, together with the uniform upper and lower bounds for \(\sigma (V^{-1}SV)\), we obtain
		\[\| QZ \|_{L^2} \lesssim \| QSZ \|_{L^2}, \qquad \| QS^*SY \|_{L^2} \lesssim \| QSY \|_{L^2}.\]
		Young's inequality therefore yields
		\[\frac{d}{dt} \left( \frac{1}{2} \| SY \|_{L^2}^2 + \frac{1}{2} \| SZ \|_{L^2}^2 \right) \le 0.\]
		
		Using again the uniform upper and lower bounds for \(\sigma(V^{-1} S V)\) together with \(0 < U'_{\inf} \le U' \le U'_{\sup} < \infty\), we obtain
		\[\| SY \|_{L^2}^2 + \| SZ \|_{L^2}^2 \approx \| Y \|_{L^2}^2 + \| Z \|_{L^2}^2 = \left\| \frac{1}{U'} \frac{\partial_y}{k} Z \right\|_{L^2}^2 + \| Z \|_{L^2}^2 \approx \left\| \frac{\partial_y}{k} Z \right\|_{L^2}^2 + \| Z \|_{L^2}^2 = \| Z \|_{\widetilde{H}^1}^2.\]
		Combining this equivalence with the monotonicity of the energy proves the desired estimate for \(r = 1\).
		
		\noindent\textbf{Step 3. Two adapted derivatives and interpolation.} The case \(r = 2\) follows from the same argument. Indeed, applying \(\frac{1}{U'} \frac{\partial_y}{k}\) once more gives
		\[\partial_t \left( \frac{1}{U'} \frac{\partial_y}{k} \right)^2 Z = ikU''\Delta_t^{-1} \left( \frac{1}{U'} \frac{\partial_y}{k} \right)^2 Z + 2i \left[ \frac{1}{U'} \partial_y, U''\Delta_t^{-1} \right] \left( \frac{1}{U'} \frac{\partial_y}{k} Z \right) + \frac{i}{k} \left[ \frac{1}{U'} \partial_y, \left[ \frac{1}{U'} \partial_y, U''\Delta_t^{-1} \right] \right] Z.\]
		To determine the orders of the commutators, we write \(\frac{1}{U'} \partial_y = \frac{1}{U'} D_t + ikt\). Since \(ikt\) is a scalar operator, it does not contribute to any commutator. Consequently, both commutators can be expressed entirely in terms of \(D_t\) and multiplication operators depending on \(y\). Using the asymptotic expansion for commutators \eqref{commutator_expansion}, we find that both commutators are of order \(-2\), the same as \(\Delta_t^{-1}\), and their symbols carry a factor of size \(O(\epsilon)\). Repeating the preceding energy argument, we obtain
		\[\frac{d}{dt} \frac{1}{2} \sum_{j = 0}^2 \left\| S \left( \frac{1}{U'} \frac{\partial_y}{k} \right)^j Z \right\|_{L^2}^2 \le 0.\]
		Moreover,
		\[\sum_{j = 0}^2 \left\| S \left( \frac{1}{U'} \frac{\partial_y}{k} \right)^j Z \right\|_{L^2}^2 \approx \| Z \|_{\widetilde{H}^2}^2.\]
		Thus,
		\[\| Z(t) \|_{\widetilde{H}^2} \lesssim \| Z(s) \|_{\widetilde{H}^2}.\]
		
		Interpolating between the \(L^2\) and \(\widetilde{H}^2\) estimates yields the desired bound \eqref{Phi_bound} for every \(0 \le r \le 2\).
		
		\noindent\textbf{Step 4. The shifted estimate.} It remains to prove \eqref{Phi_adapted}. The covariance of the homogeneous equation under time shifts gives
		\begin{equation}
			\label{Phi_shift}
			\Phi(t,s) = e^{ikUs} \Phi(t - s, 0) e^{-ikUs}.
		\end{equation}
		Indeed, let
		\[Z(\sigma) = \Phi(\sigma,s) f, \qquad \sigma \ge s.\]
		Then \(Z\) solves
		\[\partial_{\sigma} Z(\sigma) = ikU'' \Delta_{\sigma}^{-1} Z(\sigma), \qquad Z(s) = f.\]
		For \(0 \le \tau \le t-s\), define
		\[\widetilde{Z} (\tau) = e^{-ikUs} Z(s + \tau) = e^{-ikUs} \Phi(s + \tau,s) f.\]
		Recalling that \(\Delta_t^{-1} = e^{ikUt} \Delta_k^{-1} e^{-ikUt}\), we have
		\[\Delta_{s + \tau}^{-1} = e^{ikUs} \Delta_{\tau}^{-1} e^{-ikUs},\]
		and using this identity, we get
		\[\partial_{\tau} \widetilde{Z}(\tau) = e^{-ikUs} \partial_{\tau} Z(s + \tau) = e^{-ikUs} ikU'' \Delta_{s + \tau}^{-1} Z(s + \tau) = ikU'' \Delta_{\tau}^{-1} e^{-ikUs} Z(s + \tau) = ikU'' \Delta_{\tau}^{-1} \widetilde Z(\tau).\]
		Moreover, we have \(\widetilde{Z}(0) = e^{-ikUs}f\), and thus \(\widetilde{Z}\) solves the homogeneous equation \eqref{Z} with initial time \(0\) and initial datum \(e^{-ikUs}f\). By uniqueness,
		\[\widetilde{Z}(\tau) = \Phi(\tau,0) e^{-ikUs}f.\]
		Taking \(\tau = t - s\), we find
		\[e^{-ikUs} \Phi(t,s) f = \Phi(t - s,0) e^{-ikUs} f.\]
		Since this holds for every \(f\), \eqref{Phi_shift} follows.
		
		Using \eqref{Phi_shift} and the bound \eqref{Phi_bound}, we obtain
		\begin{align*}
			\| (-\Delta_s)^r \Phi(t,s) f \|_{L^2} &= \left\| e^{ikUs} |k|^{2r} \op \left( \left\langle \frac{\eta}{k} \right\rangle^{2r} \right) e^{-ikUs} \Phi(t,s) f \right\|_{L^2} = \| |k|^{2r} \Phi(t - s,0) e^{-ikUs} f \|_{\widetilde{H}^{2r}} \\
			&\lesssim \| |k|^{2r} e^{-ikUs} f \|_{\widetilde{H}^{2r}} = \| (-\Delta_s)^r f \|_{L^2}.
		\end{align*}
		This proves \eqref{Phi_adapted} and completes the proof.
	\end{proof}
	
	\begin{remark}
		The homogeneous equation \eqref{Z} associated with \eqref{Theta} coincides, in sheared coordinates, with the equation governing the perturbation vorticity for the incompressible Euler equations linearized around the same shear flow. The inhomogeneous term in \eqref{Theta} arises from the coupling to \(R\) and therefore represents a purely compressible effect. A corresponding Sobolev stability estimate was established in \cite{Zillinger2017}. We include a self-contained proof adapted to the present operator framework; the argument is closely related to that of \cite{Zillinger2017}.
	\end{remark}
	
	\subsection{Mixing estimate}
	
	Recalling the Helmholtz decomposition \eqref{Helmholtz} and the coordinate transformation \eqref{coordinate_transformation}, we obtain
	\[V \mathcal{F}_x v_{\mathrm{sol}} = (- D_t, ik) \Delta_t^{-1} W.\]
	Therefore, to estimate \(v_{\mathrm{sol}}\), we need the decay properties of negative powers of \(-\Delta_t\). We shall use the following two-time version of the mixing estimate.
	
	\begin{lemma}
		\label{lem mixing}
		Let \(0 \le r \le 1\) and \(0 \le s \le t\). Then, for every \(f\) such that \((- \Delta_s)^r f \in L^2\),
		\[\| (-\Delta_t)^{-r} f \|_{L^2} \lesssim |k|^{-4r} \langle t - s \rangle^{-2r} \| (-\Delta_s)^r f \|_{L^2}.\]
	\end{lemma}
	
	\begin{proof}
		Set \(\tau = t - s\) and \(f_s = e^{-ikUs} f\). Since
		\[-\Delta_t = e^{ikUs} (-\Delta_\tau) e^{-ikUs}, \qquad -\Delta_s =e^{ikUs} (-\Delta_k) e^{-ikUs},\]
		unitarity gives
		\[\| (-\Delta_t)^{-r} f\|_{L^2} =\| (-\Delta_\tau)^{-r} f_s \|_{L^2}, \qquad \| (-\Delta_s)^r f\|_{L^2} = \| (-\Delta_k)^r f_s\|_{L^2}.\]
		Thus, it remains to prove the corresponding estimate with initial time zero, namely
		\[\| (- \Delta_\tau)^{-r} f_s \|_{L^2} \lesssim |k|^{-2r} \langle \tau \rangle^{-2r} \| f_s \|_{\widetilde{H}^{2r}}.\]
		For simplicity, we omit the subscript \(s\) from \(f_s\).
		
		We need the following composition estimate.
		\begin{claim}\label{cla pullback}
			The pullback operators associated with \(U\) and \(U^{-1}\) are bounded on \(\widetilde{H}^\sigma\) for \(-2 \le \sigma \le 2\). More precisely, the following two inequalities hold
			\begin{equation}
				\label{Sobolev_composition}
				\| f \circ U \|_{\widetilde{H}^\sigma} \lesssim \| f \|_{\widetilde{H}^\sigma}, \qquad \| f \circ U^{-1} \|_{\widetilde{H}^\sigma} \lesssim \|f\|_{\widetilde{H}^\sigma}, \qquad -2 \le \sigma \le2.
			\end{equation}
		\end{claim}
		
		Recalling the definition of \((- \Delta_{\tau})^{-r}\) and using the unitarity of \(e^{ikU \tau}\) together with \(|k|\ge1\),
		\[\| (- \Delta_{\tau})^{-r} f \|_{L^2} = \| (- \Delta_k)^{-r} e^{-ikU \tau} f \|_{L^2} = |k|^{-2r} \| e^{-ikU \tau} f \|_{\widetilde{H}^{-2r}}.\]
		Set \(z = U(y)\) and \(\widetilde{f}(z) = f(U^{-1}(z))\). Using \eqref{Sobolev_composition}, we obtain
		\[\| e^{-ikU(y) \tau} f(y) \|_{\widetilde{H}_y^{-2r}} \lesssim \| e^{-ikz \tau} \widetilde{f}(z) \|_{\widetilde{H}_z^{-2r}}.\]
		By Plancherel's theorem,
		\begin{align*}
			\| e^{-ikz \tau} \widetilde{f}(z) \|_{\widetilde{H}_z^{-2r}}^2 &= \int_{\mathbb{R}} \left\langle \frac{\zeta}{k} \right\rangle^{-4r} |\widehat{\widetilde{f}}(\zeta + k \tau)|^2 d\zeta = \int_{\mathbb{R}} \left\langle \frac{\zeta}{k} - \tau \right\rangle^{-4r} |\widehat{\widetilde{f}}(\zeta)|^2 d\zeta \\
			&\lesssim \int_{\mathbb{R}} \langle \tau \rangle^{-4r} \left\langle \frac{\zeta}{k} \right\rangle^{4r} |\widehat{\widetilde{f}}(\zeta)|^2 d\zeta = \langle \tau \rangle^{-4r} \| \widetilde{f}(z) \|_{\widetilde{H}_z^{2r}}^2.
		\end{align*}
		Finally, applying \eqref{Sobolev_composition} once more gives
		\[\| \widetilde{f}(z) \|_{\widetilde{H}_z^{2r}} \lesssim \| f(y) \|_{\widetilde{H}_y^{2r}},\]
		which completes the proof.
	\end{proof}
	
	We now prove Claim \ref{cla pullback}.
	
	\begin{proof}[Proof of Claim \ref{cla pullback}]
		Since \(0 < U'_{\inf} \le U' \le U'_{\sup} < \infty\), the map \(U \colon \mathbb{R} \to \mathbb{R}\) is a diffeomorphism. For the case \(\sigma = 0\), the change of variables \(z = U(y)\) gives
		\[\| f \circ U \|_{L^2}^2 = \int_{\mathbb{R}} |f(z)|^2 \frac{1}{U'(U^{-1}(z))} dz \le \frac{1}{U'_{\inf}} \| f \|_{L^2}^2.\]
		For \(\sigma = 2\), the identities
		\[\partial_y (f \circ U) = (\partial_y f \circ U) U', \qquad \partial_y^2 (f \circ U) = (\partial_y^2 f \circ U) (U')^2 + (\partial_y f \circ U) U''\]
		yield
		\[\| f \circ U \|_{\widetilde{H}^2} \lesssim \| f \|_{\widetilde{H}^2},\]
		where we used the smallness assumption on \(U''\) together with the \(\sigma = 0\) estimate. The same argument applies to \(U^{-1}\). Interpolation then proves \eqref{Sobolev_composition} for \(0 \le \sigma \le 2\).
		
		For negative indices, let \(\mathcal{C}_U f = f \circ U\). Its \(L^2\)-adjoint is
		\[\mathcal{C}_U^*g = \left( \frac{g}{U'} \right) \circ U^{-1}.\]
		The positive-order estimates above, together with the boundedness of multiplication by \(\frac{1}{U'}\) on \(\widetilde{H}^\sigma\), imply that \(\mathcal C_U^*\) is bounded on \(\widetilde{H}^\sigma\) for \(0 \le \sigma \le 2\). Therefore, by duality, \(\mathcal C_U\) is bounded on \(\widetilde{H}^{-\sigma}\). The case of \(\mathcal C_{U^{-1}}\) is identical, proving the claim.
	\end{proof}
	
	\begin{remark}
		The mixing mechanism underlying Lemma \ref{lem mixing} is classical. For the Couette flow, the estimate follows directly from the Fourier representation. More generally, analogous negative-Sobolev mixing estimates hold for sufficiently smooth strictly monotone shear profiles whose shear rates are bounded above and below, without requiring the higher derivatives of \(U\) to be small \cite{Zillinger2017}. The two-time formulation used here follows by recentering the initial time at \(s\) and straightening the shear through the change of variables \(z = U(y)\). The slow-variation assumption on the shear rate is used in the present proof to make the composition bounds uniform, with constants depending only on \(U'_{\inf}\) and \(U'_{\sup}\). Without this assumption, the same argument applies, but the constants additionally depend on suitable higher-order derivative bounds for \(U\) and \(U^{-1}\).
	\end{remark}
	
	\subsection{Completion of the proof}
	
	We are now ready to complete the proof of the main theorem (Theorem \ref{thm main}).
	
	\subsubsection{Growth bounds for the acoustic variables}
	This part is devoted to the proof of the growth bounds of acoustic variables in \eqref{thm:acoustic growth}.
	
	\noindent\textbf{Step 1. Endpoint comparison of the weighted energies.} We first use the coordinate transformation \eqref{coordinate_transformation} to express the conclusions of Propositions \ref{prop higher} and \ref{prop lower} in terms of the original variables. For convenience, set \(\theta_k = V^{-1} \Theta\). Then
	\begin{align*}
		\widetilde{E}_{1, s} (t) \approx& \frac{1}{M^2} \left\| \op \left( \tau(\eta) \sqrt{k^2 + \eta^2} g \left( \frac{\eta}{k} + 2 U'_{\sup} (t - s) \right) \right) \rho_k (t) \right\|_{L^2}^2 + \left\| \op \left( \tau(\eta) g \left( \frac{\eta}{k} + 2 U'_{\sup} (t - s) \right) \right) \alpha_k (t) \right\|_{L^2}^2 \\
		&+ \left\| \op \left( \tau(\eta) g \left( \frac{\eta}{k} + 2 U'_{\sup} (t - s) \right) \right) \theta_k (t) \right\|_{L^2}^2 + \left\| \op \left( \tau(\eta) \left\langle \frac{\eta}{k} \right\rangle^{-1} g \left( \frac{\eta}{k} + 2 U'_{\sup} (t - s) \right) \right) (\theta_k - U' \rho_k) \right\|_{L^2}^2,
	\end{align*}
	and
	\begin{align*}
		\widetilde{E}_{-2, s} (t) \approx& \frac{1}{M^2} \left\| \op \left( \tau(\eta) \sqrt{k^2 + \eta^2} h \left( \frac{\eta}{k} + \frac{1}{2} U'_{\inf} (t - s) \right) \right) \rho_k (t) \right\|_{L^2}^2 + \left\| \op \left( \tau(\eta) h \left( \frac{\eta}{k} + \frac{1}{2} U'_{\inf} (t - s) \right) \right) \alpha_k (t) \right\|_{L^2}^2 \\
		&+ \left\| \op \left( \tau(\eta) h \left( \frac{\eta}{k} + \frac{1}{2} U'_{\inf} (t - s) \right) \right) \theta_k (t) \right\|_{L^2}^2 + \left\| \op \left( \tau(\eta) \left\langle \frac{\eta}{k} \right\rangle^{-1} h \left( \frac{\eta}{k} + \frac{1}{2} U'_{\inf} (t - s) \right) \right) (\theta_k - U' \rho_k) \right\|_{L^2}^2.
	\end{align*}
	Here \(\tau\) denotes the symbol of \(V^{-1}TSV\) and satisfies \eqref{tau}. Moreover,
	\begin{equation}
		\label{monotonicity}
		\widetilde{E}_{1, s} (s) \le \widetilde{E}_{1, s} (0), \qquad \widetilde{E}_{-2, s} (s) \le \widetilde{E}_{-2, s} (0).
	\end{equation}
	
	By the definitions of \(g\) and \(\tau\),
	\[\tau(\eta) g \left( \frac{\eta}{k} \right) \approx
	\begin{cases}
		\left\langle \frac{\eta}{k} \right\rangle^{1 - 1/\gamma}, &\qquad \frac{\eta}{k} \le 0, \\
		\left\langle \frac{\eta}{k} \right\rangle^{1/\gamma}, &\qquad \frac{\eta}{k} \ge 0.
	\end{cases}
	\]
	Thus
	\begin{equation}
		\label{g_t=s}
		\tau(\eta) g \left( \frac{\eta}{k} + 2 U'_{\sup} (t - s) \right)\bigg|_{t = s} \gtrsim \left\langle \frac{\eta}{k} \right\rangle^{-1/2},
	\end{equation}
	since \(\gamma = \frac{2}{3}\). At \(t = 0\), using \(g(\xi) \lesssim \langle \xi \rangle\), we obtain
	\begin{equation}
		\label{g_t=0}
		\tau(\eta) g \left( \frac{\eta}{k} + 2 U'_{\sup} (t - s) \right)\bigg|_{t = 0} \lesssim \left\langle \frac{\eta}{k} \right\rangle^{1/\gamma} \left\langle \frac{\eta}{k} - 2 U'_{\sup} s \right\rangle \lesssim \left\langle \frac{\eta}{k} \right\rangle^{1 + 1/\gamma} \langle s \rangle = \left\langle \frac{\eta}{k} \right\rangle^{5/2} \langle s \rangle.
	\end{equation}
	
	Similarly,
	\[\tau(\eta) h \left( \frac{\eta}{k} \right) \approx
	\begin{cases}
		\left\langle \frac{\eta}{k} \right\rangle^{-2 - 1/\gamma}, &\qquad \frac{\eta}{k} \le 0, \\
		\left\langle \frac{\eta}{k} \right\rangle^{1/\gamma}, &\qquad \frac{\eta}{k} \ge 0.
	\end{cases}
	\]
	It follows that
	\[\tau(\eta) h \left( \frac{\eta}{k} + \frac{1}{2} U'_{\inf} (t - s) \right)\bigg|_{t = s} \gtrsim \left\langle \frac{\eta}{k} \right\rangle^{-7/2}.\]
	On the other hand,
	\[h \left( \frac{\eta}{k} - \frac{1}{2} U'_{\inf} s \right) \lesssim \left\langle \frac{\eta}{k} \right\rangle^2 \langle s \rangle^{-2}.\]
	Indeed, if \(\frac{\eta}{k} - \frac{1}{2} U'_{\inf} s \le 0\), then
	\[h \left( \frac{\eta}{k} - \frac{1}{2} U'_{\inf} s \right) \approx \left\langle \frac{\eta}{k} - \frac{1}{2} U'_{\inf} s \right\rangle^{-2} \lesssim \left\langle \frac{\eta}{k} \right\rangle^2 \langle s \rangle^{-2}.\]
	If \(\frac{\eta}{k} - \frac{1}{2} U'_{\inf} s \ge 0\), then \(\frac{\eta}{k} \ge \frac{1}{2} U'_{\inf} s \ge 0\). Since \(h(\xi) \lesssim 1\), we have
	\[h \left( \frac{\eta}{k} - \frac{1}{2} U'_{\inf} s \right) \lesssim 1 \lesssim \left\langle \frac{\eta}{k} \right\rangle^2 \langle s \rangle^{-2}.\]
	Combining the two cases, we obtain
	\[\tau(\eta) h \left( \frac{\eta}{k} + \frac{1}{2} U'_{\inf} (t - s) \right)\bigg|_{t = 0} \lesssim \left\langle \frac{\eta}{k} \right\rangle^{1/\gamma} \left\langle \frac{\eta}{k} \right\rangle^2 \langle s \rangle^{-2} = \left\langle \frac{\eta}{k} \right\rangle^{7/2} \langle s \rangle^{-2}.\]
	
	\noindent\textbf{Step 2. Bounds for the initial weighted energies.} We next bound the initial weighted energies. For clarity, we restore the dependence of the weighted energies on \(k\) and denote them by \(\widetilde E_{1,s,k}\) and \(\widetilde E_{-2,s,k}\). All the norms entering these energies are \(L_y^2\)-norms for a fixed \(k \ne 0\).
	
	The multiplier bound \eqref{g_t=0}, together with \(|k| \ge 1\), immediately yields
	\[\widetilde{E}_{1,s,k} (0) \lesssim \left( \frac{1}{M^2} |k|^2 \| \rho_{in,k} \|_{H_y^5}^2 + \| \alpha_{in,k}\|_{H_y^4}^2 + \| \theta_{in,k} \|_{H_y^4}^2 + \| \theta_{in,k} - U' \rho_{in,k} \|_{H_y^4}^2\right) \langle s \rangle^2.\]
	Here
	\[\theta_{in,k} = \omega_{in,k} + U' \rho_{in,k} + U'' \partial_y \Delta_k^{-1} \rho_{in,k}.\]
	Since \(U'\) depends only on \(y\), multiplication by \(U'\) preserves the \(x\)-regularity. Moreover, multiplication by \(U'\) is bounded on \(H_y^4\), while the remaining term is controlled by the smallness assumption on \(U''\). Therefore,
	\[\| \theta_{in,k} \|_{H_y^4} \lesssim \| \omega_{in,k} \|_{H_y^4} + \frac{1}{M} \| \rho_{in,k} \|_{H_y^4}.\]
	Here and below, we use the allowed dependence of the implicit constants in \(\lesssim\) on \(U'_{\inf}\), \(U'_{\sup}\), and \(M\) to make a factor \(M^{-1}\) explicit whenever convenient for matching the normalization in \eqref{N_in}. The corresponding factors depending only on \(U'_{\inf}\), \(U'_{\sup}\), and \(M\) are absorbed into the implicit constant. Similarly,
	\[\| \theta_{in,k} - U' \rho_{in,k} \|_{H_y^4} \lesssim \| \omega_{in,k} \|_{H_y^4} + \frac{1}{M} \| \rho_{in,k} \|_{H_y^4}.\]
	
	Recalling the definition of \(\mathcal{N}_{in}\) in \eqref{N_in}, for simplicity, we denote by \(\mathcal{N}_{in,k}\) the \(k\)-th Fourier mode contribution to \(\mathcal{N}_{in}\), namely,
	\[\mathcal{N}_{in,k}^2 = \frac{1}{M^2} \| \rho_{in,k} \|_{H_y^5}^2 + \langle k \rangle^{-2} \| \alpha_{in,k} \|_{H_y^4}^2 + \langle k \rangle^{-2} \| \omega_{in,k} \|_{H_y^4}^2,\]
	so that
	\[\mathcal{N}_{in}^2 = \sum_{k \in \mathbb{Z} \setminus \{0\}} \mathcal{N}_{in,k}^2.\]
	Thus
	\begin{equation}
		\label{E_1,in}
		\widetilde{E}_{1,s,k} (0) \lesssim |k|^2 \mathcal{N}_{in,k}^2 \langle s \rangle^2.
	\end{equation}
	For the lower-order energy, the estimate of the \(h\)-multiplier at \(t = 0\) shows that
	\[\widetilde{E}_{-2,s,k} (0) \lesssim |k|^2 \mathcal{N}_{in,k}^2 \langle s \rangle^{-4}.\]
	
	\noindent\textbf{Step 3. Growth of the irrotational velocity.} The terminal-time lower bound \eqref{g_t=s} and the monotonicity of the \(g\)-weighted energy \eqref{monotonicity} imply
	\[\left\| \op \left( \left\langle \frac{\eta}{k} \right\rangle^{-1/2} \right) \alpha_k \right\|_{L_y^2}^2 (s) \lesssim \widetilde{E}_{1,s,k} (s) \le \widetilde{E}_{1,s,k} (0) \lesssim |k|^2 \mathcal{N}_{in,k}^2 \langle s \rangle^2,\]
	where we also used \eqref{E_1,in}. Similarly,
	\[\left\| \op \left( \left\langle \frac{\eta}{k} \right\rangle^{-7/2} \right) \alpha_k \right\|_{L_y^2}^2 (s) \lesssim |k|^2 \mathcal{N}_{in,k}^2 \langle s \rangle^{-4}.\]
	Interpolating between these two estimates, we obtain
	\[\left\| \op \left( \left\langle \frac{\eta}{k} \right\rangle^{-1} \right) \alpha_k \right\|_{L_y^2} (s) \lesssim |k|
	\mathcal{N}_{in,k} \langle s \rangle^{1/2}.\]
	
	By the Helmholtz decomposition \eqref{Helmholtz} and Plancherel's theorem in the \(x\)-variable,
	\[ \| v_{\mathrm{irr}} \|_{L_x^2 L_y^2}^2 (s) = \sum_{k \in \mathbb{Z} \setminus \{0\}} \| (-\Delta_k)^{-1/2} \alpha_k \|_{L_y^2}^2 (s).\]
	Since
	\[(-\Delta_k)^{-1/2} = |k|^{-1} \op \left( \left\langle \frac{\eta}{k} \right\rangle^{-1} \right),\]
	it follows that
	\[\| v_{\mathrm{irr}} \|_{L_x^2 L_y^2} (s) \lesssim \mathcal{N}_{in} \langle s\rangle^{1/2}.\]
	
	\noindent\textbf{Step 4. Growth bound for the density.} The growth estimate for \(\rho\) follows from the same interpolation argument. More generally, we record the resulting weighted estimate for later use. Interpolating between the two endpoint estimates provided by \(\widetilde E_{1,s,k}\) and \(\widetilde E_{-2,s,k}\), we obtain
	\begin{equation}
		\label{rho_weighted_bound}
		\frac{1}{M} \left\| \op \left( \left\langle \frac{\eta}{k} \right\rangle^r \right) \rho_k \right\|_{L_y^2} (s) = \frac{1}{M} |k|^{-r} \| (- \Delta_s)^{r/2} R(s) \|_{L_y^2} \lesssim \mathcal N_{in,k} \langle s \rangle^{r+1/2}
	\end{equation}
	for \(-\frac{5}{2} \le r \le \frac{1}{2}\). Taking \(r = 0\) in \eqref{rho_weighted_bound} and summing over \(k \ne 0\), we have
	\[\frac{1}{M} \| \rho \|_{L_x^2 L_y^2} (s) = \frac{1}{M} \left( \sum_{k \in \mathbb{Z} \setminus \{0\}} \| \rho_k \|_{L_y^2}^2 (s) \right)^{1/2} \lesssim \mathcal{N}_{in} \langle s \rangle^{1/2}.\]
	This proves \eqref{thm:acoustic growth}.
	
	\subsubsection{Inviscid damping of the solenoidal velocity}
	This part aims to prove the inviscid damping estimates \eqref{thm:horizontal damping}--\eqref{thm:vertical damping}, namely the decay estimates for \(v_{\mathrm{sol}}\). By the Helmholtz decomposition \eqref{Helmholtz} and the Fourier transform in \(x\),
	\[v_{\mathrm{sol},k} = (- \partial_y, ik) \Delta_k^{-1} \omega_k = V^{-1} (- D_t, ik) \Delta_t^{-1} W.\]
	We shall derive a suitable representation formula for \(W\). Recalling \eqref{Theta}, Duhamel's formula gives
	\[\Theta(t)=\Phi(t,0)\Theta(0)+\int_0^t\Phi(t,s)2ik^3U''\Delta_s^{-1}U'\Delta_s^{-1}R(s)\,ds.\]
	Since \(\Theta = W + L(t) R\), we obtain
	\begin{equation}
		\label{Duhamel}
		W(t) = \Phi(t,0) W(0) - L(t) R(t) + \Phi(t,0) L(0) R(0) + \int_{0}^{t} \Phi(t,s) 2ik^3 U'' \Delta_s^{-1} U' \Delta_s^{-1} R(s) ds.
	\end{equation}
	
	\paragraph{Estimates for the horizontal component.} We first consider \((v_{\mathrm{sol},k})_1\). Since
	\[\| V^{-1} D_t \Delta_t^{-1} W\|_{L_y^2} \le \| (-\Delta_t)^{-1/2} W \|_{L_y^2},\]
	it is enough to apply \((-\Delta_t)^{-1/2}\) to \eqref{Duhamel} and estimate the resulting terms.
	
	We now estimate the three terms outside the integral. Applying Lemma \ref{lem mixing} with \(r = \frac{1}{2}, s = 0\), and using Proposition \ref{prop Phi}, we obtain
	\[\| (- \Delta_t)^{-1/2} \Phi(t,0) W(0) \|_{L_y^2} \lesssim |k|^{-1} \langle t \rangle^{-1} \| \Phi(t,0) W(0) \|_{\widetilde{H}_y^1} \lesssim |k|^{-1} \| W(0) \|_{\widetilde{H}_y^1} \langle t \rangle^{-1} \lesssim \mathcal{N}_{in,k} \langle t \rangle^{-1},\]
	where we used \(W(0) = \omega_{in,k}\). The same argument gives
	\[\| (- \Delta_t)^{-1/2} \Phi(t,0) L(0) R(0) \|_{L_y^2} \lesssim |k|^{-1} \langle t \rangle^{-1} \| \Phi(t,0) L(0) R(0) \|_{\widetilde{H}_y^1} \lesssim |k|^{-1} \| L(0) R(0) \|_{\widetilde{H}_y^1} \langle t \rangle^{-1}.\]
	Recall that
	\[L(t) = U'+ U'' D_t \Delta_t^{-1}\]
	is an operator of order zero with principal symbol \(U'\), while its remainder is of size \(O(\epsilon)\). Therefore,
	\[\left\| \op \left( \left\langle \frac{\eta}{k} \right\rangle \right) L(0) \op \left( \left\langle \frac{\eta}{k} \right\rangle^{-1} \right) \right\|_{L_y^2 \to L_y^2} \le (1 + O(\epsilon)) U'_{\sup} \lesssim \frac{1}{M},\]
	where the compensating factor \(MU'_{\sup}\) is absorbed into the implicit constant according to the convention stated above. Thus
	\[\| (- \Delta_t)^{-1/2} \Phi(t,0) L(0) R(0) \|_{L_y^2} \lesssim \frac{1}{M} |k|^{-1} \| R(0) \|_{\widetilde{H}_y^1} \langle t \rangle^{-1} = \frac{1}{M} |k|^{-1} \| \rho_{in,k} \|_{\widetilde{H}_y^1} \langle t \rangle^{-1} \lesssim \mathcal{N}_{in,k} \langle t \rangle^{-1}.\]
	Applying the composition formula \eqref{composition_formula} and the Calder\'on--Vaillancourt Theorem \ref{thm CV}, we have
	\[\| (-\Delta_t)^{-1/2} L(t) R(t) \|_{L_y^2} \lesssim \frac{1}{M} \| (-\Delta_t)^{-1/2} R(t) \|_{L_y^2}.\]
	\eqref{rho_weighted_bound} with \(r = - 1\) yields
	\[\| (-\Delta_t)^{-1/2} L(t) R(t) \|_{L_y^2} \lesssim \mathcal{N}_{in,k} \langle t \rangle^{-1/2}.\]
	
	We now estimate the integral term. For convenience, recall the notation
	\begin{equation}
		\label{K}
		K(s) = 2ik^3U'' \Delta_s^{-1} U' \Delta_s^{-1}.
	\end{equation}
	Applying Minkowski's inequality and Lemma \ref{lem mixing} with \(r = \frac{1}{2}\), we have
	\begin{align*}
		\left\| (- \Delta_t)^{-1/2} \int_{0}^{t} \Phi(t,s) K(s) R(s) ds \right\|_{L_y^2} &\lesssim \int_{0}^{t} \left\| (- \Delta_t)^{-1/2} \Phi(t,s) K(s) R(s) \right\|_{L_y^2} ds \\
		&\lesssim \int_{0}^{t} |k|^{-2} \langle t - s \rangle^{-1} \left\| (- \Delta_s)^{1/2} \Phi(t,s) K(s) R(s) \right\|_{L_y^2} ds.
	\end{align*}
	For the integrand, the adapted propagator estimate \eqref{Phi_adapted} in Proposition \ref{prop Phi} yields
	\[\left\| (- \Delta_s)^{1/2} \Phi(t,s) K(s) R(s) \right\|_{L_y^2} \lesssim \left\| (- \Delta_s)^{1/2} K(s) R(s) \right\|_{L_y^2}.\]
	By \eqref{K}, the operator \(K(s)\) is of order \(-4\), with principal coefficient \(2ik^3U'U''\). Therefore,
	\[\| (-\Delta_s)^{1/2} K(s) R(s) \|_{L_y^2} \lesssim \epsilon U'_{\sup} \| k^3 (-\Delta_s)^{-3/2} R(s) \|_{L_y^2} \lesssim \frac{1}{M} \| k^3 (-\Delta_s)^{-3/2} R(s)\|_{L_y^2}.\]
	Here, in the last estimate, we wrote \(\epsilon U'_{\sup} = (\epsilon MU'_{\sup}) M^{-1}\) and absorbed the factor \(\epsilon MU'_{\sup}\) into the implicit constant, which is allowed to depend on \(U'_{\inf}\), \(U'_{\sup}\) and \(M\). The weighted estimate \eqref{rho_weighted_bound} with \(r = - 2\) gives
	\[\frac{1}{M} \| k^3 (- \Delta_s)^{-3/2} R(s) \|_{L_y^2} \lesssim \frac{1}{M} \| k^2 (- \Delta_s)^{-1} R(s) \|_{L_y^2} \lesssim \mathcal N_{in,k} \langle s \rangle^{-3/2}.\]
	Thus
	\[\left\| (- \Delta_t)^{-1/2} \int_{0}^{t} \Phi(t,s) K(s) R(s) ds \right\|_{L_y^2} \lesssim \int_{0}^{t} |k|^{-2} \langle t - s \rangle^{-1} \mathcal N_{in,k} \langle s \rangle^{-3/2} ds \lesssim \mathcal N_{in,k} \langle t \rangle^{-1}.\]
	
	Combining this estimate with the bounds for the remaining terms in \eqref{Duhamel}, and then applying Plancherel's identity in \(x\), we conclude that
	\[\| (v_{\mathrm{sol}})_1 \|_{L_x^2 L_y^2} (t) \lesssim \mathcal{N}_{in} \langle t \rangle^{-1/2}.\]
	
	\paragraph{Estimates for the vertical component.} We finally consider the second component of \(v_{\mathrm{sol}}\). Recall that
	\[\| (v_{\mathrm{sol},k})_2 \|_{L_y^2} = \| V^{-1}ik \Delta_t^{-1} W \|_{L_y^2} = \| k \Delta_t^{-1} W\|_{L_y^2}.\]
	The first three terms in \eqref{Duhamel} are treated in the same way as for the first component. Indeed, Lemma \ref{lem mixing} with \(r = 1, s = 0\), together with Proposition \ref{prop Phi}, gives
	\[\| k \Delta_t^{-1} \Phi(t,0) W(0) \|_{L_y^2}
	+ \| k \Delta_t^{-1} \Phi(t,0) L(0) R(0) \|_{L_y^2} \lesssim \mathcal{N}_{in,k} \langle t \rangle^{-2}.\]
	Moreover, the weighted estimate \eqref{rho_weighted_bound} with \(r = -2\) yields
	\[\| k \Delta_t^{-1} L(t) R(t)\|_{L_y^2} \lesssim \mathcal{N}_{in,k} \langle t \rangle^{-3/2}.\]
	
	It remains to estimate the integral term. By Minkowski's inequality and Lemma \ref{lem mixing} with \(r = 1\), it follows that
	\[\left\| k \Delta_t^{-1} \int_{0}^{t} \Phi(t,s) K(s) R(s) ds \right\|_{L_y^2} \lesssim \int_{0}^{t} |k|^{-3} \langle t - s \rangle^{-2} \|(-\Delta_s) \Phi(t,s) K(s) R(s) \|_{L_y^2} ds.\]
	Proposition \ref{prop Phi} gives
	\[\| (-\Delta_s) \Phi(t,s) K(s) R(s) \|_{L_y^2} \lesssim \| (-\Delta_s) K(s) R(s) \|_{L_y^2}.\]
	By \eqref{K} and the same symbolic calculations as above,
	\[\| (-\Delta_s) K(s) R(s) \|_{L_y^2} \lesssim \frac{1}{M} \| k^3 (-\Delta_s)^{-1} R(s) \|_{L_y^2} \lesssim |k| \mathcal{N}_{in,k} \langle s \rangle^{-3/2},\]
	where the last inequality follows from \eqref{rho_weighted_bound} with \(r = -2\). Consequently,
	\[\left\| k \Delta_t^{-1} \int_{0}^{t} \Phi(t,s) K(s) R(s) ds \right\|_{L_y^2} \lesssim \int_{0}^{t} |k|^{-3} \langle t - s \rangle^{-2} |k| \mathcal{N}_{in,k} \langle s \rangle^{-3/2} ds \lesssim \mathcal{N}_{in,k} \langle t \rangle^{-3/2}.\]
	
	Combining these estimates and applying Plancherel's identity in \(x\), we conclude that
	\[\| (v_{\mathrm{sol}})_2 \|_{L_x^2 L_y^2} (t) \lesssim \mathcal{N}_{in} \langle t \rangle^{-3/2}.\]
	This completes the proof of Theorem \ref{thm main}.

	\printbibliography
	
\end{document}